%% file: HK_20151123.tex
\documentclass[12pt,a4paper]{article}
\usepackage[cp1251]{inputenc}
\usepackage{amsmath,amsthm}
\usepackage{amsfonts,amstext,amssymb,verbatim,epsfig}
\usepackage{dsfont}
\usepackage{enumerate}
\usepackage{psfrag}
\usepackage{graphicx}
\usepackage{color}
\usepackage{float}
\usepackage{subfig}

\usepackage[top=3cm, bottom=3.5cm, left=2.5cm, right=2.5cm]{geometry}

\def\R{{\mathbb{R}}}

\def\Z{{\mathbb{Z}}}

\def \p{{\bf P1}} 
\def \pp{{\bf P2}} 
\def \ppp{{\bf P3}} 
\def \s{{\bf S1}} 
\def \sss{{\bf S2}}

\def \GG{{\mathbb G}} 
\def \ballZ{{\mathrm B}} 
\def \atom{\omega} 
\def \set{{\mathcal S}} 

\def \ballS{{\ballZ_{\scriptscriptstyle \set}}}
\def \constP{{\chi_{\scriptscriptstyle {\mathrm P}}}}
\def \constS{{\Delta_{\scriptscriptstyle {\mathrm S}}}}
\def \funcP{f_{\scriptscriptstyle {\mathrm P}}}
\def \funcS{f_{\scriptscriptstyle {\mathrm S}}}
\def \RP{R_{\scriptscriptstyle {\mathrm P}}}
\def \RS{R_{\scriptscriptstyle {\mathrm S}}}
\def \LP{L_{\scriptscriptstyle {\mathrm P}}}
\def \epsP{{\varepsilon_{\scriptscriptstyle {\mathrm P}}}}

\def \scexp{{\theta_{\scriptscriptstyle {\mathrm sc}}}}
\def \setg{{\mathcal A}}
\def \setd{{A'}}
\def \cmax{{\mathcal C}}

\def \emax{{\mathcal E}}
\def \cemax{{\widetilde {\mathcal C}}}
\def\seedin{{\overline {\mathrm I}}}
\def\seedde{{\overline {\mathrm D}}}

\def\seed{{\overline {\mathrm E}}}
\def\seedcascade{{\overline {\mathrm G}}}
\def\pl{{\mathcal Q}}
\def\dist{{\mathrm d}}
\def\vgb{{\theta_{\scriptscriptstyle {\mathrm{vgb}}}}}
\def\Rvgb{{R_{\scriptscriptstyle {\mathrm{vgb}}}}}
\def \ballG{{\ballZ_{\scriptscriptstyle G}}}
\def\distG{{\dist_{\scriptscriptstyle G}}}
\def\Rchd{{R_{\scriptscriptstyle {\mathrm{chd}}}}}
\def\chd{{\theta_{\scriptscriptstyle {\mathrm{chd}}}}}
\def\den{{\eta}} 

\newtheorem{theorem}{Theorem}[section]
\newtheorem{corollary}[theorem]{Corollary}
\newtheorem{lemma}[theorem]{Lemma}
\newtheorem{proposition}[theorem]{Proposition}

\newtheoremstyle{likedef}
  {}%
  {}%
  {}%
  {}
  {\bfseries}%
  {.}%
  {.5em}%
  {}%

\theoremstyle{likedef}

\newtheorem{definition}[theorem]{Definition}
\newtheorem{remark}[theorem]{Remark}
\newtheorem{claim}[theorem]{Claim}

\numberwithin{equation}{section}

\begin{document}
\title{Random walks on infinite percolation clusters in models with long-range correlations}

\author{
Artem Sapozhnikov\thanks{
University of Leipzig, Department of Mathematics, Room A409, 
Augustusplatz 10, 04109 Leipzig, Germany.
email: artem.sapozhnikov@math.uni-leipzig.de}
}

\maketitle

\footnotetext{MSC2000: Primary 60K37, 58J35.}
\footnotetext{Keywords: Percolation, random walk, heat kernel, Harnack inequality, harmonic function, local limit theorem, Poincar\'e inequality, isoperimetric inequality, 
long-range correlations, random interlacements, Gaussian free field.}

\begin{abstract}
For a general class of percolation models with long-range correlations on $\Z^d$, $d\geq 2$, introduced in \cite{DRS12}, 
we establish regularity conditions of Barlow \cite{Barlow} that mesoscopic subballs of all large enough balls in the unique infinite percolation cluster 
have regular volume growth and satisfy a weak Poincar\'e inequality. 
As immediate corollaries, we deduce quenched heat kernel bounds, parabolic Harnack inequality, and finiteness of the dimension of harmonic functions with at most polynomial growth. 
Heat kernel bounds and the quenched invariance principle of \cite{PRS} allow to extend various other known results about Bernoulli percolation by mimicking their proofs, 
for instance, the local central limit theorem of \cite{BH09} or the result of \cite{BDCKY14} that the dimension of at most linear harmonic functions on the infinite cluster is $d+1$. 

In terms of specific models, all these results are new for random interlacements at every level in any dimension $d\geq 3$, 
as well as for the vacant set of random interlacements \cite{SznitmanAM,SidoraviciusSznitman_RI} and the level sets of the Gaussian free field \cite{RodSz} 
in the regime of the so-called local uniqueness (which is believed to coincide with the whole supercritical regime for these models).
\end{abstract}

\section{Introduction}

Delmotte \cite{Delmotte99} proved that the transition density of the simple random walk on a graph satisfies 
Gaussian bounds and the parabolic Harnack inequality holds if all the balls have regular volume growth and satisfy a Poincar\'e inequality. 
Barlow \cite{Barlow} relaxed these conditions by imposing them only on all {\it large enough} balls, and showed that they 
imply large time Gaussian bounds and the elliptic Harnack inequality for large enough balls. 
Later, Barlow and Hambly \cite{BH09} proved that the parabolic Harnack inequality also follows from Barlow's conditions. 
Barlow \cite{Barlow} verified these conditions for the supercritical cluster of Bernoulli percolation on $\Z^d$,
which lead to the almost sure Gaussian heat kernel bounds and parabolic Harnack inequality. 
By using stationarity and heat kernel bounds, the quenched invariance principle was proved in \cite{SS04,BergerBiskup,MathieuPiatnitski}, 
which lead to many further results about supercritical Bernoulli percolation, 
including the local central limit theorem \cite{BH09} and the fact that the dimension of harmonic functions of at most linear growth is $d+1$ \cite{BDCKY14}. 

The independence property of Bernoulli percolation was essential in verifying Barlow's conditions, 
and up to now it has been the only example of percolation model for which the conditions were verified. 
On the other hand, once the conditions are verified, the derivation of all the further results 
uses rather robust methods and allows for extension to other stationary percolation models. 

The aim of this paper is to develop an approach to verifying Barlow's conditions for infinite clusters of percolation models, which
on the one hand, applies to supercritical Bernoulli percolation, but on the other, does not rely on independence 
and extends beyond models which are in any stochastic relation with Bernoulli percolation. 
Motivating examples for us are random interlacements, vacant set of random interlacements, and the level sets of the Gaussian free field \cite{SznitmanAM,SidoraviciusSznitman_RI,RodSz}. 
In all these models, the spatial correlations decay only polynomially with distance, 
and classical Peierls-type arguments do not apply. 
A unified framework to study percolation models with strong correlations was proposed in \cite{DRS12}, 
within which the shape theorem for balls \cite{DRS12} and the quenched invariance principle \cite{PRS} were proved. 
In this paper we prove that Barlow's conditions are satisfied by infinite percolation clusters in the general setting of \cite{DRS12}. 
In particular, all the above mentioned properties of supercritical Bernoulli percolation extend to all the models satisfying assumptions from \cite{DRS12}, 
which include supercritical Bernoulli percolation, 
random interlacements at every level in any dimension $d\geq 3$, the vacant set of random interlacements and the level sets of the Gaussian free field in the regime of local uniqueness.

\bigskip

\subsection{General graphs}

Let $G$ be an infinite connected graph with the vertex set $V(G)$ and the edge set $E(G)$. 
For $x,y\in V(G)$, define the weights
\[
\nu_{xy} = \left\{\begin{array}{ll} 1,&\{x,y\}\in E(G),\\ 0,&\text{otherwise,}\end{array}\right.
\qquad \mu_x = \sum_y\nu_{xy},
\]
and extend $\nu$ to the measure on $E(G)$ and $\mu$ to the measure on $V(G)$. 

For functions $f:V(G)\to \R$ and $g : E(G)\to\R$, let
$\int f d\mu = \sum_{x\in V(G)}f(x)\mu_x$ and $\int g d\nu = \sum_{e\in E(G)}g(e)\nu_e$, and define 
$|\nabla f|:E(G)\to\R$ by $|\nabla f|(\{x,y\}) = |f(x) - f(y)|$ for $\{x,y\}\in E(G)$. 

Let $\distG$ be the graph distance on $G$, and define $\ballG(x,r) = \{y\in V(G):\distG(x,y)\leq r\}$. 
We assume that $\mu(\ballG(x,r))\leq C_0 r^d$ for all $x\in V(G)$ and $r\geq 1$. 
In particular, this implies that the maximal degree in $G$ is bounded by $C_0$. 

We say that a graph $G$ satisfies the {\it volume regularity} and the {\it Poincar\'e inequality}
if for all $x\in V(G)$ and $r>0$, $\mu(\ballG(x,2r))\leq C_1\cdot \mu(\ballG(x,r))$ and, respectively, 
$\min_a\int_{\ballG(x,r)}(f-a)^2 d\mu \leq C_2\cdot r^2\cdot \int_{E(\ballG(x,r))} |\nabla f|^2 d\nu$,  
with some constants $C_1$ and $C_2$. Graphs satisfying these conditions are very well understood. 
Delmotte proved in \cite{Delmotte99} the equivalence of such conditions to Gaussian bounds on the transition density of the simple random walk 
and to the parabolic Harnack inequality for solution to the corresponding heat equation, extending results of 
Grigoryan \cite{Grigoryan} and Saloff-Coste \cite{SC92} for manifolds. Under the same assumptions, he also obtained in \cite{Delmotte98} 
explicit bounds on the dimension of harmonic functions on $G$ of at most polynomial growth. 
Results of this flavor are classical in geometric analysis, with seminal ideas going back 
to the work of De Giorgi \cite{DeGiorgi}, Nash \cite{Nash}, and Moser \cite{Moser61,Moser64} on the regularity of solutions of uniformly elliptic second order equations in divergence form.

The main focus of this paper is on random graphs, and more specifically on random subgraphs of $\Z^d$, $d\geq2$. 
Because of local defects in such graphs caused by randomness, it is too restrictive to expect that various properties 
(e.g., Poincar\'e inequality, Gaussian bounds, or Harnack inequality) should hold globally. 
An illustrative example is the infinite cluster $\mathcal C_\infty$ of supercritical Bernoulli percolation \cite{Grimmett} defined as follows. 
For $p\in[0,1]$, remove vertices of $\Z^d$ independently with probability $(1-p)$. The graph induced by the retained vertices 
almost surely contains an infinite connected component (which is unique) if $p>p_c(d)\in(0,1)$, and contains only finite components 
if $p<p_c(d)$. It is easy to see that for any $p>p_c(d)$ with probability $1$, $\mathcal C_\infty$ contains copies of any finite connected subgraph of $\Z^d$, 
and thus, none of the above global properties can hold. 

Barlow \cite{Barlow} proposed the following relaxed assumption which takes into account possible exceptional behavior on microscopic scales. 

\begin{definition}\label{def:vgb}(\cite[Definition~1.7]{Barlow})
Let $C_V$, $C_P$, and $C_W\geq 1$ be fixed constants. 
For $r\geq 1$ integer and $x\in V(G)$, we say that $\ballG(x,r)$ is $(C_V,C_P,C_W)$-{\it good} if 
$\mu(\ballG(x,r))\geq C_V r^d$ and the weak Poincar\'e inequality
\[
\min_a\int_{\ballG(x,r)}(f-a)^2 d\mu \leq C_P\cdot r^2\cdot \int_{E(\ballG(x,C_Wr))} |\nabla f|^2 d\nu.
\]
holds for all $f:\ballG(x,C_Wr) \to \R$.

We say $\ballG(x,R)$ is $(C_V,C_P,C_W)$-{\it very good} if there exists $N_{\ballG(x,R)}\leq R^{\frac{1}{d+2}}$ such that 
$\ballG(y,r)$ is $(C_V,C_P,C_W)$-good whenever $\ballG(y,r)\subseteq \ballG(x,R)$, and $N_{\ballG(x,R)} \leq r\leq R$. 
\end{definition}
\begin{remark}\label{rem:wpi:minimum}
For any finite $H\subset V(G)$ and $f:H\to\R$, the minimum $\min_a\int_H(f-a)^2 d\mu$ is attained by 
the value $a = \overline f_H = \frac{1}{\mu(H)}\int_H f d\mu$. 
\end{remark}

\medskip

For a very good ball, the conditions of volume growth and Poincar\'e inequality are allowed to fail on microscopic scales. 
Thus, if all large enough balls are very good, the graph can still have rather irregular local behavior. 
Despite that, on large enough scales it looks as if it was regular on all scales, as the following results from \cite{Barlow, BH09, BDCKY14} illustrate. 

\medskip

Let $X = (X_n)_{n\geq 0}$ and $Y = (Y_t)_{t\geq 0}$ be the discrete and continuous time simple random walks on $G$. 
$X$ is a Markov chain with transition probabilities $\frac{\nu_{xy}}{\mu_x}$, and $Y$ is the Markov process with generator 
$\mathcal L_G f(x) = \frac{1}{\mu_x}\sum_y\nu_{xy}(f(y) - f(x))$. 
In words, the walker $X$ (resp., $Y$) waits a unit time (resp., an exponential time with mean $1$) at each vertex $x$, and then jumps to a uniformly chosen neighbor of $x$ in $G$.
For $x\in V(G)$, we denote by $\mathrm P_x = \mathrm P_{G,x}$ (resp., $\mathrm Q_x = \mathrm Q_{G,x}$) the law of $X$ (resp., $Y$) started from $x$. 
The transition density of $X$ (resp., $Y$) with respect to $\mu$ is denoted by $p_n(x,y) = p_{G,n}(x,y) = \frac{\mathrm P_{G,x}[X_n = y]}{\mu_y}$ 
(resp., $q_t(x,y) = q_{G,t}(x,y) = \frac{\mathrm Q_{G,x}[Y_t = y]}{\mu_y}$).

\medskip

The first implications of Definition~\ref{def:vgb} are large time Gaussian bounds for $q_t$ and $p_n$. 
\begin{theorem}\label{thm:hk:vgb}(\cite[Theorem~5.7(a)]{Barlow} and \cite[Theorem~2.2]{BH09})
Let $x\in V(G)$. If there exists $R_0 = R_0(x,G)$ such that $\ballG(x,R)$ is $(C_V,C_P,C_W)$-very good with $N_{\ballG(x,R)}^{3(d+2)}\leq R$ for each $R\geq R_0$, 
then there exist constants $C_i = C_i(d,C_0,C_V,C_P,C_W)$ such that for all $t\geq R_0^{3/2}$ and $y\in V(G)$, 
\begin{equation}\label{eq:hk:vgb:ub}
F_t(x,y)\leq C_1\cdot t^{-\frac d2}\cdot e^{-C_2\cdot \frac{\distG(x,y)^2}{t}},\qquad \text{if $t\geq \distG(x,y)$,}
\end{equation}
\begin{equation}\label{eq:hk:vgb:lb}
F_t(x,y)\geq C_3\cdot t^{-\frac d2}\cdot e^{-C_4\cdot \frac{\distG(x,y)^2}{t}},\qquad \text{if $t\geq \distG(x,y)^{\frac 32}$,}
\end{equation}
where $F_t$ stands for either $q_t$ or $p_{\lfloor t\rfloor} + p_{\lfloor t\rfloor +1}$. 
\end{theorem}
The next result gives an elliptic Harnack inequality.
\begin{theorem}\label{thm:ehi:vgb}(\cite[Theorem~5.11]{Barlow})
There exists $C_{\scriptscriptstyle {\mathrm{ehi}}} = C_{\scriptscriptstyle {\mathrm{ehi}}}(d,C_0,C_V,C_P,C_W)$ such that for any $x\in V(G)$ and $R\geq 1$, if $\ballG(x,R\log R)$ is $(C_V,C_P,C_W)$-very good with $N_{\ballG(x,R\log R)}^{4(d+2)}\leq R$, 
then for any $y\in\ballG(x,\frac 13 R\log R)$, and $h:\ballG(y,R+1)\to\R$ nonnegative and harmonic in $\ballG(y,R)$, 
\begin{equation}\label{eq:ehi:vgb}
\sup_{\ballG(y,\frac12 R)} h\leq C_{\scriptscriptstyle {\mathrm{ehi}}}\cdot \inf_{\ballG(y,\frac12 R)} h.
\end{equation}
\end{theorem}
In fact, more general parabolic Harnack inequality also takes place. 
(For the definition of parabolic Harnack inequality see, e.g., \cite[Section~3]{BH09}.)
\begin{theorem}\label{thm:phi:vgb}(\cite[Theorem~3.1]{BH09})
There exists $C_{\scriptscriptstyle {\mathrm{phi}}} = C_{\scriptscriptstyle {\mathrm{phi}}}(d,C_0,C_V,C_P,C_W)$ such that 
for any $x\in V(G)$, $R\geq 1$, and $R_1 = R\log R\geq 16$, if $\ballG(x,R_1)$ is $(C_V,C_P,C_W)$-very good with 
$N_{\ballG(x,R_1)}^{2(d+2)}\leq \frac{R_1}{2\log R_1}$, then for any $y\in\ballG(x,\frac 13 R_1)$, 
the parabolic Harnack inequality (in both discrete and continuous time settings) holds with constant $C_{\scriptscriptstyle {\mathrm{phi}}}$ 
for $(0,R^2]\times\ballG(y,R)$.
In particular, the elliptic Harnack inequality \eqref{eq:ehi:vgb} also holds.
\end{theorem}
Next result is about the dimension of the space of harmonic functions on $G$ with at most polynomial growth. 
\begin{theorem}\label{thm:vgb:hf:poligrowth}(\cite[Theorem~4]{BDCKY14})
Let $x\in V(G)$. If there exists $R_0 = R_0(x,G)$ such that $\ballG(x,R)$ is $(C_V,C_P,C_W)$-very good for each $R\geq R_0$, 
then for any positive $k$, the space of harmonic functions $h$ with $\limsup_{\distG(x,y)\to\infty}\frac{h(y)}{\distG(x,y)^k}<\infty$ is finite dimensional, 
and the bound on the dimension only depends on $k$, $d$, $C_0$, $C_V$, $C_P$, and $C_W$. 
\end{theorem}

\medskip

The notion of very good balls is most useful in studying random subgraphs of $\Z^d$. 
Up to now, it was only applied to the unique infinite connected component of supercritical Bernoulli percolation, see \cite{Barlow, BH09}. 
Barlow \cite[Section 2]{Barlow} showed that on an event of probability $1$, 
for every vertex of the infinite cluster, all large enough balls centered at it are very good. 
Thus, all the above results are immediately transfered into the almost sure statements for all 
vertices of the infinite cluster. 

Despite the conditions of Definition~\ref{def:vgb} are rather general, 
their validity up to now has only been shown for the independent percolation. 
The reason is that most of the analysis developed for percolation is tied very sensitively with 
the independence property of Bernoulli percolation.
One usually first reduces combinatorial complexity of patterns by a coarse graining, 
and then balances the complexity out by exponential bounds coming from the independence, 
see, e.g., \cite[Section 2]{Barlow}.

The main purpose of this paper is to develop an approach to verifying properties of Definition~\ref{def:vgb} for random graphs 
which does not rely on independence or any comparison with Bernoulli percolation, and, 
as a result, extending the known results about Bernoulli percolation to models with strong correlations. 
Our primal motivation comes from percolation models with strong correlations, 
such as random interlacements, vacant set of random interlacements, or the level sets of the Gaussian free field, 
see, e.g., \cite{SznitmanAM,SidoraviciusSznitman_RI,RodSz}. 

\medskip

\begin{remark}\label{rem:vgb}
\begin{itemize}\itemsep0pt
\item[(1)]
The lower bound of Theorem~\ref{thm:hk:vgb} can be slightly generalized by following the proof of \cite[Theorem~5.7(a)]{Barlow}. 
Let $\epsilon\in(0,\frac 12]$ and $K>\frac{1}{\epsilon}$. 
If there exists $R_0 = R_0(x,G)$ such that $\ballG(x,R)$ is $(C_V,C_P,C_W)$-very good with $N_{\ballG(x,R)}^{K(d+2)}\leq R$ for each $R\geq R_0$, 
then for all $t\geq R_0^{1+\epsilon}$, 
\begin{equation}\label{eq:hk:vgb:lb:epsilon}
F_t(x,y)\geq C_3\cdot t^{-\frac d2}\cdot e^{-C_4\cdot \frac{\distG(x,y)^2}{t}},\qquad \text{if $t\geq \distG(x,y)^{1+\epsilon}$.}
\end{equation}
The constants $C_3$ and $C_4$ are the same as in \eqref{eq:hk:vgb:lb}, in particular, they do not depend on $K$ and $\epsilon$.
For $\epsilon=\frac 12$ and $K=3$, we recover \eqref{eq:hk:vgb:lb}.
(There is a small typo in the statements of \cite[Theorem~5.7(a)]{Barlow} and \cite[Theorem~2.2]{BH09}: $R_0^{2/3}$ should be replaced by $R_0^{3/2}$.)

Indeed, the proof of \cite[Theorem~5.7(a)]{Barlow} is reduced to verifying assumptions of \cite[Theorem~5.3]{Barlow} for some choice of $R$. 
The original choice of Barlow is $R = t^{\frac23}$, and it implies \eqref{eq:hk:vgb:lb}. 
By restricting the choice of $N_{\ballG(x,R)}$ as above, one notices that 
all the conditions of \cite[Theorem~5.3]{Barlow} are satisfied by $R = t^{\frac{1}{1+\epsilon}}$, implying \eqref{eq:hk:vgb:lb:epsilon}.
\item[(2)]
In order to prove the lower bound of \eqref{eq:hk:vgb:lb} for the same range of $t$'s as in the upper bound \eqref{eq:hk:vgb:ub}, one needs to impose 
a stronger assumption on the regularity of the balls $\ballG(x,R)$ (see, for instance, \cite[Definition~5.4]{Barlow} of the exceedingly good ball 
and \cite[Theorem~5.7(b)]{Barlow}).
In fact, the recent result of \cite[Theorem~1.10]{BarlowChen14} states that the volume doubling property and the Poincar\'e inequality 
satisfied by large enough balls are equivalent to certain partial Gaussian bounds (and also to the parabolic Harnack inequality in large balls).
\item[(3)]
Under the assumptions of Theorem~\ref{thm:phi:vgb}, various estimates of the heat kernels for the processes $X$ and $Y$ killed on exiting from a box 
are given in \cite[Theorem~2.1]{BH09}. 
\item[(4)]
Theorem~\ref{thm:vgb:hf:poligrowth} holds under much weaker assumptions, although reminiscent of the ones of Definition~\ref{def:vgb} (see \cite[Theorem~4]{BDCKY14}).  
Roughly speaking, one assumes that the conditions from Definition~\ref{def:vgb} hold with $N_{\ballG(x,R)}$ only sublinear in $R$, i.e., a volume growth condition and 
the weak Poincar\'e inequality should hold only for macroscopic subballs of $\ballG(x,R)$.  
\end{itemize}
\end{remark}

\subsection{The model}

We consider the measurable space $\Omega = \{0,1\}^{\Z^d}$, $d\geq 2$, equipped with 
the sigma-algebra $\mathcal F$ generated by the coordinate maps $\{\atom\mapsto\atom(x)\}_{x\in\Z^d}$.
For any $\atom\in\{0,1\}^{\Z^d}$, we denote the induced subset of $\Z^d$ by 
\[
\set = \set(\atom) = \{x\in\Z^d~:~\atom(x) = 1\} \subseteq \Z^d .\
\]
We view $\set$ as a subgraph of $\Z^d$ in which the edges are drawn between any two vertices of $\set$ within $\ell^1$-distance $1$ from each other, 
where the $\ell^1$ and $\ell^\infty$ norms of $x=(x(1),\dots,x(d))\in \R^d$ are defined in the usual way by 
$|x|_1 = \sum_{i=1}^d|x(i)|$ and $|x|_\infty = \max\{|x(1)|,\ldots|x(d)|\}$, respectively.
For $x \in \Z^d$ and $r \in \R_+$, we denote by 
$\ballZ(x,r) = \{y\in\Z^d~:~|x-y|_\infty\leq \lfloor r \rfloor \}$ the closed $\ell^{\infty}$-ball in $\Z^d$ with 
radius $\lfloor r \rfloor$ and center at $x$.

\begin{definition}\label{def:setr}
For $r\in [0,\infty]$, we denote by $\set_r$, the set of vertices of $\set$ 
which are in connected components of $\set$ of $\ell^1$-diameter $\geq r$. 
In particular, $\set_\infty$ is the subset of vertices of $\set$ which are in infinite connected components of $\set$. 
\end{definition}

\subsubsection{Assumptions}

On $(\Omega,\mathcal F)$ we consider a family of probability measures $(\mathbb P^u)_{u\in(a,b)}$ with $0<a<b<\infty$, 
satisfying the following assumptions \p{} -- \ppp{} and \s{} -- \sss{} from \cite{DRS12}. 
Parameters $d$, $a$, and $b$ are considered fixed throughout the paper, and dependence of various constants on them is omitted.

\medskip

An event $G \in \mathcal F$ is called \emph{increasing} (respectively, \emph{decreasing}), if
for all $\atom\in G$ and $\atom' \in \{0,1\}^{\Z^d}$ with $\atom(y) \leq \atom(y')$ 
(respectively, $\atom(y) \geq \atom(y')$) for all $y\in\Z^d$, one has $\atom' \in G$.

\begin{itemize}
\item[\p{}] {\it (Ergodicity)}
For each $u\in(a,b)$, every lattice shift is measure preserving and ergodic on $(\Omega,\mathcal F,\mathbb P^u)$.
\item[\pp{}] {\it (Monotonicity)}
For any $u,u'\in(a,b)$ with $u<u'$, and any increasing event $G\in\mathcal F$, 
$\mathbb P^u[G] \leq \mathbb P^{u'}[G]$.
\item[\ppp{}] {\it (Decoupling)}
Let $L\geq 1$ be an integer and $x_1,x_2\in\Z^d$. 
For $i\in\{1,2\}$, let $A_i\in\sigma(\{\atom\mapsto\atom(y)\}_{y\in \ballZ(x_i,10L)})$ be decreasing events, and 
$B_i\in\sigma(\{\atom\mapsto\atom(y)\}_{y\in \ballZ(x_i,10L)})$ increasing events. 
There exist $\RP,\LP <\infty$ and $\epsP,\constP>0$ such that for any integer $R\geq \RP$ and $a<\widehat u<u<b$ satisfying 
\[
u\geq \left(1 + R^{-\constP}\right)\cdot \widehat u ,\
\]
if $|x_1 - x_2|_\infty \geq R\cdot L$, then 
\[
\mathbb P^u\left[A_1\cap A_2\right] \leq 
\mathbb P^{\widehat u}\left[A_1\right] \cdot
\mathbb P^{\widehat u}\left[A_2\right] 
+ e^{-\funcP(L)} ,\
\]
and
\[
\mathbb P^{\widehat u}\left[B_1\cap B_2\right] \leq 
\mathbb P^u\left[B_1\right] \cdot
\mathbb P^u\left[B_2\right] 
+ e^{-\funcP(L)} ,\
\]
where $\funcP$ is a real valued function satisfying $\funcP(L) \geq e^{(\log L)^\epsP}$ for all $L\geq \LP$.
\item[\s{}] {\it (Local uniqueness)}
There exists a function $\funcS:(a,b)\times\Z_+\to \mathbb R$ such that for each $u\in(a,b)$,
\begin{equation}\label{eq:funcS}
\begin{array}{c}
\text{there exist $\constS = \constS(u)>0$ and $\RS = \RS(u)<\infty$}\\
\text{such that $\funcS(u,R) \geq (\log R)^{1+\constS}$ for all $R\geq \RS$,}
\end{array}
\end{equation}
and for all $u\in(a,b)$ and $R\geq 1$, the following inequalities are satisfied:
\begin{equation*}
\mathbb P^u\left[ \, 
\set_R\cap\ballZ(0,R) \neq \emptyset \,
\right]
\geq 
1 - e^{-\funcS(u,R)} , 
\end{equation*}
and
\begin{equation*}
\mathbb P^u\left[
\begin{array}{c}
\text{for all $x,y\in\set_{\scriptscriptstyle{R/10}}\cap\ballZ(0,R)$,}\\
\text{$x$ is connected to $y$ in $\set\cap\ballZ(0,2R)$}
\end{array}
\right]
\geq 1 - e^{-\funcS(u,R)} . 
\end{equation*}
\item[\sss{}] {\it (Continuity)}
Let $\eta(u) = \mathbb P^u\left[0\in\set_\infty\right]$. The function $\eta(\cdot)$ is positive and continuous on $(a,b)$. 
\end{itemize}

\medskip

\begin{remark}
\begin{itemize}\itemsep0pt
\item[(1)]
The use of assumptions \pp{}, \ppp{}, and \sss{} will not be explicit in this paper. 
They are only used to prove likeliness of certain patterns in $\set_\infty$ produced by 
a multi-scale renormalization, see \eqref{eq:nbad:proba}. 
(Of course, they are also used in already known results of Theorems~\ref{thm:chd} and \ref{thm:qip}). 
Roughly speaking, we use \ppp{} repeatedly on multiple scales for a convergent sequence of parameters $u_k$ and use \pp{} and \sss{} 
to establish convergence of iterations.
\item[(2)]
If the family $\mathbb P^u$, $u\in(a,b)$, satisfies \s{}, then a union bound argument gives that for any $u\in(a,b)$, 
$\mathbb P^u$-a.s., the set $\set_\infty$ is non-empty and connected, 
and there exist constants $C_i = C_i(u)$ such that for all $R\geq 1$, 
\begin{equation}\label{eq:C1:infty}
\mathbb P^u\left[ \,
\set_\infty\cap\ballZ(0,R) \neq \emptyset \, 
\right]
\geq 
1 - C_1 \cdot e^{-C_2\cdot(\log R)^{1+\constS}} . 
\end{equation}
\end{itemize}
\end{remark} 

\medskip

\subsubsection{Examples}\label{sec:examples}

Here we briefly list some motivating examples (already announced earlier in the paper) of families of probability measures satisfying assumptions \p{} -- \ppp{} and \s{} -- \sss{}. 
All these examples were considered in details in \cite{DRS12}, and 
we refer the interested reader to \cite[Section~2]{DRS12} for the proofs and further details. 

\begin{itemize}\itemsep0pt
\item[(1)]
Bernoulli percolation with parameter $u\in[0,1]$ corresponds to the product measure $\mathbb P^u$ with $\mathbb P^u[\atom(x) = 1] = 1 - \mathbb P^u[\atom(x) = 0] = u$. 
The family $\mathbb P^u$, $u\in(a,b)$, satisfies assumptions \p{} -- \ppp{} and \s{} -- \sss{} for any $d\geq 2$ and $p_c(d)<a<b\leq 1$, see \cite{Grimmett}.
\item[(2)]
Random interlacements at level $u>0$ is the random subgraph of $\Z^d$, $d\geq 3$, corresponding to the measure $\mathbb P^u$ defined by the equations 
\[
\mathbb P^u[\set\cap K=\emptyset] = e^{-u\cdot \mathrm{cap}(K)},\qquad \text{for all finite }K\subset\Z^d,
\]
where $\mathrm{cap}(\cdot)$ is the discrete capacity. 
It follows from \cite{RS:Transience,SznitmanAM,Sznitman:Decoupling} that the family $\mathbb P^u$, $u\in(a,b)$, satisfies 
assumptions \p{} -- \ppp{} and \s{} -- \sss{} for any $0<a<b<\infty$.
Curiously, for any $u>0$, $\set$ is $\mathbb P^u$-almost surely connected \cite{SznitmanAM}, i.e., $\set_\infty=\set$.
\item[(3)]
Vacant set of random interlacements at level $u>0$ is the complement of the random interlacements at level $u$ in $\Z^d$. 
It corresponds to the measure $\mathbb P^u$ defined by the equations 
\[
\mathbb P^u[K\subseteq \set] = e^{-u\cdot \mathrm{cap}(K)},\qquad \text{for all finite }K\subset\Z^d.
\]
Unlike random interlacements, the vacant set undergoes a percolation phase transition in $u$ \cite{SznitmanAM,SidoraviciusSznitman_RI}.
If $u<u_*(d)\in(0,\infty)$ then $\mathbb P^u$-almost surely $\set_\infty$ is non-empty and connected, and if $u>u_*(d)$, $\set_\infty$ is $\mathbb P^u$-almost surely empty. 
It is known that the family $\mathbb P^{\frac1u}$, $u\in(a,b)$, satisfies 
assumptions \p{} -- \ppp{} for any $0<a<b<\infty$ \cite{SznitmanAM,Sznitman:Decoupling}, \sss{} for any $\frac{1}{u_*(d)}<a<b<\infty$ \cite{Teixeira09}, 
and \s{} for some $\frac{1}{u_*(d)}<a<b<\infty$ \cite{DRS}.
\item[(4)]
The Gaussian free field on $\Z^d$, $d \geq 3$, is a centered Gaussian field with covariances given by the Green function of 
the simple random walk on $\Z^d$. The excursion set above level $h\in\R$ is the random subset of $\Z^d$ where the fields exceeds $h$. 
Let $\mathbb P^h$ be the measure on $\Omega$ for which $\set$ has the law of the excursion set above level $h$. 
The model exhibits a non-trivial percolation phase transition \cite{BLM,RodSz}. 
If $h<h_*(d)\in[0,\infty)$ then $\mathbb P^h$-almost surely $\set_\infty$ is non-empty and connected, and if $h>h_*(d)$, $\set_\infty$ is $\mathbb P^h$-almost surely empty. 
It was proved in \cite{DRS12,RodSz} that the family $\mathbb P^{h_*(d)-h}$, $h\in(a,b)$, satisfies 
assumptions \p{} -- \ppp{} and \sss{} for any $0<a<b<\infty$, and \s{} for some $0<a<b<\infty$.
\end{itemize}

\medskip

The last three examples are particularly interesting, since they have {\it polynomial decay of spatial correlations} 
and cannot be studied by comparison with Bernoulli percolation on any scale. 
In particular, many of the methods developed for Bernoulli percolation do not apply. 
As we see from the examples, assumptions \p{} -- \ppp{} and \sss{} are satisfied by all the $4$ models through their {\it whole} 
supercritical phases. However, assumption \s{} is currently verified for the whole range of interesting parameters only 
in the cases of Bernoulli percolation and random interlacements, and only for a non-empty subset of interesting parameters in the last two examples. 
We call all the parameters $u$ for which $\mathbb P^u$ satisfies \s{} the regime of {\it local uniqueness} 
(since under \s{}, there is a unique giant cluster in each large box). 
It is a challenging open problem to verify if the regime of local uniqueness coincides with 
the supercritical phase for the vacant set of random interlacements and the level sets of the Gaussian free field. 
A positive answer to this question will imply that all the results of this paper hold unconditionally also for the last two considered examples through their whole 
supercritical phases. 

\medskip

\subsubsection{Known results}

Below we recall some results from \cite{DRS12,PRS} about the large scale behavior of graph distances in $\set_\infty$ 
and the quenched invariance principle for the simple random walk on $\set_\infty$. 
Both results are formulated in the form suitable for our applications. 

\begin{theorem}\label{thm:chd}(\cite[Theorem~1.3]{DRS12})
Let $d\geq 2$ and $\chd\in(0,1)$. 
Assume that the family of measures $\mathbb P^u$, $u\in(a,b)$, satisfies assumptions \p{} -- \ppp{} and \s{} -- \sss{}. 
Let $u\in(a,b)$. There exist $\Omega_{\scriptscriptstyle {\mathrm{chd}}}\in\mathcal F$ with $\mathbb P^u[\Omega_{\scriptscriptstyle {\mathrm{chd}}}] = 1$, 
constants $C_{\scriptscriptstyle {\mathrm{chd}}}$, $c_{\scriptscriptstyle {\ref{thm:chd}}}$, and $C_{\scriptscriptstyle {\ref{thm:chd}}}$ all dependent on $u$ and $\chd$,  
and random variables $\Rchd(x)$, $x\in\Z^d$, such that 
for all $\omega\in\Omega_{\scriptscriptstyle {\mathrm{chd}}}\cap\{0\in\set_\infty\}$ and $x\in\set_\infty(\omega)$, 
\begin{itemize}\itemsep0pt
\item[(a)]
$\Rchd(x,\omega)<\infty$,
\item[(b)]
for all $R\geq \Rchd(x,\omega)$ and $y,z\in\ballZ_{\Z^d}(x,R)\cap\set_\infty(\omega)$, 
\[
\dist_{\set_\infty(\omega)}(y,z)\leq C_{\scriptscriptstyle {\mathrm{chd}}} \cdot \max\left\{\dist_{\Z^d}(y,z),~R^{\chd}\right\},
\]
\item[(c)]
for all $z\in\Z^d$ and $r\geq 1$, 
\[
\mathbb P^u[\Rchd(z)\geq r]\leq C_{\scriptscriptstyle {\ref{thm:chd}}}\cdot e^{-c_{\scriptscriptstyle {\ref{thm:chd}}}\cdot (\log r)^{1 + \constS}},
\]
where $\constS$ is defined in \eqref{eq:funcS}.
\end{itemize}
\end{theorem}
For $T>0$, let $C[0,T]$ be the space of continuous functions from $[0,T]$ to $\R^d$, 
and $\mathcal W_T$ the Borel sigma-algebra on it. 
Let 
\begin{equation}\label{def:widetildeBn}
\widetilde B_n(t) = \frac{1}{\sqrt n}\left(X_{\lfloor tn\rfloor} + (tn-\lfloor tn\rfloor)\cdot(X_{\lfloor tn\rfloor + 1} - X_{\lfloor tn\rfloor})\right).
\end{equation}
\begin{theorem}\label{thm:qip} (\cite[Theorem~1.1, Lemma~A.1, and Section~5]{PRS})
Let $d\geq 2$. 
Assume that the family of measures $\mathbb P^u$, $u\in(a,b)$, satisfies assumptions \p{} -- \ppp{} and \s{} -- \sss{}. 
Let $u\in(a,b)$ and $T>0$. There exist $\Omega_{\scriptscriptstyle {\mathrm{qip}}}\in\mathcal F$ with $\mathbb P^u[\Omega_{\scriptscriptstyle {\mathrm{qip}}}] = 1$ 
and a non-degenerate matrix $\Sigma = \Sigma(u)$, such that for all $\omega\in\Omega_{\scriptscriptstyle {\mathrm{qip}}}\cap\{0\in\set_\infty\}$, 
\begin{itemize}\itemsep0pt
\item[(a)]
there exists $\chi:\set_\infty(\omega)\to\R^d$ such that $x\mapsto x+\chi(x)$ is harmonic on $\set_\infty(\omega)$, 
and $\lim_{n\to\infty}\frac1n\max_{x\in\set_\infty\cap\ballZ(0,n)}|\chi(x)| = 0$,
\item[(b)]
the law of $\left(\widetilde B_n(t)\right)_{0\leq t\leq T}$ on $(C[0,T],\mathcal W_T)$ converges weakly (as $n\to\infty$) to the law of Brownian motion 
with zero drift and covariance matrix $\Sigma$.
\end{itemize}
In addition, if reflections and rotations of $\Z^d$ by $\frac\pi2$ preserve $\mathbb P^u$, then the limiting Brownian motion isotropic, i.e., 
$\Sigma = \sigma^2\cdot \mathrm{I}_d$ with $\sigma^2>0$. 
\end{theorem}
\begin{remark}\label{rem:qip}
\cite[Theorem~1.1]{PRS} is stated for the (``blind'') random walk which jumps to a neighbor with probability $\frac{1}{2d}$ and stays put with probability $1 - \frac{1}{2d}\cdot(\text{number of neighbors})$. 
Since the blind walk and the simple random walk are time changes of each other, 
the invariance principle for one process implies the one for the other (see, for instance, \cite[Lemma~6.4]{BergerBiskup}).
\end{remark}

\medskip

\subsection{Main results}

The main contribution of this paper is Theorem~\ref{thm:vgb:main}, where we prove that 
under the assumptions \p{} -- \ppp{} and \s{} -- \sss{}, 
all large enough balls in $\set_\infty$ are very good in the sense of Definition~\ref{def:vgb}. 
This result has many immediate applications, including Gaussian heat kernel bounds, Harnack inequalities, and finiteness of the dimension of harmonic functions on $\set_\infty$ with 
prescribed polynomial growth, see Theorems~\ref{thm:hk:vgb}, \ref{thm:phi:vgb}, \ref{thm:ehi:vgb}, \ref{thm:vgb:hf:poligrowth}. 
In fact, all the results from \cite{BH09,BDCKY14} can be easily translated from Bernoulli percolation to our setting, since (as also pointed out by the authors) 
their proofs only rely on (some combinations of) stationarity, Gaussian heat kernel bounds, and the invariance principle. 
Among such results are estimates on the gradient of the heat kernel (Theorem~\ref{thm:hk:grad}) and on the Green function (Theorem~\ref{thm:gf:bounds}), 
which will be deduced from the heat kernel bounds by replicating the proofs of \cite[Theorem~6]{BDCKY14} and \cite[Theorem~1.2(a)]{BH09}, 
the fact that the dimension of at most linear harmonic functions on $\set_\infty$ is $d+1$ (Theorem~\ref{thm:hf:d+1}), 
the local central limit theorem (Theorem~\ref{thm:localclt}), and the asymptotic for the Green function (Theorem~\ref{thm:gf:asymp}), 
which we derive from the heat kernel bounds and the quenched invariance principle by mimicking the proofs of \cite[Theorem~5]{BDCKY14}, \cite[Theorem~1.1]{BH09}, and \cite[Theorem~1.2(b,c)]{BH09}.

We begin by stating the main result of this paper. 

\begin{theorem}\label{thm:vgb:main}
Let $d\geq 2$ and $\vgb\in(0,\frac{1}{d+2})$. 
Assume that the family of measures $\mathbb P^u$, $u\in(a,b)$, satisfies assumptions \p{} -- \ppp{} and \s{} -- \sss{}. 
Let $u\in(a,b)$. There exist $\Omega_{\scriptscriptstyle {\mathrm{vgb}}}\in\mathcal F$ with $\mathbb P^u[\Omega_{\scriptscriptstyle {\mathrm{vgb}}}] = 1$, 
constants $C_V$, $C_P$, $C_W$, $c_{\scriptscriptstyle {\ref{thm:vgb:main}}}$, and $C_{\scriptscriptstyle {\ref{thm:vgb:main}}}$ all dependent on $u$ and $\vgb$, 
and random variables $\Rvgb(x)$, $x\in\Z^d$, such that 
for all $\omega\in\Omega_{\scriptscriptstyle {\mathrm{vgb}}}\cap\{0\in\set_\infty\}$ and $x\in\set_\infty(\omega)$, 
\begin{itemize}\itemsep0pt
\item[(a)]
$\Rvgb(x,\omega)<\infty$,
\item[(b)]
for all $R\geq \Rvgb(x,\omega)$, $\ballZ_{\set_\infty(\omega)}(x,R)$ is $(C_V,C_P,C_W)$-very good 
with $N_{\ballZ_{\set_\infty(\omega)}(x,R)}\leq R^{\vgb}$,
\item[(c)]
for all $z\in\Z^d$ and $r\geq 1$, 
\begin{equation}\label{eq:vgb:main:R}
\mathbb P^u[\Rvgb(z)\geq r]\leq C_{\scriptscriptstyle {\ref{thm:vgb:main}}}\cdot e^{-c_{\scriptscriptstyle {\ref{thm:vgb:main}}}\cdot (\log r)^{1 + \constS}},
\end{equation}
where $\constS$ is defined in \eqref{eq:funcS}.
\end{itemize}
\end{theorem}

Theorem~\ref{thm:vgb:main} will immediately follow from a certain isoperimetric inequality, see Definition~\ref{def:vrb}, Claim~\ref{cl:rb-gb}, and Proposition~\ref{prop:verygoodbox}. 
This isoperimetric inequality is more than enough to imply the weak Poincar\'e inequality that we need. 
In fact, as we learned from a discussion with Jean-Dominique Deuschel, it implies stronger Sobolev inequalities,
and may be useful in situations beyond the goals of this paper (see, e.g., \cite[Section~3]{Nguyen}). 

\medskip

\begin{corollary}\label{cor:vgb:main}
Theorem~\ref{thm:vgb:main} immediately implies that all the results of Theorems~\ref{thm:hk:vgb}, \ref{thm:phi:vgb}, \ref{thm:ehi:vgb}, and \ref{thm:vgb:hf:poligrowth}
hold almost surely for $G = \set_\infty$. Since the constants $C_V$, $C_P$, and $C_W$ in the statement of Theorem~\ref{thm:vgb:main} are deterministic, 
all the constants in Theorems~\ref{thm:hk:vgb}, \ref{thm:phi:vgb}, \ref{thm:ehi:vgb}, and \ref{thm:vgb:hf:poligrowth} are also deterministic.
\end{corollary}

Combining Corollary~\ref{cor:vgb:main} with Theorem~\ref{thm:chd} and Remark~\ref{rem:vgb}(1), we notice that the quenched heat kernel bounds of Theorem~\ref{thm:hk:vgb} hold almost surely for $G=\set_\infty$ 
with $\dist_G$ replaced by $\dist_{\Z^d}$ in \eqref{eq:hk:vgb:ub}, \eqref{eq:hk:vgb:lb}, and \eqref{eq:hk:vgb:lb:epsilon}. 
Since we will use the quenched heat kernel bounds often in the paper, we give a precise statement here. 
\begin{theorem}\label{thm:hk:dZd}
Let $d\geq 2$. 
Assume that the family of measures $\mathbb P^u$, $u\in(a,b)$, satisfies assumptions \p{} -- \ppp{} and \s{} -- \sss{}. 
Let $u\in(a,b)$ and $\epsilon>0$. There exist $\Omega_{\scriptscriptstyle {\mathrm{hk}}}\in\mathcal F$ with $\mathbb P^u[\Omega_{\scriptscriptstyle {\mathrm{hk}}}] = 1$, 
constants $C_i = C_i(u)$, $C_{\scriptscriptstyle {\ref{thm:hk:dZd}}}=C_{\scriptscriptstyle {\ref{thm:hk:dZd}}}(u,\epsilon)$, 
and $c_{\scriptscriptstyle {\ref{thm:hk:dZd}}}=c_{\scriptscriptstyle {\ref{thm:hk:dZd}}}(u,\epsilon)$, and random variables $T_{\scriptscriptstyle {\mathrm{hk}}}(x,\epsilon)$, $x\in\Z^d$, such that 
for all $\omega\in\Omega_{\scriptscriptstyle {\mathrm{hk}}}\cap\{0\in\set_\infty\}$ and $x\in\set_\infty(\omega)$, 
\begin{itemize}\itemsep0pt
\item[(a)]
$T_{\scriptscriptstyle {\mathrm{hk}}}(x,\epsilon,\omega)<\infty$,
\item[(b)]
for all $t\geq T_{\scriptscriptstyle {\mathrm{hk}}}(x,\epsilon,\omega)$ and $y\in\set_\infty(\omega)$, 
\begin{equation}\label{eq:hk:ub}
F_t(x,y)\leq C_1\cdot t^{-\frac d2}\cdot e^{-C_2\cdot \frac{\mathrm{D}(x,y)^2}{t}},\qquad \text{if $t\geq \mathrm{D}(x,y)$,}
\end{equation}
\begin{equation}\label{eq:hk:lb}
F_t(x,y)\geq C_3\cdot t^{-\frac d2}\cdot e^{-C_4\cdot \frac{\mathrm{D}(x,y)^2}{t}},\qquad \text{if $t\geq \mathrm{D}(x,y)^{1+\epsilon}$,}
\end{equation}
where $F_t$ stands for either $q_t$ or $p_{\lfloor t\rfloor} + p_{\lfloor t\rfloor +1}$, and $\mathrm{D}$ for either $\dist_{\set_\infty(\omega)}$ or $\dist_{\Z^d}$,
\item[(c)]
for all $z\in\Z^d$ and $r\geq 1$, 
\begin{equation}\label{eq:hk:T0}
\mathbb P^u[T_{\scriptscriptstyle {\mathrm{hk}}}(z,\epsilon)\geq r]\leq C_{\scriptscriptstyle {\ref{thm:hk:dZd}}}\cdot e^{-c_{\scriptscriptstyle {\ref{thm:hk:dZd}}}\cdot (\log r)^{1 + \constS}},
\end{equation}
where $\constS$ is defined in \eqref{eq:funcS}.
\end{itemize}
\end{theorem}
In the applications of Theorem~\ref{thm:hk:dZd} in this paper, {\it we always take} $\epsilon = \frac 12$ (the original choice of Barlow) and omit the dependence on $\epsilon$ from the notation. 
For instance, we will always write $T_{\scriptscriptstyle {\mathrm{hk}}}(x)$ meaning $T_{\scriptscriptstyle {\mathrm{hk}}}(x,\frac 12)$. Any other choice of $\epsilon$ would also do. 

\medskip

It is well known that the parabolic Harnack inequality of Theorem~\ref{thm:phi:vgb} implies H\"older continuity of caloric functions (e.g., $q_t$ and $p_n$), 
see \cite[Proposition~3.2]{BH09}, in particular, by Corollary~\ref{cor:vgb:main} this is true almost surely for $G=\set_\infty$. 
The next result is a sharp bound on the discrete gradient of the heat kernel, proved in \cite[Theorem~6]{BDCKY14} for supercritical Bernoulli percolation using an elegant entropy argument.
\begin{theorem}\label{thm:hk:grad}
Let $d\geq 2$. 
Assume that the family of measures $\mathbb P^u$, $u\in(a,b)$, satisfies assumptions \p{} -- \ppp{} and \s{} -- \sss{}. 
Let $u\in(a,b)$. There exist constants $C_i = C_i(u)$, such that 
for all $x,x',y\in\Z^d$ and $n>\max\left\{\dist_{\Z^d}(x,y),\dist_{\Z^d}(x',y)\right\}$, 
\[
\mathbb E^u\left[\left(p_n(x,y) - p_{n-1}(x',y)\right)^2\cdot \mathds{1}_{\{y\in\set_\infty\}}\cdot\mathds{1}_{\{\text{$x$ and $x'$ are neighbors in $\set_\infty$}\}}\right]
\leq \frac{C_1}{n^{d+1}}\cdot e^{-C_2\cdot \frac{\dist_{\Z^d}(x,y)^2}{n}}.
\]
\end{theorem}

\medskip

The heat kernel bounds of Theorem~\ref{thm:hk:dZd} imply also the following quenched estimates on the Green function 
$g_G(x,y) = \int_0^\infty q_{G,t}(x,y) dt=\sum_{n\geq 0}p_{G,n}(x,y)$ for almost all $G=\set_\infty$.
It is proved in \cite[Theorem~1.2]{BH09} for supercritical Bernoulli percolation, 
but extension to our setting is rather straightforward.
\begin{theorem}\label{thm:gf:bounds}
Let $d\geq 3$. 
Assume that the family of measures $\mathbb P^u$, $u\in(a,b)$, satisfies assumptions \p{} -- \ppp{} and \s{} -- \sss{}. 
Let $u\in(a,b)$. There exist constants $C_i=C_i(u)$ such that 
for all $\omega\in\Omega_{\scriptscriptstyle {\mathrm{hk}}}$ and distinct $x,y\in\set_\infty(\omega)$, if 
$\dist_{\Z^d}(x,y)^2 \geq \min\left\{T_{\scriptscriptstyle {\mathrm{hk}}}(x),T_{\scriptscriptstyle {\mathrm{hk}}}(y)\right\}\cdot\left(1+C_3\cdot \log\dist_{\Z^d}(x,y)\right)$,
then
\[
C_1\cdot \dist_{\Z^d}(x,y)^{2-d}\leq g_{\set_\infty(\omega)}(x,y) \leq C_2\cdot \dist_{\Z^d}(x,y)^{2-d}.
\]
\end{theorem}

\bigskip

The remaining results are derived from the Gaussian heat kernel bounds and the quenched invariance principle. 
In the setting of supercritical Bernoulli percolation, all of them were obtained in \cite{BH09,BDCKY14}, 
but all the proofs extend directly to our setting. 

We begin with results about harmonic functions on $\set_\infty$. 
It is well known that Theorems~\ref{thm:vgb:main} and Theorem~\ref{thm:ehi:vgb} imply the almost sure Liouville property for positive harmonic functions on $\set_\infty$. 
The absence of non-constant sublinear harmonic functions on $\set_\infty$ is even known assuming just stationary of $\set$ (see \cite[Theorem~3 and discussion below]{BDCKY14}). 
In particular, it implies the uniqueness of the function $\chi$ in Theorem~\ref{thm:qip}(a).
The following result about the dimension of at most linear harmonic functions is classical on $\Z^d$. 
It was extended to supercritical Bernoulli percolation on $\Z^d$ in \cite[Theorem~5]{BDCKY14}.
\begin{theorem}\label{thm:hf:d+1}
Let $d\geq 2$. 
Assume that the family of measures $\mathbb P^u$, $u\in(a,b)$, satisfies assumptions \p{} -- \ppp{} and \s{} -- \sss{}. 
Let $u\in(a,b)$. There exist $\Omega_{\scriptscriptstyle {\mathrm{hf}}}\in\mathcal F$ with $\mathbb P^u[\Omega_{\scriptscriptstyle {\mathrm{hf}}}] = 1$ such that 
for all $\omega\in\Omega_{\scriptscriptstyle {\mathrm{hf}}}\cap\{0\in\set_\infty\}$, 
the dimension of the vector space of harmonic functions on $\set_\infty(\omega)$ with at most linear growth equals $d+1$. 
\end{theorem}

\medskip

Since the parabolic Harnack inequality for solutions to the heat equation on $\set_\infty$ implies 
H\"older continuity of $p_n$ and $q_t$, it is possible to replace the weak convergence of Theorem~\ref{thm:qip} by pointwise convergence. 
\cite[Theorems~4.5 and 4.6]{BH09} give general sufficient conditions for the local central limit theorem on general graphs. 
They were verified in \cite[Theorem~1.1]{BH09} for supercritical Bernoulli percolation. 
Theorems~\ref{thm:qip} and \ref{thm:hk:dZd} allow to check these conditions in our setting leading to the following (same as for Bernoulli percolation) result. 
For $x\in\R^d$, $t>0$, the Gaussian heat kernel with covariance matrix $\Sigma$ is defined as 
\[
k_{\Sigma,t}(x) = (2\pi \mathrm{det}(\Sigma) t)^{-\frac d2}\cdot \exp\left(-\frac{x'\Sigma^{-1}x}{2 t}\right),
\]
where $x'$ is the transpose of $x$.
\begin{theorem}\label{thm:localclt}
Let $d\geq 2$. 
Assume that the family of measures $\mathbb P^u$, $u\in(a,b)$, satisfies assumptions \p{} -- \ppp{} and \s{} -- \sss{}. 
Let $u\in(a,b)$, $m=\mathbb E^u[\mu_0\cdot\mathds{1}_{0\in\set_\infty}]$, and $T>0$. 
There exist $\Omega_{\scriptscriptstyle {\mathrm{lclt}}}\in\mathcal F$ with $\mathbb P^u[\Omega_{\scriptscriptstyle {\mathrm{lclt}}}] = 1$, 
and a non-degenerate covariance matrix $\Sigma = \Sigma(u)$ such that for all $\omega\in\Omega_{\scriptscriptstyle {\mathrm{lclt}}}\cap\{0\in\set_\infty\}$, 
\begin{equation}\label{eq:localclt}
\lim_{n\to\infty}\sup_{x\in\R^d}\sup_{t\geq T}\left|n^{\frac d2}\cdot F_{nt}(0,g_n(x)) - \frac{C(F)}{m}\cdot k_{\Sigma,t}(x) \right| = 0,
\end{equation}
where $F_s$ stands for $q_s$ or $p_{\lfloor s\rfloor} + p_{\lfloor s\rfloor +1}$, $C(F)$ is $1$ if $F=q$ and $2$ otherwise, and 
$g_n(x)$ is the closest point in $\set_\infty$ to $\sqrt n x$. 
\end{theorem}

\medskip

Theorems~\ref{thm:hk:dZd} and \ref{thm:localclt} imply the following asymptotic for the Green function, 
extending results of \cite[Theorem~1.2(b,c)]{BH09} to our setting. 
For a covariance matrix $\Sigma$, let $\mathrm{G}_{\Sigma}(x) = \int_0^\infty k_{\Sigma,t}(x)dt$
be the Green function of a Brownian motion with covariance matrix $\Sigma$. In particular, if $\Sigma = \sigma^2\cdot \mathrm{I}_d$, then 
$\mathrm{G}_\Sigma(x) = (2\sigma^2\pi^{\frac d2})^{-1}\Gamma(\frac d2 - 1)|x|^{2-d}$ for all $x\neq 0$, where $|\cdot |$ stands for the Euclidean norm on $\R^d$.
\begin{theorem}\label{thm:gf:asymp}
Let $d\geq 3$. 
Assume that the family of measures $\mathbb P^u$, $u\in(a,b)$, satisfies assumptions \p{} -- \ppp{} and \s{} -- \sss{}. 
Let $u\in(a,b)$, $m$ and $\Sigma$ as in Theorem~\ref{thm:localclt}, and $\varepsilon>0$. 
There exist $\Omega_{\scriptscriptstyle {\mathrm{gf}}}\in\mathcal F$ with $\mathbb P^u[\Omega_{\scriptscriptstyle {\mathrm{gf}}}] = 1$ and a proper random variable $M = M(\varepsilon)$, such that 
for all $\omega\in\Omega_{\scriptscriptstyle {\mathrm{gf}}}\cap\{0\in\set_\infty\}$,
\begin{itemize}\itemsep0pt
\item[(a)]
for all $x\in\set_\infty(\omega)$ with $|x|\geq M$, 
\[
\frac{(1-\varepsilon)\mathrm{G}_\Sigma(x)}{m}\leq g_{\set_\infty(\omega)}(0,x) \leq \frac{(1+\varepsilon)\mathrm{G}_\Sigma(x)}{m},
\]
\item[(b)]
for all $y\in\R^d$, 
$\lim_{k\to\infty}k^{2-d}\cdot \mathbb E^u\left[g_{\set_\infty(\omega)}(0,\lfloor ky\rfloor)~\Big|~0\in\set_\infty\right] = \frac{\mathrm{G}_\Sigma(y)}{m}$.
\end{itemize}
\end{theorem}

\bigskip

\begin{remark}\label{rem:results}
\begin{itemize}\itemsep0pt
\item[(1)]
Let us emphasize that our method does not allow to replace $(\log r)^{1 + \constS}$ in \eqref{eq:vgb:main:R} by $\funcS(u,R)$ from \s{}. 
In particular, even if $\funcS(u,R)$ growth polynomially with $R$, we are not able to improve the bound in \eqref{eq:vgb:main:R} to stretched exponential. 
In the case of independent Bernoulli percolation, it is known from \cite[Section~2]{Barlow} that the result of Theorem~\ref{thm:vgb:main} holds with 
a stretched exponential bound in \eqref{eq:vgb:main:R}. 
\item[(2)]
The fact that the right hand side of \eqref{eq:hk:T0} decays faster than any polynomial 
will be crucially used in the proofs of Theorems~\ref{thm:hk:grad}, \ref{thm:hf:d+1}, and \ref{thm:gf:asymp}. 
Quenched bounds on the diagonal $p_n(x,x)$ under the assumptions \p{} -- \ppp{} and \s{} -- \sss{} were obtained in \cite{PRS} (see Remarks 1.3 (4) and (5) there) for all $n\geq n_0(\omega)$, 
although without any control on the tail of $n_0(\omega)$.  
\item[(3)]
In the case of supercritical Bernoulli percolation, Barlow showed in \cite[Theorem~1]{Barlow} that the bound \eqref{eq:hk:lb} holds for all 
$t\geq\max\{T_{\scriptscriptstyle {\mathrm{hk}}}(x),\mathrm{D}(x,y)\}$. 
The step ``from $\epsilon>0$ to $\epsilon=0$'' is highly nontrivial and follows from the fact that very good boxes on \emph{microscopic} scales are dense, 
see \cite[Definition~5.4 and Theorem~5.7(b)]{Barlow}. 
We do not know if such property can be deduced from the assumptions \p{} -- \ppp{} and \s{} -- \sss{} or proved for any of the specific models  
considered in Section~\ref{sec:examples} (except for Bernoulli percolation). 
Our renormalization does not exclude the possibility of dense mesoscopic traps in $\set_\infty$, but we do not have a counterexample either. 
For comparison, let us mention that the heat kernel bounds \eqref{eq:hk:ub} and \eqref{eq:hk:lb} were obtained in \cite{BD10,ABDH13} for the random conductance model with i.i.d. weights,
where it is also stated in \cite[Remark~3.4]{BD10} and \cite[Remark~4.12]{ABDH13} 
that the lower bound for times comparable with $\mathrm{D}(x,y)$ can likely be obtained by adapting Barlow's proof, 
but omitted there because of a considerable amount of extra work and few applications. 
\item[(4)]
The first proofs of the quenched invariance principle for random walk on the infinite cluster of Bernoulli percolation \cite{SS04,BergerBiskup,MathieuPiatnitski} 
relied significantly on the quenched upper bound on the heat kernel. 
It was then observed in \cite{BiskupPrescott} that it is sufficient to control only the diagonal of the heat kernel 
(proved for Bernoulli percolation in \cite{MathieuRemy}). 
This observation was essential in proving the quenched invariance principle for percolation models satisfying \p{} -- \ppp{} and \s{} -- \sss{} in \cite{PRS}, 
where the desired upper bound on the diagonal of the heat kernel was obtained by means of an isoperimetric inequality (see \cite[Theorem~1.2]{PRS}). 
Theorem~\ref{thm:hk:dZd} allows now to prove the quenched invariance principle of \cite{PRS} by following the original path, for instance, 
by a direct adaptation of the proof of \cite[Theorem~1.1]{BergerBiskup}.
\item[(5)]
Our proof of Theorem~\ref{thm:localclt} follows the approach of \cite{BH09} in the setting of supercritical Bernoulli percolation, namely, it is deduced from 
the quenched invariance principle, parabolic Harnack inequality, and the upper bound on the heat kernel. 
If we replace in \eqref{eq:localclt} $\sup_x$ by $\sup_{|x|<K}$ for any fixed $K>0$, then 
it is not necessary to assume the upper bound on the heat kernel, see \cite[Theorem~1]{CH08}.
\item[(6)]
A new approach to limit theorems and Harnack inequalities for the elliptic random conductance model 
under assumptions on moments of the weights and their reciprocals has been recently developed in \cite{ADS13,ADS14}. 
It relies on Moser's iteration and new weighted Sobolev and Poincar\'e inequalities, and 
is applicable on general graphs 
satisfying globally conditions of regular volume growth and an isoperimetric inequality (see \cite[Assumption~1.1]{ADS14}).
We will comment more on these conditions in Remark~\ref{rem:ADS}. 
The method of \cite{ADS13} was recently used in \cite{Nguyen} to prove the quenched invariance principle for 
the random conductance model on the infinite cluster of supercritical Bernoulli percolation under the same assumptions on moments of the weights as in \cite{ADS13}.
\end{itemize}
\end{remark}

\medskip

\subsection{Some words about the proof of Theorem~\ref{thm:vgb:main}}

Theorem~\ref{thm:chd} is enough to control the volume growth, thus we only discuss here the weak Poincar\'e inequality. 
A finite subset $H$ of $V(G)$ satisfies the (strong) Poincar\'e inequality $P(C,r)$, if for any function $f:H\to\R$, 
$\min_a\int_H(f-a)^2 d\mu \leq C\cdot r^2\cdot \int_{E(H)} |\nabla f|^2 d\nu$. 
The well known sufficient condition for $P(C,r)$ is the following isoperimetric inequality for subsets of $H$ (see, e.g., \cite[Proposition~3.3.10]{Kumagai} or \cite[Lemma~3.3.7]{SC97}): 
there exists $c>0$ such that for all $A\subset H$ with $|A|\leq \frac12|H|$, the number of edges between $A$ and $H\setminus A$ is at least $\frac cr|A|$. 
Thus, if the ball $\ballG(y,r)$ is contained in a subset $\mathcal C(y,r)$ of $V(G)$ such that $\mathcal C(y,r)\subseteq \ballG(y,C'r)$ and 
the above isoperimetric inequality holds for subsets of $\mathcal C(y,r)$, then it is easy to see that 
the weak Poincar\'e inequality with constants $C$ and $C'$ holds for $\ballG(y,r)$ (see Claim~\ref{cl:rb-gb}). 
In the case $G=\set_\infty\subset \Z^d$, the natural choice is to take $\mathcal C(y,r)$ to be the cluster of $y$ in $\set_\infty\cap\ballZ(y,r)$, 
which turns out to be also the largest cluster in $\set\cap\ballZ(y,r)$ (here and below, we implicitly assume that $r$ is large enough). 
In the setting of Bernoulli percolation, it is known that subsets of $\mathcal C(y,r)$ satisfy the above isoperimetric inequality (see \cite[Proposition~2.11]{Barlow}).
In our setting, Theorem~\ref{thm:chd} implies that $\mathcal C(y,r)\subseteq\ballG(y,C'r)$, thus we only need to prove the isoperimetric inequality. 
The first isoperimetric inequality for subsets of $\mathcal C(y,r)$ was proved in \cite[Theorem~1.2]{PRS}. 
It states that for any $A\subset \mathcal C(y,r)$ with $|A|\geq r^{\delta}$, the number of edges between $A$ and $\set_\infty\setminus A$ is at least $c|A|^{\frac{d-1}{d}}$ 
(thus, also at least $\frac{c'}{r}|A|$).
Note the key difference, the edges are taken between $A$ and $\set_\infty\setminus A$, not just between $A$ and $\mathcal C(y,r)\setminus A$. 
The above isoperimetric inequality implies certain Nash-type inequalities sufficient to prove a diffusive upper bound on the heat kernel 
(see \cite[Theorem~2]{MorrisPeres}, \cite[Proposition~6.1]{BiskupPrescott}, \cite[Lemma~3.2]{BBHK}, \cite[(A.4)]{PRS}), 
but it is too weak to imply the Poincar\'e inequality (see, e.g., \cite[Sections~3.2 and 3.3]{Kumagai} for an overview of the two isoperimetric inequalities and 
their relation to various functional inequalities). 
Let us also mention that in the setting of Bernoulli percolation, the ``weak'' isoperimetric inequality admits a simple proof (\cite[Theorem~A.1]{BBHK}), 
but the proof of the ``strong'' one is significantly more involved (\cite[Proposition~2.11]{Barlow}). 
After all said, we have to admit that we are not able to prove the strong isoperimetric inequality for subsets of $\mathcal C(y,r)$, 
and do not know if it holds in our setting. 
Nevertheless, we can rescue the situation by proving that a certain {\it enlarged} set $\widetilde {\mathcal C}(y,r)$, 
obtained from $\mathcal C(y,r)$ by adding to it all vertices from $\set_\infty$ to which it is {\it locally} connected, 
satisfies the desired strong isoperimetric inequality (see Proposition~\ref{prop:verygoodbox}, Theorem~\ref{thm:isop:cemax}, and Corollaries~\ref{cor:isop:cemax} and \ref{cor:tildeC:existence}).
The general outline of the proof of our isoperimetric inequality for $\widetilde{\mathcal C}(y,r)$ is similar to the one 
of the proof of the weak isoperimetric inequality for $\mathcal C(y,r)$ in \cite{PRS}, 
but we have to modify renormalization and coarse graining of subsets of $\widetilde{\mathcal C}(y,r)$ and rework some arguments 
to get good control of the boundary and the volume of subsets of $\widetilde{\mathcal C}(y,r)$ in terms of the boundary and the volume of 
the corresponding coarse grainings. For instance, it is crucial for us (but not for \cite{PRS}) that the coarse graining of a big set (say, of size $\frac12 |\widetilde{\mathcal C}(y,r)|$) 
should not be too big (see, e.g., the proof of Claim~\ref{cl:isopmain2}). 

We partition the lattice $\Z^d$ into large boxes of equal size. For each configuration $\omega\in\Omega$, we 
subdivide all the boxes into {\it good} and {\it bad}. Restriction of $\set$ to a good box contains a unique largest {\it in volume} cluster, 
and the largest clusters in two adjacent good boxes are connected in $\set$ in the union of the two boxes. 
Traditionally in the study of Bernoulli percolation, the good boxes 
are defined to contain a unique cluster of large diameter. 
In our case, the existence of several clusters of large diameter in good boxes is not excluded. 
The reason to work with volumes is that the existence of a unique giant cluster in a box can be expressed 
as an intersection of two events, an increasing (existence of cluster with big volume) and decreasing (smallness of the total volume of large clusters). 
Assumption \ppp{} gives us control of correlations between monotone events, which is sufficient to set up two multi-scale renormalization schemes with scales $L_n$
(one for increasing and one for decreasing events) and conclude that bad boxes tend to organize in blobs on multiple scales, 
so that the majority of boxes of size $L_n$ contain at most $2$ blobs of diameter bigger than $L_{n-1}$ each, but even their diameters are much smaller than 
the actual scale $L_n$. By removing two boxes of size $r_{n-1}L_{n-1}\ll L_n$ containing the biggest blobs of an $L_n$-box, 
then by removing from each of the remaining $L_{n-1}$-boxes two boxes of size $r_{n-2}L_{n-2}\ll L_{n-1}$ containing its biggest blobs, 
and so on, we end up with a subset of good boxes, 
which is a dense in $\Z^d$, locally well connected, and well structured coarse graining of $\set_\infty$. 
Similar renormalization has been used in \cite{RS:Disordered,DRS12,PRS}.
By reworking some arguments from \cite{PRS}, 
we prove that large subsets of the restriction of the coarse graining to any large box satisfy a $d$-dimensional isoperimetric inequality, 
if the scales $L_n$ grow sufficiently fast (Theorem~\ref{thm:isop:pl}). 
We deduce from it the desired isoperimetric inequality for large subsets $A$ of $\widetilde{\mathcal C}(y,r)$ (Theorem~\ref{thm:isop:cemax}) as follows. 
If $A$ is spread out in $\widetilde{\mathcal C}(y,r)$, then it has large boundary, otherwise, 
we associate with it a set of those good boxes from the coarse graining, the unique largest cluster of which is entirely contained in $A$. 
It turns out that the boundary and the volume of the resulting set are comparable with those of $A$. Moreover, if $|A|\leq \frac 12|\widetilde{\mathcal C}(y,r)|$, 
then the volume of its coarse graining is also only a fraction of the total volume of the coarse graining of $\widetilde{\mathcal C}(y,r)$. 
The isoperimetric inequality then follows from the one for subsets of the coarse graining.

\subsection{Structure of the paper}

In Section~\ref{sec:perforatedlattices} we define perforated sublattices of $\Z^d$ and state an isoperimetric inequality 
for subsets of perforations. The main definition there is \eqref{def:pl}, and the main result is Theorem~\ref{thm:isop:pl}.
The proof of Theorem~\ref{thm:isop:pl} is given in Section~\ref{sec:pl:isop:proof}.  
In Section~\ref{sec:propertiesofclusters} we define a coarse graining of $\set_\infty$ 
and 
study certain extensions of largest clusters of $\set_\infty$ in boxes (Definition~\ref{def:cemax}). 
Particularly, we prove that they satisfy the desired isoperimetric inequality (Theorem~\ref{thm:isop:cemax}) and  
the volume growth (Corollary~\ref{cor:chemdist}). 
In Section~\ref{sec:proof:main} we introduce the notions of {\it regular} and {\it very regular} balls, so that 
a (very) regular ball is always (very) good, and use it to prove the main result of the paper. 
In fact, in Proposition~\ref{prop:verygoodbox} we prove that large balls are very likely to be very regular, which is stronger than Theorem~\ref{thm:vgb:main}. 
In Section~\ref{sec:proofs}, we sketch the proofs of Theorems~\ref{thm:hk:grad} -- \ref{thm:gf:asymp}. 

\bigskip

Finally, let us make a convention about constants. 
As already said, we omit from the notation dependence of constants on $a$, $b$, and $d$.
We usually also omit the dependence on $\epsP$, $\constP$, and $\constS$.
Dependence on other parameters is reflected in the notation, for example, as $c(u,\vgb)$.
Sometimes we use $C$, $C'$, $c$, etc., to denote ``intermediate'' constants, their values may change 
from line to line, and even within a line.

\section{Perforated lattices}\label{sec:perforatedlattices}

In this section we define lattices perforated on multiple scales and study their isoperimetric properties. 
Informally, for a sequence of scales $L_n = l_{n-1}\cdot L_{n-1}$, we define a perforation of the box $[0,L_n)^d$ 
by removing small rectangular regions of $L_{n-1}$-boxes from it, then removing small rectangular regions of $L_{n-2}$-boxes from 
each of the remaining $L_{n-1}$-boxes, and so on down to scale $L_0$. The precise definition is given in \eqref{def:pl}. 
Such perforated lattices will be used in Section~\ref{sec:propertiesofclusters} as coarse approximations of largest connected components of $\set$ in boxes.
The main result of this section is an isoperimetric inequality for subsets of perforations, see Theorem~\ref{thm:isop:pl}.

The rules for perforation (the shape and location of removed regions) are determined by certain cascading events, which we define first, see \eqref{def:seedcascade} and Definition~\ref{def:dinbad}. 
The recursive construction of the perforated lattice is given in Section~\ref{sec:perforatedlattice:construction}, where the main definition is \eqref{def:pl}. 

\medskip

Let $l_n,r_n,L_n$, $n\geq 0$ be sequences of positive integers such that $l_n>r_n$ and $L_n = l_{n-1}\cdot L_{n-1}$, for $n \geq 1$. 
To each $L_n$ we associate the rescaled lattice
\[
\GG_n = L_n\cdot \Z^d = \left\{L_n\cdot x ~:~ x\in\Z^d\right\} ,\
\]
with edges between any pair of ($\ell^1$-)nearest neighbor vertices of $\GG_n$.

%

\subsection{Cascading events}

Let $\seed=(\seed_{x,L_0}~:~L_0\geq 1,x\in\GG_0)$ be a family of events from some sigma-algebra. 
For each $L_0\geq 1, n\geq 0$, $x\in\GG_n$, define recursively the events $\seedcascade_{x,n,L_0}(\seed)$ by 
$\seedcascade_{x,0,L_0} (\seed) = \seed_{x,L_0}$ and 
\begin{equation}\label{def:seedcascade}
\seedcascade_{x,n,L_0}(\seed)= 
\bigcup_{\begin{array}{c}\scriptscriptstyle{x_1,x_2\in \GG_{n-1}\cap(x + [0,L_n)^d)} \\ \scriptscriptstyle{|x_1-x_2|_\infty \geq r_{n-1} \cdot L_{n-1}}\end{array}}
\seedcascade_{x_1,n-1,L_0}(\seed) \cap \seedcascade_{x_2,n-1,L_0}(\seed) ~.
\end{equation}
The events in \eqref{def:seedcascade} also depend on the scales $l_n$ and $r_n$, 
but we omit this dependence from the notation, since these sequences will be properly chosen and fixed later.

\begin{definition}\label{def:dinbad}
Given sequences $l_n, r_n, L_n$, $n\geq 0$, as above, and two families of events $\seedde$ and $\seedin$, 
we say that for $n\geq 0$, $x\in\GG_n$ is $(\seedde,\seedin,n)$-{\it bad} (resp., $(\seedde,\seedin,n)$-{\it good}), if the event $\seedcascade_{x,n,L_0}(\seedde)\cup \seedcascade_{x,n,L_0}(\seedin)$ occurs 
(resp., does not occur).
\end{definition}

Good vertices give rise to certain geometrical structures on $\Z^d$ (perforated lattices), which we define in the next subsection.

\medskip

The choice of the families $\seedde$ and $\seedin$ throughout the paper is either irrelevant for the result (as in Sections~\ref{sec:perforatedlattices} and \ref{sec:pl:isop:proof}) 
or fixed (as in Section~\ref{sec:localevents}). 
Thus, from now on we write $n$-bad (resp., $n$-good) instead of $(\seedde,\seedin,n)$-{\it bad} (resp., $(\seedde,\seedin,n)$-{\it good}), hopefully without causing any confusions.

\begin{remark}
Definition~\ref{def:dinbad} can be naturally generalized to $k$ families of events $\seed_1,\dots,\seed_k$, for any fixed $k$, 
and all the results of Sections~\ref{sec:perforatedlattices} and \ref{sec:pl:isop:proof} still hold (with suitable changes of constants). 
For our applications, it suffices to consider only two families of events (see Section~\ref{sec:localevents}). 
Thus, for simplicity of notation, we restrict to this special case. 
\end{remark}

\subsection{Recursive construction}\label{sec:perforatedlattice:construction}

Throughout this subsection, we fix sequences $l_n, r_n, L_n$, $n\geq 0$, such that $l_n>8r_n$ and $l_n$ is divisible by $r_n$ for all $n$. 
We also fix two local families of events $\seedde$ and $\seedin$, and integers $s\geq 0$ and $K\geq 1$. 
Recall Definition~\ref{def:dinbad} of $n$-good vertices in $\GG_n$. 
For $x\in\Z^d$, define
\begin{equation}\label{def:QKs}
Q_{K,s}(x) = x + \Z^d\cap [0,KL_s)^d ~,
\end{equation}
and write $Q_{K,s}$ for $Q_{K,s}(0)$. 
We also fix $x_s\in\GG_s$ and assume that 
\begin{equation}\label{eq:assumptionframe}
\text{all the vertices in $\GG_s\cap Q_{K,s}(x_s)$ are $s$-good}.
\end{equation}
Our aim is to construct a subset of $0$-good vertices in the lattice box $\GG_0\cap Q_{K,s}(x_s)$  
by recursively perforating it on scales $L_s, L_{s-1},\ldots, L_1$. 
We use Definition~\ref{def:dinbad} to determine the rules of perforation on each scale. 


\medskip

We first recursively define certain subsets of $i$-good vertices in $\GG_i\cap Q_{K,s}(x_s)$ for $i\leq s$, see \eqref{def:GKss} and \eqref{def:GKsi-1}. Let
\begin{equation}\label{def:GKss}
\mathcal G_{K,s,s}(x_s) = \GG_s\cap Q_{K,s}(x_s).
\end{equation}
By \eqref{eq:assumptionframe}, all $z_s\in\mathcal G_{K,s,s}(x_s)$ are $s$-good.

Assume that $\mathcal G_{K,s,i}(x_s)\subset \GG_i$ 
is defined for some $i\leq s$ so that all $z_i\in\mathcal G_{K,s,i}(x_s)$ are $i$-good. 
By Definition~\ref{def:dinbad}, for each $z_i\in\mathcal G_{K,s,i}(x_s)$, there exist 
\[
a_{z_i},b_{z_i}\in (r_{i-1}L_{i-1})\cdot \Z^d\cap(z_i + [0,L_i)^d)
\]
such that the boxes $(a_{z_i}+[0,2r_{i-1}L_{i-1})^d)$ and $(b_{z_i}+[0,2r_{i-1}L_{i-1})^d)$ are contained in $(z_i + [0,L_i)^d)$, and 
all the vertices in 
\[
\left(\GG_{i-1}\cap(z_i + [0,L_i)^d)\right)\setminus\left(\left(a_{z_i}+[0,2r_{i-1}L_{i-1})^d\right)\cup\left(b_{z_i}+[0,2r_{i-1}L_{i-1})^d\right)\right)
\]
are $(i-1)$-good. If the choice is not unique, we choose the pair arbitrarily. 
All the results below hold for any allowed choice of $a_{z_i}$ and $b_{z_i}$. 
To save notation, we will not mention it in the statements.

Define $\mathcal R_{z_i}\subseteq \GG_{i-1}$ to be 
\begin{itemize}
\item[(a)]
the empty set, if all the vertices in $\GG_{i-1}\cap(z_i + [0,L_i)^d)$ are $(i-1)$-good, or
\item[(b)]
$\GG_{i-1}\cap((a_{z_i}+[0,2r_{i-1}L_{i-1})^d)\cup(b_{z_i}+[0,2r_{i-1}L_{i-1})^d))$ 
if $|a_{z_i} - b_{z_i}|_\infty> 2r_{i-1}L_{i-1}$, or 
\item[(c)]
a box $\GG_{i-1}\cap(c_{z_i} + [4r_{i-1}L_{i-1})^d)$ in $\GG_{i-1}\cap (z_i + [0,L_i)^d)$, with $c_{z_i}\in (r_{i-1}L_{i-1})\cdot \Z^d$, which contains 
$\GG_{i-1}\cap((a_{z_i}+[0,2r_{i-1}L_{i-1})^d)\cup(b_{z_i}+[0,2r_{i-1}L_{i-1})^d))$.
\end{itemize}
Possible outcomes (b) and (c) of $\mathcal R_{z_i}$ are illustrated on Figure~\ref{fig:Rzi}.

\begin{figure}[!tp]
\centering
\resizebox{15cm}{!}{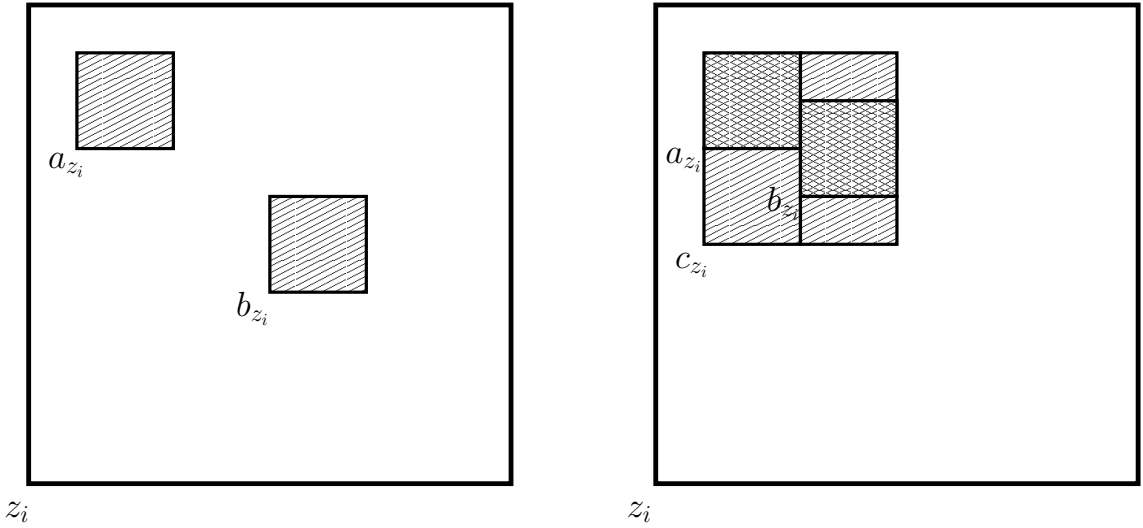}
\caption{Two possible outcomes of $\mathcal R_{z_i}$. On the left, the points $a_{z_i}$ and $b_{z_i}$ are far from each other, 
on the right, they are close.}
\label{fig:Rzi}
\end{figure}

\begin{remark}\label{rem:propRzi}
By construction, the set $\mathcal R_{z_i}$ is the disjoint union of $0$, $2$, or $2^d$ boxes $\GG_{i-1}\cap(x+[0,2r_{i-1}L_{i-1})^d)$ with 
$x\in (r_{i-1}L_{i-1})\cdot \Z^d$.
\end{remark}
To complete the construction, let
\begin{equation}\label{def:GKsi-1}
\mathcal G_{K,s,i-1}(x_s) = \GG_{i-1}\cap\bigcup_{z_i\in\mathcal G_{K,s,i}(x_s)}\left((z_i+[0,L_i)^d)\setminus\mathcal R_{z_i}\right).
\end{equation}
Note that all $z_{i-1}\in\mathcal G_{K,s,i-1}(x_s)$ are $(i-1)$-good.

\medskip

Now that the sets $(\mathcal G_{K,s,j}(x_s))_{j\leq s}$, are constructed by \eqref{def:GKss} and \eqref{def:GKsi-1}, 
we define the multiscale perforations of $\GG_0\cap Q_{K,s}(x_s)$ by 
\begin{equation}\label{def:pl}
\pl_{K,s,j}(x_s) =\GG_0\cap\bigcup_{z_j\in\mathcal G_{K,s,j}(x_s)}(z_j + [0,L_j)^d), \quad j\leq s.
\end{equation}
See Figure~\ref{fig:QKs0} for an illustration.
\begin{figure}[!tp]
\centering
\resizebox{16cm}{!}{\includegraphics{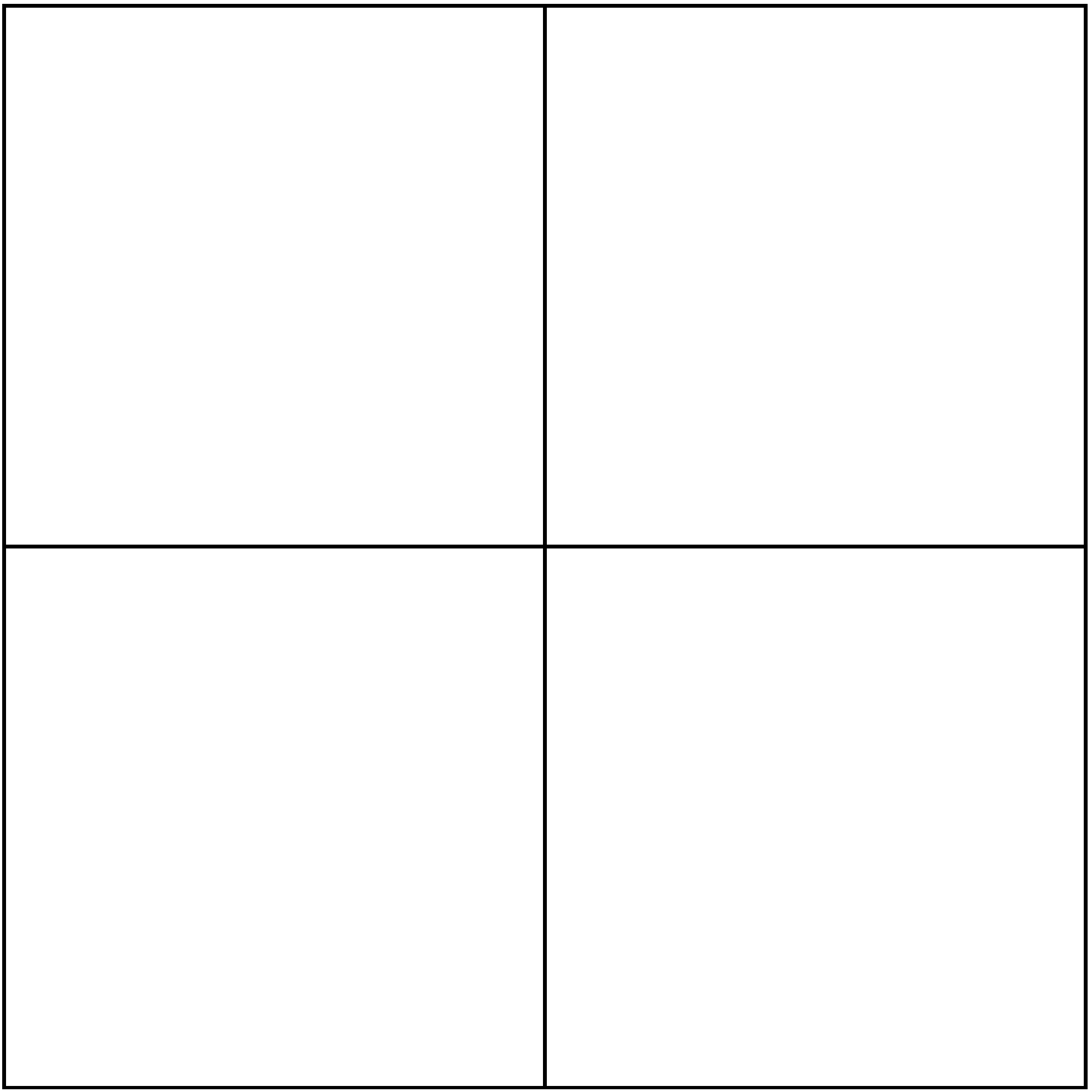}~\includegraphics{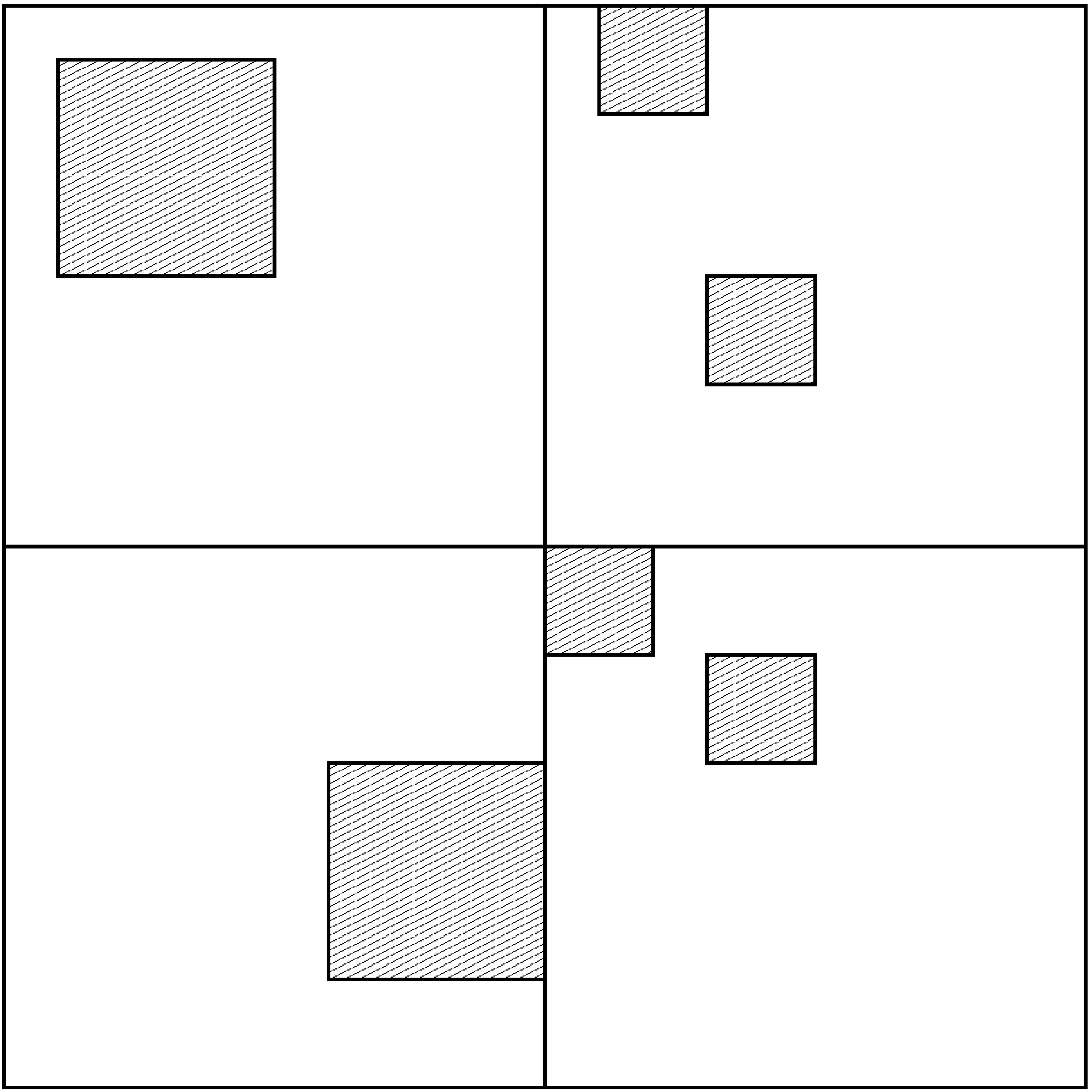}~\includegraphics{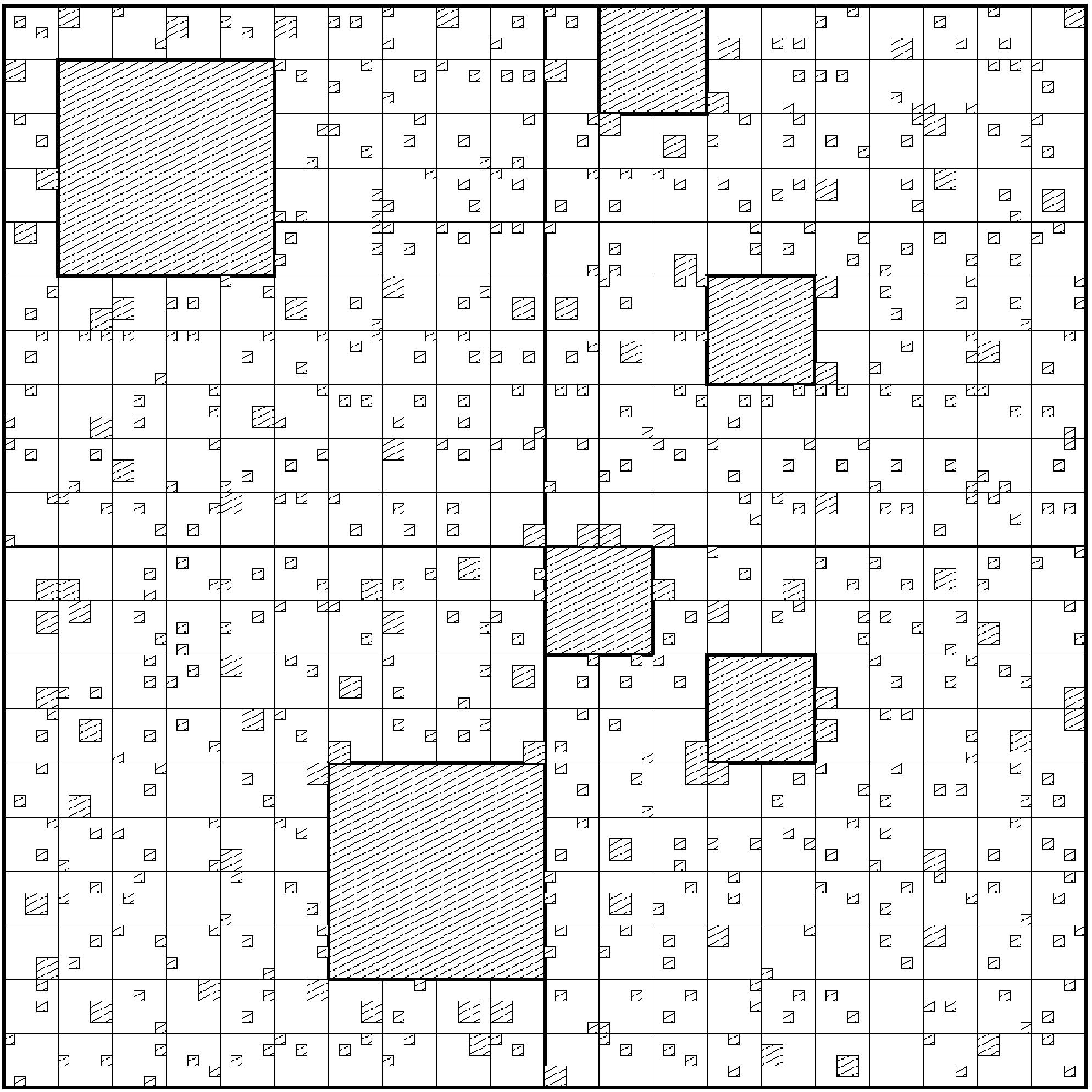}}
\caption{Perforations $\pl_{2,s,s}$, $\pl_{2,s,s-1}$, and $\pl_{2,s,s-2}$ of $Q_{2,s}$.}
\label{fig:QKs0}
\end{figure}
By construction, 
\begin{itemize}\itemsep0pt
\item[(a)]
for all $j$, $\pl_{K,s,j-1}(x_s)\subseteq\pl_{K,s,j}(x_s)$,
\item[(b)]
all the vertices of $\pl_{K,s,0}(x_s)$ are $0$-good.
\end{itemize}
We will view the sets $\pl_{K,s,j}(x_s)$ as subgraphs of $\GG_0$ with edges drawn between any two vertices of the set which are at $\ell^1$ distance $L_0$ from each other. 
The next lemma summarizes some basic properties of $\pl_{K,s,0}(x_s)$'s, which are immediate from the construction. 
\begin{lemma}\label{l:rect}
Let $d\geq 2$, $K\geq 1$, and $s\geq 0$. 
For any choice of scales $l_n, r_n, L_n$, $n\geq 0$, such that $l_n>8r_n$ and $l_n$ is divisible by $r_n$ for all $n$, 
and for any admissible choice of $a_{z_i},b_{z_i}$ or $c_{z_i}$ in the construction of $\pl_{K,s,0}(x_s)$, 
\begin{itemize}\itemsep0pt
\item[(a)]
$\pl_{K,s,0}(x_s)$ is connected in $\GG_0$, 
\item[(b)]
$|\pl_{K,s,0}(x_s)| \geq \prod_{j=0}^\infty\left(1 - \left(\frac{4r_j}{l_j}\right)^d\right)\cdot |Q_{K,s}|$.
\end{itemize}
\end{lemma}

\subsection{Isoperimetric inequality}\label{sec:pl:isopineq}

For a graph $G$ and a subset $A$ of $G$, the boundary of $A$ in $G$ is the subset of edges of $G$, $E(G)$, defined as 
\[
\partial_G A = \{\{x,y\}\in E(G)~:~x\in A,~y\in G\setminus A\}.
\]

\medskip

The next theorem states that under assumption \eqref{eq:assumptionframe} 
and some assumptions on $l_n$ and $r_n$ (basically that $\sum_{n\geq 0}\frac{r_n}{l_n}$ is sufficiently small), 
there exist $\gamma>0$ such that for all large enough $A\subset \pl_{K,s,0}(x_s)$ with $|A|\leq \frac12\cdot |\pl_{K,s,0}(x_s)|$, 
$|\partial_{\pl_{K,s,0}(x_s)}A| \geq \gamma\cdot |A|^{\frac{d-1}{d}}$. 

\medskip

\begin{theorem}\label{thm:isop:pl}
Let $d\geq 2$.
Let $l_n$ and $r_n$, $n\geq 0$, be integer sequences such that for all $n$, $l_n>8r_n$, $l_n$ is divisible by $r_n$, and 
\begin{equation}\label{eq:isop:pl}
\prod_{j=0}^\infty \left(1 - \left(\frac{4r_j}{l_j}\right)^2\right) \geq 
\max\left\{\frac{15}{16}, e^{-\frac{1}{16(d-1)}}, \frac{1 - \frac{1}{2^{d+2}}}{1 - \frac{1}{2^{d+3}}}\right\}
\quad\text{and}\quad
3456\cdot \sum_{j=0}^\infty\frac{r_j}{l_j}\leq \frac{1}{10^6}. 
\end{equation}
Then for any integers $s\geq 0$, $L_0\geq 1$, and $K\geq 1$, $x_s\in\GG_s$, 
and two families of events $\seedde$ and $\seedin$,  
if all the vertices in $\GG_s\cap Q_{K,s}(x_s)$ are $s$-good, then
any $\mathcal A\subseteq \pl_{K,s,0}(x_s)$ with
\[
\left(\frac{L_s}{L_0}\right)^{d^2} \leq |\mathcal A|\leq \frac 12\cdot |Q_{K,s}\cap \GG_0|
\]
satisfies
\[
|\partial_{\pl_{K,s,0}(x_s)} \mathcal A|\geq 
\frac{1}{2d\cdot 32^d\cdot 27^d\cdot 10^6}\cdot \left(1 - \left(\frac 23\right)^{\frac 1d}\right)\cdot \left(1 - e^{-\frac{1}{16(d-1)}}\right)\cdot |\mathcal A|^{\frac{d-1}{d}}.
\]
\end{theorem}
\begin{remark}\label{rem:isop:pl}
In the setting of Theorem~\ref{thm:isop:pl}, if $\mathcal A\subseteq \pl_{K,s,0}(x_s)$ satisfies 
\[
\left(\frac{L_s}{L_0}\right)^{d^2} \leq |\mathcal A|\leq C\cdot |\pl_{K,s,0}(x_s)|, 
\]
for some $\frac12<C<1$, then 
\[
|\partial_{\pl_{K,s,0}(x_s)} \mathcal A|\geq 
(1-C)\cdot \frac{1}{2d\cdot 32^d\cdot 27^d\cdot 10^6}\cdot \left(1 - \left(\frac 23\right)^{\frac 1d}\right)\cdot \left(1 - e^{-\frac{1}{16(d-1)}}\right)\cdot |\mathcal A|^{\frac{d-1}{d}}.
\]
This easily follows from Theorem~\ref{thm:isop:pl} by passing, if necessary, to the complement of $\mathcal A$ in $\pl_{K,s,0}(x_s)$, 
see, for instance, Remark~\ref{rem:isop}.
\end{remark}

\medskip

We postpone the proof of Theorem~\ref{thm:isop:pl} to Section~\ref{sec:pl:isop:proof}. 
In fact, in two dimensions, we are able to prove the analogue of Theorem~\ref{thm:isop:pl} for {\it all} subsets $\mathcal A\subseteq \pl_{K,s,0}(x_s)$ with
$1\leq |\mathcal A|\leq \frac 12\cdot |Q_{K,s}\cap \GG_0|$, see Lemma~\ref{l:isopineqG:2d}. 
We believe that also in any dimension $d\geq 3$, the isoperimetric inequality of Theorem~\ref{thm:isop:pl} holds for all 
subsets $\mathcal A\subseteq \pl_{K,s,0}(x_s)$ with $1\leq |\mathcal A|\leq \frac 12\cdot |Q_{K,s}\cap \GG_0|$, 
but cannot prove it. Theorem~\ref{thm:isop:pl} follows immediately from a more general isoperimetric inequality in Theorem~\ref{thm:isopframe}.

\bigskip

\section{Properties of the largest clusters}\label{sec:propertiesofclusters}

In this section we study properties of the largest subset of $\set\cap Q_{K,s}$ (where $Q_{K,s}$ is defined in \eqref{def:QKs}). 
We first define two families of events such that the corresponding perforated lattices defined in \eqref{def:pl} 
serve as a ``skeleton'' of the largest subset of $\set\cap Q_{K,s}$. 
Then, we provide sufficient conditions for the uniqueness of the largest subset of $\set\cap Q_{K,s}$ (Lemma~\ref{l:cmax:uniqueness}).
To avoid problems, which may be caused by roughness of the boundary of the largest subset of $\set\cap Q_{K,s}$, 
we {\it enlarge} it by adding to it all the points of $\set$ which are {\it locally} connected to it (Definition~\ref{def:cemax}). 
For the enlarged set we prove under some general conditions (Definition~\ref{def:mainevent}) that 
its subsets satisfy an isoperimetric inequality (Theorem~\ref{thm:isop:cemax} and Corollary~\ref{cor:isop:cemax}). 
Under the same condition we prove that the graph distance is controlled by that on $\Z^d$ (Lemma~\ref{l:chemdist}), 
large enough balls have regular volume growth (Corollary~\ref{cor:chemdist}) and 
have local extensions satisfying an isoperimetric inequality (Corollary~\ref{cor:tildeC:existence}).

\medskip

\subsection{Special sequences of events}\label{sec:localevents}

Recall Definition~\ref{def:setr} of $\set_r$. Consider an ordered pair of real numbers
\begin{equation}\label{def:eta}
\den = (\den_1,\den_2),\quad \text{with }\den_1\in(0,1)\text{ and }\den_1\leq \den_2<2\den_1.
\end{equation}
Two families of events $\seedde^\den = (\seedde^\den_{x,L_0}, L_0\geq 1, x\in \GG_0)$ and $\seedin^\den = (\seedin^\den_{x,L_0}, L_0\geq 1, x\in \GG_0)$ are defined as follows. 
\begin{itemize}
\item
The complement of $\seedde^\den_{x,L_0}$ is the event that 
for each $y\in\GG_0$ with $|y-x|_1 \leq L_0$, the set $\set_{L_0}\cap(y+[0,L_0)^d)$ contains a connected component $\mathcal C_y$ with at least $\den_1 L_0^d$ vertices 
such that for all $y\in\GG_0$ with $|y-x|_1 \leq L_0$, $\mathcal C_y$ and $\mathcal C_x$ are connected in $\set \cap((x+[0,L_0)^d)\cup(y+[0,L_0)^d))$. 
\item
The event $\seedin^\den_{x,L_0}$ occurs if 
$\left|\set_{L_0}\cap(x+[0,L_0)^d)\right| > \den_2 L_0^d$. 
\end{itemize}
Note that $\seedde^\den_{x,L_0}$ are decreasing and $\seedin^\den_{x,L_0}$ increasing events. 
From now on we fix these two local families, and say that $x\in\GG_n$ is $n$-bad / $n$-good, if 
it is $n$-bad / $n$-good for the two local families $\seedde^\den$ and $\seedin^\den$ in the sense of Definition~\ref{def:dinbad}. 
In particular, $x\in\GG_0$ is $0$-good if both $\seedde^\den_{x,L_0}$ and $\seedin^\den_{x,L_0}$ do not occur, see Figure~\ref{fig:0good}. 

\begin{figure}[!tp]
\centering
\resizebox{6cm}{!}{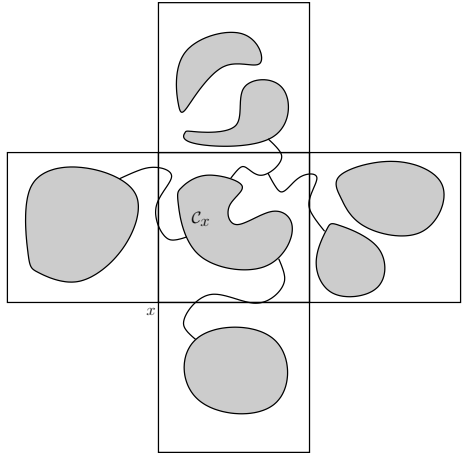}
\caption{A $0$-good vertex $x$. A unique connected component $\mathcal C_x$ of size $\geq \den_1 L_0^d$ in $(x + [0,L_0)^d)$ 
is connected to a connected component of size $\geq \den_1 L_0^d$ in each of the adjacent boxes.}
\label{fig:0good}
\end{figure}

The following lemma is immediate from the definition of $0$-good vertex and the conditions \eqref{def:eta} on $\den$.
(See, e.g., \cite[Lemma~5.2]{DRS12} for a similar result.)
\begin{lemma}\label{l:fromG0toZd}
Let $L_0\geq 1$ and $\den$ as in \eqref{def:eta}.
\begin{itemize}\itemsep0pt
\item[(a)]
For any $0$-good vertex $x\in\GG_0$, connected component $\mathcal C_x$ in $\set_{L_0}\cap(x+[0,L_0)^d)$ 
with at least $\den_1 L_0^d$ vertices is defined uniquely.
\item[(b)]
For any $0$-good $x,y\in\GG_0$ with $|x-y|_1 = L_0$, 
(uniquely chosen) $\mathcal C_x$ and $\mathcal C_y$ are connected in the graph $\set\cap ((x+[0,L_0)^d)\cup(y+[0,L_0)^d))$. 
\end{itemize}
\end{lemma}

\subsection{Uniqueness of the largest cluster}

\begin{definition}\label{def:cmax}
Let $(L_n)_{n\geq 0}$ be an increasing sequence of scales. 
For $x\in \Z^d$ and $r\geq 1$, let $\cmax_{K,s,r}(x)$ be the largest connected component in $\set_r\cap Q_{K,s}(x)$ (with ties broken arbitrarily), 
and write $\cmax_{K,s,r} = \cmax_{K,s,r}(0)$. 
\end{definition}
Fix $\den$ as in \eqref{def:eta} and two families of events $\seedde^\den$ and $\seedin^\den$ as in Section~\ref{sec:localevents}. 
\begin{lemma}\label{l:cmax:uniqueness}
Let $l_n$ and $r_n$ be integer sequences such that for all $n$, $l_n$ is divisible by $r_n$, $l_n > 8r_n$, and 
\begin{equation}\label{eq:cmax:uniqueness:ratio}
\prod_{i=0}^\infty\left(1 - \left(\frac{4r_i}{l_i}\right)^d\right)> \frac{1 + \den_2}{1+2\den_1}.
\end{equation}
Let $L_0\geq 1$, $K\geq 1$, and $s\geq 0$ integers, $x_s\in\GG_s$. 
If all the vertices in $\GG_s\cap Q_{K,s}(x_s)$ are $s$-good, then $\cmax_{K,s,L_0}(x_s)$ is uniquely defined and 
\begin{equation}\label{eq:cmax:uniqueness:size}
|\cmax_{K,s,L_0}(x_s)| \geq \frac12\den_2\cdot |Q_{K,s}|.
\end{equation}
\end{lemma}
\begin{proof}
Without loss of generality we assume that $x_s = 0$. 
Since all vertices in $\GG_s\cap Q_{K,s}$ are $s$-good, we can define the perforation $\pl_{K,s,0}$ by \eqref{def:pl}.
By definition, all the vertices of $\pl_{K,s,0}$ are $0$-good, and by Lemma~\ref{l:rect}, $\pl_{K,s,0}$ is connected in $\GG_0$.

By Lemma~\ref{l:fromG0toZd}, for any $x\in\pl_{K,s,0}$, 
there is a uniquely defined connected subset $\mathcal C_x$ of $\set_{L_0}\cap(x + [0,L_0)^d)$ with at least $\den_1 L_0^d$ vertices. 
Since $\pl_{K,s,0}$ is connected in $\GG_0$, by Lemma~\ref{l:fromG0toZd}, 
the set $\bigcup_{x\in\pl_{K,s,0}} \mathcal C_x$ is contained in a connected component of $\set_{L_0}\cap Q_{K,s}$ and 
\begin{equation}\label{eq:sizecmax1}
\left|\bigcup_{x\in\pl_{K,s,0}} \mathcal C_x\right|
\geq \den_1\cdot|\pl_{K,s,0}|
\geq 
\den_1\cdot  \prod_{i=0}^\infty\left(1 - \left(\frac{4r_i}{l_i}\right)^d\right)\cdot |Q_{K,s}|,
\end{equation}
where the second inequality follows from Lemma~\ref{l:rect}. 

On the other hand, since for any $0$-good vertex $x$, the set $x + [0,L_0)^d$ contains at most $\den_2 L_0^d$ vertices from $\set_{L_0}$,
\begin{eqnarray}
\left|\set_{L_0}\cap Q_{K,s}\right|
&\leq 
&\den_2 L_0^d\cdot |\pl_{K,s,0}| + L_0^d\cdot \left(|Q_{K,s}\cap\GG_0| - |\pl_{K,s,0}|\right)\nonumber\\
&\leq &\left(\den_2 + 1 - \prod_{i=0}^\infty\left(1 - \left(\frac{4r_i}{l_i}\right)^d\right)\right)\cdot |Q_{K,s}|\nonumber\\
&< &2\den_1\cdot  \prod_{i=0}^\infty\left(1 - \left(\frac{4r_i}{l_i}\right)^d\right)\cdot |Q_{K,s}|,\label{eq:sizecmax2}
\end{eqnarray}
where the second inequality follows from 
the inequality $|\pl_{K,s,0}|\leq |\pl_{K,s,s}| = \frac{|Q_{K,s}|}{L_0^d}$ and Lemma~\ref{l:rect}, 
and the third inequality follows from the assumption \eqref{eq:cmax:uniqueness:ratio}.

We have shown that the connected component of $\set_{L_0}\cap Q_{K,s}$ which contains $\bigcup_{x\in\pl_{K,s,0}} \mathcal C_x$ has volume 
$>\frac 12\cdot |\set_{L_0}\cap Q_{K,s}|$. 
In particular, it is the unique largest in volume connected component of $\set_{L_0}\cap Q_{K,s}$. 
Moreover, by \eqref{eq:sizecmax1}, its volume is 
\[
\geq \den_1\cdot  \prod_{i=0}^\infty\left(1 - \left(\frac{4r_i}{l_i}\right)^d\right)\cdot |Q_{K,s}|
\stackrel{\eqref{eq:cmax:uniqueness:ratio}}\geq \den_1\cdot\frac{1 + \den_2}{1+2\den_1}\cdot |Q_{K,s}| 
\stackrel{\eqref{def:eta}}\geq  \frac12\den_2\cdot |Q_{K,s}|,
\]
which proves \eqref{eq:cmax:uniqueness:size}.
%
%
\end{proof}
\begin{corollary}\label{cor:cmax:uniqueness}
From the proof of Lemma~\ref{l:cmax:uniqueness}, if the conditions of Lemma~\ref{l:cmax:uniqueness} are satisfied, then 
\begin{equation}\label{eq:framecmax}
\bigcup_{x\in\pl_{K,s,0}} \mathcal C_x \subseteq \cmax_{K,s,L_0}.
\end{equation}
In particular, for any $1\leq K'\leq K''\leq K$ and $x',x''\in\GG_s\cap Q_{K,s}$ such that $Q_{K',s}(x')\subseteq Q_{K'',s}(x'') \subseteq Q_{K,s}$,
$\cmax_{K',s,L_0}(x') \subseteq \cmax_{K'',s,L_0}(x'')\subseteq \cmax_{K,s,L_0}$.
\end{corollary}

\subsection{Isoperimetric inequality}

In this section we prove an isoperimetric inequality for subsets of a certain extension $\cemax_{K,s,L_0}(x)$ of $\cmax_{K,s,L_0}(x)$ 
obtained by adding to $\cmax_{K,s,L_0}(x)$ all the vertices to which it is locally connected.
\begin{definition}\label{def:cemax}
Let $\emax_{K,s,r}(x)$ be the set of vertices $y'\in\set$ 
such that for some $y\in\cmax_{K,s,r}(x)$, $y'$ is connected to $y$ in $\set\cap\ballZ(y,2L_s)$, and define
\[
\cemax_{K,s,r}(x) = \cmax_{K,s,r}(x)\cup\emax_{K,s,r}(x).
\]
\end{definition}
\begin{remark}
Mind that $\cemax_{K,s,r}(x)$ is contained in $x + [-2L_s,(K+2)L_s)^d$, but it is different from the largest cluster of $\set_r\cap(x + [-2L_s,(K+2)L_s)^d)$.
\end{remark}

We study isoperimetric properties of $\cemax_{K,s,L_0}(x)$ for configurations from the following event. 
\begin{definition}\label{def:mainevent}
Let $\den$ be as in \eqref{def:eta}, $K\geq 1$ and $s\geq 0$ integers, $x_s\in\GG_s$. 
The event $\mathcal H^\den_{K,s}(x_s)\in\mathcal F$ occurs if 
\begin{itemize}\itemsep0pt
\item[(a)]
all the vertices in $\GG_s\cap (x_s + [-2L_s,(K+2)L_s)^d)$ are $s$-good, 
\item[(b)]
any $x,y\in\set_{L_s}\cap Q_{K,s}(x_s)$ with $|x-y|_\infty\leq L_s$ 
are connected in $\set\cap\ballZ(x,2L_s)$. 
\end{itemize}
We write $\mathcal H^\den_{K,s}$ for $\mathcal H^\den_{K,s}(0)$. 
\end{definition}

Here is the main result of this section.

\begin{theorem}\label{thm:isop:cemax}
Let $\den$ be as in \eqref{def:eta}.
Assume that the sequences $l_n$ and $r_n$ satisfy the conditions of Theorem~\ref{thm:isop:pl} and 
\begin{equation}\label{eq:isop:cemax:ratio}
\prod_{i=0}^\infty\left(1 - \left(\frac{4r_i}{l_i}\right)^d\right)\geq \frac{1 + \den_2}{1+\frac{\den_2 + 2\den_1}{2}}.
\end{equation}
Let $L_0\geq 1$, $K\geq 1$, and $s\geq 0$ integers, $x_s\in\GG_s$. 
If $\mathcal H^\den_{K,s}(x_s)$ occurs, then $\cmax_{K,s,L_0}(x_s)$ is uniquely defined and 
there exists $\gamma_{\scriptscriptstyle \ref{thm:isop:cemax}} = \gamma_{\scriptscriptstyle \ref{thm:isop:cemax}}(\den,L_0)\in(0,1)$ 
such that 
for any $A\subseteq\cemax_{K,s,L_0}(x_s)$ with $L_s^{d(d+1)}\leq |A|\leq \frac 12\cdot |\cemax_{K,s,L_0}(x_s)|$, 
\[
|\partial_{\cemax_{K,s,L_0}(x_s)}A| \geq \gamma_{\scriptscriptstyle \ref{thm:isop:cemax}}\cdot |A|^{\frac{d-1}{d}}.
\]
\end{theorem}

\medskip

In the applications, we will not use directly the result of Theorem~\ref{thm:isop:cemax}, but only the following corollary, 
which estimates from below the size of the boundary of {\it any} subset of $\cemax_{K,s,L_0}(x_s)$ with volume $\leq \frac 12\cdot |\cemax_{K,s,L_0}(x_s)|$.

\begin{corollary}\label{cor:isop:cemax}
Let $\den$ be as in \eqref{def:eta} and $\epsilon\in(0,\frac1d]$. 
Assume that the sequences $l_n$ and $r_n$ satisfy the conditions of Theorem~\ref{thm:isop:cemax}. 
Assume that 
\[
K\geq L_s^{d + \frac{d^2 - 1}{\epsilon d}}. 
\]
If $\mathcal H^\den_{K,s}(x_s)$ occurs, then for any $A\subseteq\cemax_{K,s,L_0}(x_s)$ with $|A|\leq \frac 12\cdot |\cemax_{K,s,L_0}(x_s)|$, 
\[
|\partial_{\cemax_{K,s,L_0}(x_s)}A| \geq \gamma_{\scriptscriptstyle \ref{thm:isop:cemax}}\cdot |A|^{\frac{d-1}{d} + \epsilon} \cdot ((K+4)L_s)^{-\epsilon d}.
\]
In particular, if $\epsilon = \frac1d$, then 
$|\partial_{\cemax_{K,s,L_0}(x_s)}A| \geq \gamma_{\scriptscriptstyle \ref{thm:isop:cemax}}\cdot \frac{|A|}{(K+4)L_s}$.
\end{corollary}
\begin{proof}[Proof of Corollary~\ref{cor:isop:cemax}]
If $|A|\geq L_s^{d(d+1)}$, then we apply Theorem~\ref{thm:isop:cemax}, 
\[
|\partial_{\cemax_{K,s,L_0}(x_s)}A|\geq\gamma_{\scriptscriptstyle \ref{thm:isop:cemax}}\cdot |A|^{\frac{d-1}{d}}
\geq\gamma_{\scriptscriptstyle \ref{thm:isop:cemax}}\cdot |A|^{\frac{d-1}{d} + \epsilon} \cdot ((K+4)L_s)^{-\epsilon d}.
\]
If $|A|\leq L_s^{d(d+1)}$, then we use the trivial bound $|\partial_{\cemax_{K,s,L_0}(x_s)}A|\geq 1$. 
By the assumption on $K$, 
\[
((K+4)L_s)^{\epsilon d} \geq (L_s^{d(d+1)})^{\frac{d-1}{d} + \epsilon}, 
\]
which implies, using the assumption on $|A|$, that $|A|^{\frac{d-1}{d} + \epsilon} \leq ((K+4)L_s)^{\epsilon d}$. 
Thus, in this case, 
\[
|\partial_{\cemax_{K,s,L_0}(x_s)}A|\geq 1
\geq\gamma_{\scriptscriptstyle \ref{thm:isop:cemax}}\cdot |A|^{\frac{d-1}{d} + \epsilon} \cdot ((K+4)L_s)^{-\epsilon d}.
\]
The proof of corollary is complete.
\end{proof}

\bigskip

The proof of Theorem~\ref{thm:isop:cemax} is subdivided into several claims. 
In Claim~\ref{cl:cemax:localconnected} we prove that $\cemax_{K,s,L_0}$ is locally connected and
in Claims~\ref{cl:isopmain} and \ref{cl:isopmain2} we reduce the isoperimetric problem for subsets of $\cemax_{K,s,L_0}$ 
to the one for subsets of a perforated lattice.

\begin{claim}\label{cl:cemax:localconnected}
Any $x,y\in\cemax_{K,s,L_0}$ with $|x-y|_\infty\leq L_s$ are connected in $\cemax_{K,s,L_0}\cap\ballZ(x,15L_s)$.
\end{claim}
\begin{proof}
Fix $x,y\in\cemax_{K,s,L_0}$ with $|x-y|_\infty\leq L_s$, and take $x',y'\in\cmax_{K,s,L_0}$ such that 
$x$ and $x'$ are connected in $\cemax_{K,s,L_0}\cap\ballZ(x',2 L_s)$, $y$ and $y'$ are connected in $\cemax_{K,s,L_0}\cap\ballZ(y',2L_s)$. 
By the triangle inequality, $|x'-y'|_\infty\leq 5L_s$. 

\medskip

Since all the vertices in $\GG_s\cap Q_{K,s}$ are $s$-good, there exist $x'',y''\in\bigcup_{z\in\pl_{K,s,0}}\mathcal C_z$ such that 
$|x' -x''|_\infty\leq L_s$ and $|y'-y''|_\infty\leq L_s$. 
By the definitions of $\mathcal H^\den_{K,s}$ and $\cemax_{K,s,L_0}$, 
$x''$ is connected to $x'$ in $\cemax_{K,s,L_0}\cap\ballZ(x',2L_s)$ and $y''$ is connected to $y'$ in $\cemax_{K,s,L_0}\cap\ballZ(y',2L_s)$. 

\medskip

By the triangle inequality, $|x''-y''|_\infty\leq 7L_s$.
Let $(z + [0,8L_s)^d)$ be a box in $Q_{K,s}$ which contains both $x''$ and $y''$, where $z\in \GG_s$. 
Since all the vertices in $\GG_s\cap (z + [0,8L_s)^d)$ are $s$-good, 
the perforation $\pl_{8,s,0}(z) = \pl_{K,s,0}\cap(z+ [0,8L_s)^d)$ of $(z + [0,8L_s)^d)$ is 
connected in $\GG_0$ by Lemma~\ref{l:rect}. 
Thus, by Lemma~\ref{l:fromG0toZd}, the set $\bigcup_{w\in\pl_{8,s,0}(z)}\mathcal C_w$ 
is contained in a connected component of $\set\cap(z + [0,8L_s)^d)$. 
In particular, $x''$ and $y''$ are connected in $\set\cap(z + [0,8L_s)^d)$. 
By \eqref{eq:framecmax} and the fact that \eqref{eq:isop:cemax:ratio} implies \eqref{eq:cmax:uniqueness:ratio}, 
the set $\bigcup_{w\in\pl_{8,s,0}(z)}\mathcal C_w$ is contained in $\cmax_{K,s,L_0}$. 
Therefore, $x''$ is connected to $y''$ in $\cmax_{K,s,L_0}\cap(z + [0,8L_s)^d)\subset \cmax_{K,s,L_0}\cap\ballZ(x'',8L_s)$. 

\medskip

We conclude that $x$ is connected to $y$ in $\cemax_{K,s,L_0}\cap\ballZ(x,15L_s)$.
\end{proof}

\bigskip

Let 
\[
x_s' = (-2L_s,\dots,-2L_s)\in\GG_s\quad\text{and}\quad K' = K+4.
\]
Since all the vertices in $\GG_s\cap Q_{K',s}(x_s')$ are $s$-good, 
we can define its perforation $\pl_{K',s,0}(x_s')$ as in \eqref{def:pl}. 
By definition, $\pl_{K',s,0}(x_s')$ is a subset of $0$-good vertices in $\GG_0\cap Q_{K',s}(x_s')$, 
and by Lemma~\ref{l:rect}, $\pl_{K',s,0}(x_s')$ is connected in $\GG_0$. 

By the fact that \eqref{eq:isop:cemax:ratio} implies \eqref{eq:cmax:uniqueness:ratio}, 
Lemma~\ref{l:fromG0toZd}, \eqref{eq:framecmax}, and the definition of $\cemax_{K,s,L_0}$, 
\begin{equation}\label{eq:framecemax}
\bigcup_{x\in\pl_{K',s,0}(x_s')} \mathcal C_x \subseteq \cemax_{K,s,L_0}.
\end{equation}

The next two claims allow to reduce the isoperimetric problem for subsets of $\cemax_{K,s,L_0}$ 
to the isoperimetric problem for subsets of $\pl_{K',s,0}(x_s')$. 
The crucual step for the proof is the following definition of $\setg$ and $\setd$. 
\begin{definition}
For $A\subseteq\cemax_{K,s,L_0}$, 
let $\setg$ be the set of all $x\in\pl_{K',s,0}(x_s')$ such that $\mathcal C_x\subseteq A$, and  
$\setd$ the set of $x\in A$ such that there exists $y\in \cemax_{K,s,L_0}\setminus A$ with $|x-y|_\infty\leq L_s$. 
\end{definition}
\begin{claim}\label{cl:isopmain}
\begin{equation}\label{eq:isopmain:boundary}
|\partial_{\cemax_{K,s,L_0}} A| 
\geq \max\left\{\frac{1}{2d}\cdot |\partial_{\pl_{K',s,0}(x_s')}\setg|,\frac{|\setd|}{(31\cdot L_s)^d} \right\} 
\end{equation}
and
\begin{equation}\label{eq:isopmain:volume}
|A| \leq 2\cdot 3^d\cdot L_0^d\cdot|\setg| + |\setd| .\
\end{equation}
\end{claim}
\begin{proof}
We begin with the proof of \eqref{eq:isopmain:boundary}.
For any $x\in\setg$ and $y\in \pl_{K',s,0}(x_s')\setminus \setg$ such that $|x-y|_1 = L_0$, 
$\mathcal C_x\subseteq A$ and $\mathcal C_y \nsubseteq A$.
By Lemma~\ref{l:fromG0toZd} and \eqref{eq:framecemax}, 
$\mathcal C_x$ and $\mathcal C_y$ are connected in $\cemax_{K,s,L_0}\cap((x+[0,L_0)^d)\cup(y+[0,L_0)^d))$. 
Each path in $\cemax_{K,s,L_0}$ connecting $\mathcal C_x$ and $\mathcal C_y\setminus A$ contains an edge from $\partial_{\cemax_{K,s,L_0}} A$. 
This implies that 
\begin{equation}\label{eq:isopmain:boundary:1}
|\partial_{\cemax_{K,s,L_0}} A| \geq \frac{1}{2d}\cdot |\partial_{\pl_{K',s,0}(x_s')}\setg| .
\end{equation}
Next, by the definition of $\setd$, for any $x\in \setd$, there exists $y\in \cemax_{K,s,L_0}\setminus A$ such that $|x-y|_\infty\leq L_s$.
By Claim~\ref{cl:cemax:localconnected}, $x$ and $y$ are connected in $\cemax_{K,s,L_0}\cap\ballZ(x,15L_s)$. 
In particular, the ball $\ballZ(x,15L_s)$ contains an edge from $\partial_{\cemax_{K,s,L_0}} A$. 
Since every edge from $\partial_{\cemax_{K,s,L_0}} A$ is within $\ell^\infty$ distance $15L_s$ from at most $(31L_s)^d$ vertices of $\setd$, 
\begin{equation}\label{eq:isopmain:boundary:2}
|\partial_{\cemax_{K,s,L_0}} A| \geq \frac{|\setd|}{(31\cdot L_s)^d} .
\end{equation}
Inequalities \eqref{eq:isopmain:boundary:1} and \eqref{eq:isopmain:boundary:2} imply \eqref{eq:isopmain:boundary}.

\bigskip

We proceed with the proof of \eqref{eq:isopmain:volume}.
We need to show that 
\begin{equation}\label{eq:aminussetd}
|A\setminus\setd| \leq 2\cdot 3^d\cdot L_0^d\cdot |\setg|.
\end{equation}
Let $z\in A\setminus\setd$. 
By the definition of $\cemax_{K,s,L_0}$, there exists $z_s\in\GG_s\cap Q_{K',s}(x_s')$ such that 
\[
z_s + [0,L_s)^d\subset \ballZ(z,L_s). 
\]
By the definition of $\setd$ and \eqref{eq:framecemax}, for any $x\in\pl_{K',s,0}(x_s')\cap(z_s + [0,L_s)^d)$, $\mathcal C_x\subset A$. 
Thus, $\pl_{K',s,0}(x_s')\cap(z_s + [0,L_s)^d)\subseteq \setg$. 
By Lemma~\ref{l:rect} and \eqref{eq:isop:cemax:ratio}, 
\[
|\pl_{K',s,0}(x_s')\cap (z_s + [0,L_s)^d)| = 
|\pl_{1,s,0}(z_s)| 
\geq \frac{1 + \den_2}{1+\frac{\den_2 + 2\den_1}{2}}\cdot \left(\frac{L_s}{L_0}\right)^d
\geq \frac 12\cdot \left(\frac{L_s}{L_0}\right)^d.
\]
Thus, 
\[
|\setg\cap \ballZ(z,L_s)| \geq \frac 12\cdot \left(\frac{L_s}{L_0}\right)^d, 
\]
and we conclude that 
\[
\frac12\cdot \left(\frac{L_s}{L_0}\right)^d\cdot |A\setminus\setd|\leq \left|\{z\in A\setminus\setd, x\in\setg~:~ x\in\ballZ(z,L_s)\}\right|
\leq |\ballZ(0,L_s)|\cdot |\setg|,
\]
which implies \eqref{eq:aminussetd}.
\end{proof}

\bigskip

Let $\gamma_{\scriptscriptstyle \ref{thm:isop:pl}}$ be the isoperimetric constant from Theorem~\ref{thm:isop:pl}:
\[
\gamma_{\scriptscriptstyle \ref{thm:isop:pl}} = \frac{1}{2d\cdot 32^d\cdot 27^d\cdot 10^6}\cdot \left(1 - \left(\frac 23\right)^{\frac 1d}\right)\cdot \left(1 - e^{-\frac{1}{16(d-1)}}\right).
\]
\begin{claim}\label{cl:isopmain2}
Let $c_\den = \frac{2\den_1 - \den_2}{4\den_1}$. Then
\begin{equation}\label{eq:ineqmax}
\max\left\{|\partial_{\pl_{K',s,0}(x_s')}\setg|,\frac{|\setd|}{L_s^d} \right\}
\geq c_\den\cdot \gamma_{\scriptscriptstyle \ref{thm:isop:pl}}\cdot \max\left\{|\setg|^{\frac{d-1}{d}},\frac{|\setd|}{L_s^d} \right\} .\ 
\end{equation} 
\end{claim}
\begin{proof}
If $|\setg|^{\frac{d-1}{d}} < \frac{|\setd|}{L_s^d}$, then \eqref{eq:ineqmax} trivially holds. 
Thus, we assume that $|\setg|^{\frac{d-1}{d}} \geq \frac{|\setd|}{L_s^d}$. We will deduce \eqref{eq:ineqmax} from Theorem~\ref{thm:isop:pl}.
By \eqref{eq:isopmain:boundary:2}, 
\[
|A| \leq 2\cdot 3^d\cdot L_0^d\cdot|\setg| + L_s^d\cdot |\setg|^{\frac{d-1}{d}} \leq 3^{d+1}\cdot L_s^d\cdot|\setg| .\
\]
Since $|A|\geq L_s^{d(d+1)}$, we obtain that $|\setg|\geq \left(\frac{L_s}{L_0}\right)^{d^2}$.

Since $\setg\subseteq\pl_{K',s,0}(x_s')$, for all $x\in\setg$, $|\mathcal C_x| \geq \den_1 L_0^d$. 
Thus, $|A| \geq \den_1 L_0^d\cdot |\setg|$. 
Since also all the vertices in $\GG_s\cap Q_{K',s}(x_s')$ are $s$-good, we obtain as in \eqref{eq:sizecmax2} that 
\begin{multline*}
|A|\leq \frac 12\cdot |\cemax_{K,s,L_0}|
\leq \frac 12\cdot \left(\den_2 + 1 - \prod_{i=0}^\infty\left(1 - \left(\frac{4r_i}{l_i}\right)^d\right)\right)\cdot |Q_{K',s}(x_s')|\\
\stackrel{\eqref{eq:isop:cemax:ratio}}\leq \frac{\den_2 + 2\den_1}{4}\cdot \prod_{i=0}^\infty\left(1 - \left(\frac{4r_i}{l_i}\right)^d\right)\cdot |Q_{K',s}(x_s')|
\leq \frac{\den_2 + 2\den_1}{4}L_0^d\cdot |\pl_{K',s,0}(x_s')|,
\end{multline*}
where the last inequality follows from Lemma~\ref{l:rect}. Thus, $|\setg| \leq (1-c_\den)\cdot |\pl_{K',s,0}(x_s')|$.
By Theorem~\ref{thm:isop:pl} and Remark~\ref{rem:isop:pl}, 
\[
|\partial_{\pl_{K',s,0}(x_s')}\setg|\geq c_\den\cdot \gamma_{\scriptscriptstyle \ref{thm:isop:pl}}\cdot |\setg|^{\frac{d-1}{d}},
\]
completing the proof of \eqref{eq:ineqmax}.
\end{proof}

\bigskip

We are now ready to prove Theorem~\ref{thm:isop:cemax}. It easily follows from Claims~\ref{cl:isopmain} and \ref{cl:isopmain2}.
\begin{proof}[Proof of Theorem~\ref{thm:isop:cemax}] 
By \eqref{eq:isopmain:boundary}, \eqref{eq:isopmain:volume}, and \eqref{eq:ineqmax},
\[
\frac{|\partial_{\cemax_{K,s,L_0}} A|}{|A|^{\frac{d-1}{d}}}
\geq  \frac{\frac{1}{31^d}\cdot c_\den\cdot\gamma_{\scriptscriptstyle \ref{thm:isop:pl}}\cdot\max\left\{|\setg|^{\frac{d-1}{d}},\frac{|\setd|}{L_s^d} \right\}}
{(2\cdot 3^d\cdot L_0^d\cdot|\setg| + |\setd|)^{\frac{d-1}{d}}}\\
\geq  \frac{\frac{1}{31^d}\cdot c_\den\cdot\gamma_{\scriptscriptstyle \ref{thm:isop:pl}}\cdot \max\left\{|\setg|^{\frac{d-1}{d}},\frac{|\setd|}{L_s^d} \right\}}
{2\cdot 3^{d-1}\cdot L_0^{d-1}\cdot|\setg|^{\frac{d-1}{d}} + |\setd|^{\frac{d-1}{d}}} .\
\]
On the one hand, if $L_0^d\cdot |\setg| \geq |\setd|$, then 
\[
\frac{|\partial_{\cemax_{K,s,L_0}} A|}{|A|^{\frac{d-1}{d}}}
\geq \frac{\frac{1}{31^d}\cdot c_\den\cdot\gamma_{\scriptscriptstyle \ref{thm:isop:pl}}\cdot |\setg|^{\frac{d-1}{d}}}
{2\cdot 3^{d-1}\cdot L_0^{d-1}\cdot|\setg|^{\frac{d-1}{d}} + |\setd|^{\frac{d-1}{d}}}
\geq \frac{\frac{1}{31^d}\cdot c_\den\cdot\gamma_{\scriptscriptstyle \ref{thm:isop:pl}}}{3\cdot (3\cdot L_0)^{d-1}} .\
\]
On the other hand, if $L_0^d\cdot |\setg| \leq |\setd|$, then by \eqref{eq:isopmain:volume}, 
$|\setd| \geq \frac{1}{3^{d+1}}\cdot |A| \geq \frac{1}{3^{d+1}}\cdot L_s^{d(d+1)}\geq L_s^{d^2}$, and 
\[
\frac{|\partial_{\cemax_{K,s,L_0}} A|}{|A|^{\frac{d-1}{d}}}
\geq  \frac{\frac{1}{31^d}\cdot c_\den\cdot\gamma_{\scriptscriptstyle \ref{thm:isop:pl}}\cdot|\setd|^{\frac 1d}}
{3^d\cdot L_s^d}
\geq \frac{1}{93^d}\cdot c_\den\cdot\gamma_{\scriptscriptstyle \ref{thm:isop:pl}} .\
\]
The proof of Theorem~\ref{thm:isop:cemax} is complete with 
$\gamma_{\scriptscriptstyle \ref{thm:isop:cemax}} = \frac{1}{93^d\cdot L_0^{d-1}}\cdot c_\den\cdot\gamma_{\scriptscriptstyle \ref{thm:isop:pl}}$.
\end{proof}

\medskip

\begin{remark}
With a more careful analysis and assuming that Theorem~\ref{thm:isop:pl} holds for all subsets of size at least $\left(\frac{L_s}{L_0}\right)^{2d}$ 
(see Remark~\ref{rem:isopframe:conditions}), condition on $A$ in Theorem~\ref{thm:isop:cemax} can be relaxed to $|A|\geq L_s^{2d}$. 
Assuming that Theorem~\ref{thm:isop:pl} holds for all subsets (see Remark~\ref{rem:isopframe:conditions}), 
condition on $A$ in Theorem~\ref{thm:isop:cemax} can be relaxed to $|A|\geq L_s^d$. 
Since for our purposes the current statement of Theorem~\ref{thm:isop:cemax} suffices, we do not prove the stronger statement here. 
\end{remark}

\subsection{Graph distance}

In this section we study the graph distances $\dist_\set$ in $\set$ between vertices of $\cmax_{K,s,L_0}(x_s)$ for configurations in $\mathcal H^\den_{K,s}(x_s)$. 
As consequences, we prove that large enough balls centered at vertices of $\cmax_{K,s,L_0}(x_s)$ have 
regular volume growth (Corollary~\ref{cor:chemdist}) and allow for local extensions which satisfy an isoperimetric inequality (Corollary~\ref{cor:tildeC:existence}).
These results will be used in Section~\ref{sec:proof:main} to prove our main result. 

\begin{lemma}\label{l:chemdist}
Let $d\geq 2$ and $\den$ as in \eqref{def:eta}. Let $l_n$ and $r_n$, $n\geq 0$, be integer sequences such that for all $n$, $l_n>16r_n$ and 
$\prod_{n\geq 0}\left(1 + \frac{32r_n}{l_n}\right) \leq 2$. 
Let $L_0\geq 1$, $K\geq 1$, and $s\geq 0$ integers, $x_s\in\GG_s$. 
There exists $C_{\scriptscriptstyle \ref{l:chemdist}}= C_{\scriptscriptstyle \ref{l:chemdist}}(L_0)$ such that 
if $\mathcal H^\den_{K,s}(x_s)$ occurs, then for all $y,y'\in\cmax_{K,s,L_0}(x_s)$,
\[
\dist_{\set}(y,y') \leq C_{\scriptscriptstyle \ref{l:chemdist}}\cdot \max\left\{|y-y'|_\infty,L_s^d\right\}.
\]
\end{lemma}
\begin{proof}
Let $y_s, y_s'\in Q_{K,s}(x_s)\cap\GG_s$ be such that $(y_s + [0,L_s)^d)\subset\ballZ(y,L_s)$ and $(y_s'+[0,L_s)^d)\subset\ballZ(y',L_s)$. 
By \cite[Lemma~5.3]{DRS12} (applied to sequences $l_n$ and $4r_n$), there exist $y_0\in \pl_{K,s,0}(x_s)\cap (y_s + [0,L_s)^d)$ and $y_0'\in \pl_{K,s,0}(x_s)\cap (y_s' + [0,L_s)^d)$ 
which are connected by a nearest neighbor path of $0$-good vertices $z_1 = y_0, z_2,\ldots, z_{k-1},z_k = y_0'$ in $\pl_{K,s,0}(x_s)$, where  
$k\leq \prod_{n\geq 0}\left(1 + \frac{32r_n}{l_n}\right)\cdot\frac{|y_s - y_s'|_1 + L_s}{L_0}$. 

\medskip

Let $\widetilde z_i$ be an arbitrary vertex in $\mathcal C_{z_i}$. (Recall the definition of $\mathcal C_z$ from Lemma~\ref{l:fromG0toZd}.)
By Lemma~\ref{l:fromG0toZd}, for all $1\leq i < k$, $\widetilde z_i$ is connected to $\widetilde z_{i+1}$ in 
$\set\cap((z_i + [0,L_0)^d)\cup(z_{i+1} + [0,L_0)^d)$. 
Therefore, any vertices $\widetilde y\in \mathcal C_{y_0}$ and $\widetilde y'\in\mathcal C_{y_0'}$ are connected by a nearest neighbor path in 
$\set\cap \cup_{i=1}^k (z_i + [0,L_0)^d)$. Any such path consists of at most 
$L_0^d\cdot \prod_{n\geq 0}\left(1 + \frac{32r_n}{l_n}\right)\cdot\frac{|y_s - y_s'|_1 + L_s}{L_0}$ vertices. 

\medskip

By Corollary~\ref{cor:cmax:uniqueness}, $\widetilde y\in \cmax_{K,s,L_0}(x_s)\cap \ballZ(y,L_s)$ and $\widetilde y' \in \cmax_{K,s,L_0}(x_s)\cap \ballZ(y',L_s)$. 
Thus, by the definition of $\mathcal H^\den_{K,s}(x_s)$, $y$ is connected to $\widetilde y$ in $\set\cap\ballZ(y,2L_s)$ and 
$y'$ is connected to $\widetilde y'$ in $\set\cap\ballZ(y',2L_s)$. 

\medskip

By putting all the arguments together, we obtain that $y$ is connected to $y'$ by a nearest neighbor path 
in $\set$ of at most $2\cdot |\ballZ(0,2L_s)| + L_0^d\cdot \prod_{n\geq 0}\left(1 + \frac{32r_n}{l_n}\right)\cdot\frac{|y_s - y_s'|_1 + L_s}{L_0}$ vertices.
Since $|y_s - y_s'|_1\leq d\cdot |y-y'|_\infty + 2dL_s$, the result follows. 
\end{proof}

\medskip

\begin{corollary}\label{cor:chemdist}
In the setup of Lemmas~\ref{l:cmax:uniqueness} and \ref{l:chemdist},
there exists $c_{\scriptscriptstyle \ref{cor:chemdist}} = c_{\scriptscriptstyle \ref{cor:chemdist}}(\den,L_0)>0$ such that 
for any $C_{\scriptscriptstyle \ref{l:chemdist}} L_s^d\leq r\leq KL_s$ and $y\in\cmax_{K,s,L_0}(x_s)$, 
\[
\mu(\ballS(y,r))\geq c_{\scriptscriptstyle \ref{cor:chemdist}}\cdot r^d.
\]
\end{corollary}
\begin{proof}
Let $K' = \max\{k~:~kL_s\leq \frac{r}{C_{\scriptscriptstyle \ref{l:chemdist}}}\}$. 
There exists $y_s\in Q_{K,s}(x_s)\cap \GG_s$ such that $Q_{K',s}(y_s)\subset\ballZ(y,\frac{r}{C_{\scriptscriptstyle \ref{l:chemdist}}})\cap Q_{K,s}(x_s)$. 
Since $\mathcal H^\den_{K,s}(x_s)$ occurs, we can define the perforation $\pl_{K,s,0}(x_s)$ of $Q_{K,s}(x_s)$ as in \eqref{def:pl}. 
Consider also the perforation $\pl_{K',s,0}(y_s) = \pl_{K,s,0}(x_s)\cap Q_{K',s}(y_s)$ of $Q_{K',s}(y_s)$. 
By \eqref{eq:framecmax}, 
\[
\bigcup_{x\in\pl_{K',s,0}(y_s)}\mathcal C_x\subset\cmax_{K,s,L_0}(x_s).
\]
Since also $\bigcup_{x\in\pl_{K',s,0}(y_s)}\mathcal C_x\subset\ballZ(y,\frac{r}{C_{\scriptscriptstyle \ref{l:chemdist}}})$,
Lemma~\ref{l:chemdist} implies that 
\[
\bigcup_{x\in\pl_{K',s,0}(y_s)}\mathcal C_x \subset \ballS(y,r).
\]
By applying Lemma~\ref{l:rect} to $\pl_{K',s,0}(y_s)$ and using the fact that $|\mathcal C_x|\geq \den_1L_0^d$, we conclude from the above inclusion that 
\begin{eqnarray*}
|\ballS(y,r)| 
&\geq &\den_1\cdot (K'L_s)^d\cdot \prod_{i\geq 0}\left(1 - \left(\frac{4r_i}{l_i}\right)^d\right)\\
&\stackrel{\eqref{eq:isop:cemax:ratio}}\geq & \den_1\cdot (K'L_s)^d\cdot\frac{1 + \den_2}{1+\frac{\den_2+2\den_1}{2}}\\
&\stackrel{\eqref{def:eta}}\geq &\frac12\den_2\cdot \left(\frac{r}{2C_{\scriptscriptstyle \ref{l:chemdist}}}\right)^d~.
\end{eqnarray*}
Since $\mu(\ballS(y,r))\geq |\ballS(y,r)|$, the result follows with $c_{\scriptscriptstyle \ref{cor:chemdist}} = \frac12\den_2\cdot \frac{1}{(2C_{\scriptscriptstyle \ref{l:chemdist}})^d}$. 
\end{proof}

\medskip

\begin{corollary}\label{cor:tildeC:existence}
Let $\epsilon\in(0,\frac1d]$. 
In the setup of Theorem~\ref{thm:isop:cemax} and Lemma~\ref{l:chemdist}, 
if $\mathcal H^\den_{5K,s}(x_s')$ occurs with $x_s' = x_s + (-2KL_s,\dots,-2KL_s)$, 
then for all $L_s^{d+1+\frac{d^2-1}{\epsilon d}}\leq r\leq KL_s$ and $y\in\cmax_{K,s,L_0}(x_s)$, there exists $\mathcal C_{\ballS(y,r)}$ such that 
$\ballS(y,r)\subseteq \mathcal C_{\ballS(y,r)}\subseteq \ballS(y,8C_{\scriptscriptstyle \ref{l:chemdist}} r)$ and for all $A\subset\mathcal C_{\ballS(y,r)}$ with 
$|A|\leq \frac 12\cdot |\mathcal C_{\ballS(y,r)}|$, 
\[
|\partial_{\mathcal C_{\ballS(y,r)}}A|\geq \gamma_{\scriptscriptstyle \ref{thm:isop:cemax}}\cdot |A|^{\frac{d-1}{d} + \epsilon}\cdot (8r)^{-\epsilon d}.
\]
In particular, if $\epsilon = \frac1d$, then $|\partial_{\mathcal C_{\ballS(y,r)}}A|\geq \gamma_{\scriptscriptstyle \ref{thm:isop:cemax}}\cdot\frac{|A|}{8r}$.
\end{corollary}
\begin{proof}
Let $K' = \min\{k~:~kL_s\geq 2r+1\}+1$. (Note that $K'L_s\leq 4r$.) For $y\in\cmax_{K,s,L_0}(x_s)$, let $y_s\in\GG_s\cap Q_{5K,s}(x_s')$ be such that 
$\ballZ(y,r)\subseteq Q_{K',s}(y_s)\subseteq Q_{5K,s}(x_s')$. 
Since $\mathcal H^\den_{K',s}(y_s)$ occurs, by Corollary~\ref{cor:cmax:uniqueness}, $\ballS(y,r)\subseteq\cmax_{K',s,L_0}(y_s)\subseteq \cemax_{K',s,L_0}(y_s)$. 

By Lemma~\ref{l:chemdist}, for $r\geq L_s^d$, 
\[
\cemax_{K',s,L_0}(y_s)\subseteq \ballS(y,C_{\scriptscriptstyle \ref{l:chemdist}}(K'+4)L_s)\subseteq \ballS(y,8C_{\scriptscriptstyle \ref{l:chemdist}}r).
\]
By Corollary~\ref{cor:isop:cemax}, since $K' \geq L_s^{d+\frac{d^2-1}{\epsilon d}}$, for any $A\subset \cemax_{K',s,L_0}(y_s)$ with $|A|\leq \frac 12|\cemax_{K',s,L_0}(y_s)|$, 
\[
|\partial_{\cemax_{K',s,L_0}(y_s)}A|\geq \gamma_{\scriptscriptstyle \ref{thm:isop:cemax}}\cdot |A|^{\frac{d-1}{d} + \epsilon}\cdot ((K'+4)L_s)^{-\epsilon d}
\geq \gamma_{\scriptscriptstyle \ref{thm:isop:cemax}}\cdot |A|^{\frac{d-1}{d} + \epsilon}\cdot (8r)^{-\epsilon d}.
\]
The proof is complete by taking $\mathcal C_{\ballS(y,r)} = \cemax_{K',s,L_0}(y_s)$. 
\end{proof}

\section{Proof of Theorem~\ref{thm:vgb:main}}\label{sec:proof:main}

In this section we collect together the deterministic results that large enough balls have regular volume growth (Corollary~\ref{cor:chemdist}) 
and allow for local extensions satisfying an isoperimetric inequality (Corollary~\ref{cor:tildeC:existence})
to deduce Theorem~\ref{thm:vgb:main}. 
In fact, the result that we prove here is stronger. 
In Definition~\ref{def:vrb} we introduce the notions of regular and very regular balls, so that 
(very) regular ball is always (very) good (see Claim~\ref{cl:rb-gb}), and then prove in Proposition~\ref{prop:verygoodbox} that 
large balls are likely to be very regular. The main result is an immediate consequence of Proposition~\ref{prop:verygoodbox}. 

The following definition will only be used for the special choice of $\epsilon = \frac 1d$, see Claim~\ref{cl:rb-gb}. 
Nevertheless, we choose to work with the more general definition involving arbitrary $\epsilon \in(0,\frac1d]$, 
since smaller $\epsilon$'s give better isoperimetric inequalities, and could be used to prove stronger functional inequalities than 
the Poincar\'e inequality, as we learned from Jean-Dominique Deuschel (see, e.g., \cite[Section~3.2]{Nguyen}). 

\begin{definition}\label{def:vrb}
Let $C_V$, $C_P$, and $C_W\geq 1$ be fixed constants. Let $\epsilon\in(0,\frac1d]$. 
For $r\geq 1$ integer and $x\in V(G)$, we say that $\ballG(x,r)$ is $(C_V,C_P,C_W,\epsilon)$-{\it regular} if 
\[
\mu(\ballG(x,r))\geq C_V r^d
\]
and there exists a set $\mathcal C_{\ballG(x,r)}$ such that 
$\ballG(x,r)\subseteq\mathcal C_{\ballG(x,r)}\subseteq \ballG(x,C_Wr)$ and for any $A\subset \mathcal C_{\ballG(x,r)}$ with $|A|\leq \frac 12\cdot |\mathcal C_{\ballG(x,r)}|$, 
\[
|\partial_{\mathcal C_{\ballG(x,r)}}A|\geq \frac{1}{\sqrt{C_P}}\cdot |A|^{\frac{d-1}{d} + \epsilon} \cdot r^{-\epsilon d}.
\] 
We say $\ballG(x,R)$ is $(C_V,C_P,C_W,\epsilon)$-{\it very regular} if there exists $N_{\ballG(x,R)}\leq R^{\frac{1}{d+2}}$ such that 
$\ballG(y,r)$ is $(C_V,C_P,C_W)$-regular whenever $\ballG(y,r)\subseteq \ballG(x,R)$, and $N_{\ballG(x,R)} \leq r\leq R$. 

\medskip

In the special case $\epsilon = \frac 1d$, we omit $\epsilon$ from the notation and call $(C_V,C_P,C_W,\frac1d)$-(very) regular ball simply $(C_V,C_P,C_W)$-(very) regular.
\end{definition}

\medskip

\begin{claim}\label{cl:rb-gb}
If $\ballG(x,r)$ is $(C_V,C_P,C_W)$-{\it regular}, then it is $(C_V,C_P,C_W)$-{\it good}. 
\end{claim}
\begin{proof}
By \cite[Proposition~3.3.10]{Kumagai} and Remark~\ref{rem:wpi:minimum}, 
\[
\min_a\int_{\mathcal C_{\ballG(x,r)}}(f-a)^2 d\mu = \int_{\mathcal C_{\ballG(x,r)}}\left(f-\overline f_{\mathcal C_{\ballG(x,r)}}\right)^2 d\mu 
\leq C_P\cdot r^2\cdot \int_{E(\mathcal C_{\ballG(x,r)})} |\nabla f|^2 d\nu.
\]
Thus, again by Remark~\ref{rem:wpi:minimum}, 
\begin{multline*}
\min_a\int_{\ballG(x,r)}(f-a)^2 d\mu \leq \int_{\ballG(x,r)}\left(f-\overline f_{\mathcal C_{\ballG(x,r)}}\right)^2 d\mu 
\leq \int_{\mathcal C_{\ballG(x,r)}}\left(f-\overline f_{\mathcal C_{\ballG(x,r)}}\right)^2 d\mu\\
\leq C_P\cdot r^2\cdot \int_{E(\mathcal C_{\ballG(x,r)})} |\nabla f|^2 d\nu 
\leq C_P\cdot r^2\cdot \int_{E(\ballG(x,C_W r))} |\nabla f|^2 d\nu.
\end{multline*}
\end{proof}

\bigskip

Theorem~\ref{thm:vgb:main} is immediate from Claim~\ref{cl:rb-gb} and the following proposition, in which one needs to take $\epsilon = \frac1d$.

\begin{proposition}\label{prop:verygoodbox}
Let $d\geq 2$, $u\in(a,b)$, and $\vgb\in(0,\frac{1}{d+2})$. Let $\epsilon\in(0,\frac1d]$. 
Assume that the family of measures $\mathbb P^u$, $u\in(a,b)$, satisfies assumptions \p{} -- \ppp{} and \s{} -- \sss{}. 
There exist constants $C_V$, $C_P$, and $C_W$, $c_{\scriptscriptstyle \ref{prop:verygoodbox}}$ and $C_{\scriptscriptstyle \ref{prop:verygoodbox}}$ 
depending on $u$, $\vgb$, and $\epsilon$, such that for all $R\geq 1$, 
\[
\mathbb P^u\left[
\begin{array}{c}\text{$\ballS(0,R)$ is $(C_V,C_P,C_W,\epsilon)$-very regular}\\
\text{with $N_{\ballS(0,R)}\leq R^\vgb$}
\end{array}~\Big|~0\in\set_\infty\right]
\geq 1 - C_{\scriptscriptstyle \ref{prop:verygoodbox}}\cdot e^{-c_{\scriptscriptstyle \ref{prop:verygoodbox}}(\log R)^{1+\constS}}.
\]
\end{proposition}
\begin{proof}
We first make a specific choice of various parameters. 
Fix $u\in(a,b)$. We take 
\begin{equation}\label{def:eta:main}
\den_1 = \frac34\eta(u) \quad\text{and}\quad \den_2 = \frac54\eta(u),
\end{equation}
where $\eta(u)$ is defined in \sss{}. 
It is easy to see that $\den_1$ and $\den_2$ satisfy assumptions \eqref{def:eta}. 
We fix this choice of $\den=(\den_1,\den_2)$ throughout the proof. 

\medskip

Next we choose the scales for renormalization. For positive integers $l_0$, $r_0$, and $L_0$, we take
\begin{equation}\label{def:scales}
\scexp = \lceil 1/\epsP \rceil,\qquad l_n = l_0\cdot 4^{n^\scexp},\qquad r_n = r_0\cdot 2^{ n^\scexp},\qquad 
L_n = l_{n-1}\cdot L_{n-1},\quad n\geq 1,
\end{equation}
where $\epsP$ is defined in \ppp{}. By \cite[Lemmas~4.2 and 4.4]{DRS12}, under the assumptions \p{} -- \ppp{} and \s{} -- \sss{}, 
there exist $C_1 = C_1(u)<\infty$ and $C_2 = C_2(u,l_0)<\infty$ such that for all 
$l_0,r_0\geq C_1$, $L_0\geq C_2$, and $n\geq 0$, 
\begin{equation}\label{eq:nbad:proba}
\mathbb P^u\left[\text{$0$ is $n$-bad}\right] \leq 2\cdot 2^{-2^n} .\
\end{equation}
We choose $l_0,r_0\geq C_1$ so that the scales $l_n$ and $r_n$ defined in \eqref{def:scales} satisfy the conditions of 
Lemma~\ref{l:cmax:uniqueness}, Theorem~\ref{thm:isop:cemax}, and Lemma~\ref{l:chemdist}, 
and choose $L_0\geq C_2$. Thus, \eqref{eq:nbad:proba} is also satisfied. 

\medskip

Next we choose $s$ and $K$. 
Fix $R\geq 1$. Without loss of generality, we can assume that 
\[
R^\vgb\geq \max(C_{\scriptscriptstyle \ref{l:chemdist}}L_0^d,L_0^{d+1+\frac{d^2-1}{\epsilon d}}).
\]
Let
\[
s=\max \left\{s'~:~\max\{C_{\scriptscriptstyle \ref{l:chemdist}}L_{s'}^d,L_{s'}^{d+1+\frac{d^2-1}{\epsilon d}}\} \leq R^{\vgb}\right\}.
\]
With this choice of $s$, let $K = \min\{k : kL_s\geq 2R+1\}+1$, $x_s\in \GG_s$ such that $\ballZ(0,R)\subseteq Q_{K,s}(x_s)$, 
and $x_s' = x_s + (-2KL_s,\dots,-2KL_s)$.

\bigskip

We begin with the proof. 
If the event $\mathcal H^\den_{5K,s}(x_s')\cap\{0\in\set_\infty\}$ occurs, then $\ballS(0,R)\subseteq \cmax_{K,s,L_0}(x_s)$. 
Therefore, for all $y\in\ballS(0,R)$ and $R^\vgb\leq r\leq R$, by Corollaries~\ref{cor:chemdist} and \ref{cor:tildeC:existence}, 
the ball $\ballS(y,r)$ is $(c_{\scriptscriptstyle \ref{cor:chemdist}},\frac{64^{\epsilon d}}{\gamma_{\scriptscriptstyle \ref{thm:isop:cemax}}^2},8C_{\scriptscriptstyle \ref{l:chemdist}},\epsilon)$-regular. 
Thus, 
\begin{equation}\label{eq:Hu:veryregular}
\begin{array}{c}
\text{if the event $\mathcal H^\den_{5K,s}(x_s')\cap\{0\in\set_\infty\}$ occurs, then the ball $\ballS(0,R)$}\\[7pt]
\text{ is $(c_{\scriptscriptstyle \ref{cor:chemdist}},\frac{64^{\epsilon d}}{\gamma_{\scriptscriptstyle \ref{thm:isop:cemax}}^2},8C_{\scriptscriptstyle \ref{l:chemdist}},\epsilon)$-very regular 
with $N_{\ballS(0,R)}\leq R^\vgb$.}
\end{array}
\end{equation}
Let 
\[
C_V = c_{\scriptscriptstyle \ref{cor:chemdist}}, \quad 
C_P = \frac{64^{\epsilon d}}{\gamma_{\scriptscriptstyle \ref{thm:isop:cemax}}^2}, \quad C_W = 8C_{\scriptscriptstyle \ref{l:chemdist}}. 
\]
By \eqref{eq:Hu:veryregular}, it suffices to prove that there exist constants $c=c(u,\vgb,\epsilon,\epsP)>0$ and $C = C(u,\vgb,\epsilon,\epsP)<\infty$ such that for all $R\geq 1$, 
\begin{equation}\label{eq:HuKs:proba}
\mathbb P^u\left[\mathcal H^\den_{5K,s}(x_s')~|~0\in\set_\infty\right] \geq 1 - Ce^{-c(\log R)^{1+\constS}}.
\end{equation}
By Definition~\ref{def:mainevent}, \eqref{eq:nbad:proba}, and \s{}, there exists $C = C(u)<\infty$ such that 
\[
\mathbb P^u\left[\mathcal H^\den_{5K,s}(x_s')^c\right] \leq  
(5K+4)^d\cdot 2\cdot 2^{-2^s} + (5KL_s)^d\cdot C\cdot e^{-\funcS(u,2L_s)} ~.\
\]
Thus, it remains to show that for our choice of all the parameters, the right hand side of the above display is at most $Ce^{-c(\log R)^{1+\constS}}$. 

\medskip

Let $D = d+1 + \frac{d^2-1}{\epsilon d}$. 
By \eqref{def:scales} and the choice of $s$, for all $R\geq C_{\scriptscriptstyle \ref{l:chemdist}}\cdot L_0^{D/\vgb}$, 
\[
\left(\frac{R}{C_{\scriptscriptstyle \ref{l:chemdist}}}\right)^{\frac{\vgb}{D}} \leq L_{s+1} = l_s\cdot L_s \leq l_0\cdot 4 \cdot (L_s)^{1+2^\scexp},
\]
which implies that 
\begin{equation}\label{eq:Ls:lowerbound}
L_s \geq \frac{1}{4l_0} \left(\frac{R}{C_{\scriptscriptstyle \ref{l:chemdist}}}\right)^{\frac{\vgb}{D(1+2^\scexp)}} .\
\end{equation}
By \eqref{def:scales} and \eqref{eq:Ls:lowerbound}, there exists a constant $c = c(\vgb, \scexp, l_0, L_0,\epsilon)>0$ such that for all $R\geq C_{\scriptscriptstyle \ref{l:chemdist}}\cdot L_0^{D/\vgb}$, 
\begin{equation}\label{eq:bounds}
s\geq c\cdot(\log R)^{\frac{1}{1+\scexp}} - 1 .\
\end{equation}
Using \eqref{eq:funcS}, \eqref{eq:Ls:lowerbound}, and \eqref{eq:bounds},  
we deduce that there exist $c' = c'(u,\vgb,\scexp,\epsilon)>0$ and $C' = C'(u,\vgb, \scexp, l_0, L_0,\epsilon)<\infty$ such that for all $R\geq C'$, 
\[
2^s \geq (\log R)^{1+\constS} \quad\mbox{and}\quad \funcS(u,2L_s) \geq c'(\log R)^{1+\constS} ~.\
\]
By the choice of $K$, $KL_s\leq 4R$. 
Therefore, there exist $c''=c''(u,\vgb,\scexp,\epsilon)>0$ and $C'' = C''(u,\vgb,\scexp,l_0,L_0,\epsilon)<\infty$ such that for all $R\geq C''$,  
\begin{equation}\label{eq:eventH:proba}
\mathbb P^u\left[\mathcal H^\den_{5K,s}(x_s')^c \right]\leq C''e^{-c''(\log R)^{1+\constS}} .\
\end{equation}
Since $\mathbb P^u[0\in\set_\infty] = \eta(u)>0$, \eqref{eq:eventH:proba} implies \eqref{eq:HuKs:proba}. 
The proof is complete. 
\end{proof}

\medskip

\begin{remark}
The events $\seedde^\den_{x,L_0}$ and $\seedin^\den_{x,L_0}$ slightly differ 
from the corresponding events $\overline A^u_x$ and $\overline B^u_x$ in \cite{DRS12}, 
but only minor modifications are needed to adapt \cite[Lemmas~4.2 and 4.4]{DRS12} 
to our setting. 
 
There is room for flexibility in the choice of $\den$. 
For instance, if $\epsilon = \epsilon(u)\geq 0$ is chosen so that $\eta(u(1-\epsilon)) > \frac56\cdot \eta(u(1+\epsilon))$, 
Then $\den_1 = \frac34\eta(u(1-\epsilon))$ and $\den_2 = \frac54 \eta(u(1+\epsilon))$ satisfy \eqref{def:eta}, 
and \eqref{eq:nbad:proba} remains true for this choice of $\den$ by monotonicity. 
%
\end{remark}
%
%

\medskip

\begin{remark}\label{rem:ADS}
As we already mentioned in Remark~\ref{rem:results}(6), a new approach to the random conductance model 
on general graphs satisfying some regularity assumptions has been recently developed in \cite{ADS13,ADS14}. 
The main assumption on graphs there is \cite[Assumption~1.1]{ADS14}, which is reminiscent of Definition~\ref{def:vrb}, but stronger.
The main difference is that we do not require that an isoperimetric inequality is satisfied by subsets of a ball, but by those of a local extension of the ball. 
In fact, we do not know how to show (and if it is true) that subsets of balls satisfy the desired isoperimetric inequality of \cite[Assumption~1.1]{ADS14} in our setting.
It would be very interesting to see if the machinery developed in \cite{ADS13,ADS14} can be applied to 
graphs with all large balls being very regular in the sense of Definition~\ref{def:vrb}.
\end{remark}

\bigskip

\section{Proof of Theorem~\ref{thm:isop:pl}}\label{sec:pl:isop:proof}

The rough outline of the proof is the following. We first prove the isoperimetric inequality for all subsets of perforated lattices in two dimensions, 
see Lemma~\ref{l:isopineqG:2d}. In dimensions $d\geq 3$, we proceed in two steps. 
We first consider only macroscopic subsets $\mathcal A$ of the perforated lattice, i.e., those with the volume comparable with the volume of the perforated lattice. 
By applying a selection lemma, see Lemma~\ref{l:selection}, we identify a large number of disjoint two dimensional slices in the ambient box which on the one hand 
have a small non-empty intersection with $\mathcal A$, and on the other, all together contain a positive fraction of the volume of $\mathcal A$. 
We estimate the boundary of $\mathcal A$ in each of the slices using the two dimensional result, and conclude by estimating the boundary of $\mathcal A$ in the perforation 
by the sum of the boundaries of $\mathcal A$ in each of the slices. 
Finally, we treat the general case by constructing a suitable coarse graining of $\mathcal A$ 
from mesoscopic boxes in which $\mathcal A$ has positive density. 
The restriction of the boundary of $\mathcal A$ to such boxes is estimated by using the result from the first case. 
Both isoperimetric inequalities in $d\geq 3$ are stated in Theorem~\ref{thm:isopframe}.

%

\medskip

We begin with a number of auxiliary ingredients for the proof: (a) some general facts about isoperimetric inequalities (Section~\ref{sec:isop:generalfacts}) and 
(b) a combinatorial selection lemma (Section~\ref{sec:selection}). 

\subsection{Auxiliary results}

\subsubsection{General facts about isoperimetric inequalities}\label{sec:isop:generalfacts}

Here we collect some isoperimetric inequalities that we will frequently use. 
\begin{lemma}\label{l:boundaries}
Let $d\geq 2$, $n_1,\dots,n_d\geq 1$ integers with $\max_i n_i\leq N\cdot \min_i n_i$, 
and $C$ a positive real such that $N\cdot C^{\frac 1d}<1$. 
Then, for any subset $A$ of $G=\Z^d\cap[0,n_1)\times\dots\times[0,n_d)$ with $|A|\leq C\cdot |G|$, 
\[
|\partial_G A|\geq \max\left\{\left(1 + 2d\cdot (1 - N\cdot C^{\frac 1d})^{-1}\right)^{-1}\cdot |\partial_{\Z^d}A|,~ 
(1 - N\cdot C^{\frac 1d}) \cdot|A|^{\frac{d-1}{d}}\right\}.
\] 
\end{lemma}
\begin{proof}
The proof is similar to that of \cite[Proposition~2.2]{DeuschelPisztora}. 
Let $\pi_i$ be the projection of $\Z^d$ onto the $(d-1)$ dimensional sublattice of vertices with $i$th coordinate equal to $0$.
Let $P_i = \pi_i(A)$, $i'$ be a coordinate corresponding to $P_i$ with the maximal size, and $P' = P_{i'}$.  
Let $P'' = P'\cap\pi_{i'}(G\setminus A)$, i.e., the projection of those $i'$-columns that contain vertices from both $A$ and $G\setminus A$. 
Note that $|\partial_G A| \geq |P''|$ and $|\partial_{\Z^d} A| \leq |\partial_G A| + 2d\cdot |P'|$. 
Also note that $|P'\setminus P''| \leq \frac{|A|}{n_{i'}}\leq N\cdot C^{\frac 1d}\cdot |A|^{\frac{d-1}{d}}$.
By the Loomis-Whitney inequality, $|A|^{\frac{d-1}{d}}\leq |P'|$. 
Thus, $|\partial_G A|\geq |P''|\geq (1 - N\cdot C^{\frac 1d}) \cdot |P'|\geq (1 - N\cdot C^{\frac 1d}) \cdot|A|^{\frac{d-1}{d}}$ and 
$|\partial_{\Z^d} A|\leq |\partial_G A| \cdot \left(1 + 2d\cdot (1 - N\cdot C^{\frac 1d})^{-1}\right)$. 
\end{proof}

\begin{remark}\label{rem:isop}
Let $G$ be a finite graph, and assume that for all $A\subseteq G$ with $c_1\cdot |G|\leq |A|\leq \frac 12\cdot |G|$, 
$|\partial_G A|\geq c_2\cdot |A|^{\frac{d-1}{d}}$. 
Then for any $A'\subset G$ with $\frac 12\cdot |G|\leq |A'|\leq (1-c_1)\cdot |G|$, 
$|\partial_G A'| = |\partial_G (G\setminus A')| \geq c_2\cdot |G\setminus A'|^{\frac{d-1}{d}} \geq (c_1c_2)\cdot |A'|^{\frac{d-1}{d}}$. 
Thus, any such $A'$ also satisfies an isoperimetric inequality, but possibly with a smaller constant. 
\end{remark}

\subsubsection{Selection lemma}\label{sec:selection}

The aim of this section is to prove the following combinatorial lemma. 
Its Corollaries~\ref{cor:selection1} and \ref{cor:selection2} together with the two dimensional isoperimetric inequality of Lemma~\ref{l:isopineqG:2d} will be crucially used in 
the proof of the isoperimetric inequality for macroscopic subsets of perforated lattices in any dimension $d\geq 3$ in Theorem~\ref{thm:isopframe}.
\begin{lemma}\label{l:selection}
Let $\frac 67\leq C_2<1$, and for $d\geq 2$, let 
\[
C_d = \frac{C_2^{d-1}}{\prod_{j=1}^{d-2}\left(1 + \frac{3}{9^j}\right)},\qquad
\delta_d = \frac{1}{9^{d-2}} ~.
\]
Let $R_1,\dots,R_d$ be positive integers. Then, for any subset $A$ of $Q = [0,R_1)\times\dots\times[0,R_d)\cap\Z^d$ satisfying 
\[
1\leq |A|\leq C_d\cdot |Q|,
\]
there exist $S_1,\dots S_k$, disjoint two dimensional subrectangles of $Q$ such that 
\[
|A\cap \cup_i S_i| \geq \delta_d\cdot |A|,
\]
and for all $1\leq i\leq k$, 
\[
1\leq |A\cap S_i| \leq C_2\cdot |S_i|.
\]
\end{lemma}
\begin{corollary}\label{cor:selection1}
Note that $\prod_{j=1}^{d-2}\left(1 + \frac{3}{9^j}\right) \leq e^{\sum_{j\geq 1}\frac{3}{9^j}} = e^{\frac 38}$.
Thus, if we take $C_2 = e^{-\frac{1}{8(d-1)}} > \frac 67$, then $C_d > e^{-\frac 12} > \frac 12$, and Lemma~\ref{l:selection} implies that 
for any $A\subset Q$ with $|A|\leq \frac 12\cdot |Q|$, there exist disjoint two dimensional rectangles $S_1,\dots,S_k$ such that 
$|A\cap \cup_i S_i| \geq \frac{1}{9^{d-2}}\cdot |A|$ and $1\leq |A\cap S_i| \leq e^{-\frac{1}{8(d-1)}}\cdot |S_i|$.
\end{corollary}
\begin{corollary}\label{cor:selection2}
If $R_1 = \dots = R_d = R$, and $|A|\geq c_d\cdot R^d$ for some $c_d>0$, 
then at least $\frac{\delta_d c_d}{2} R^{d-2}$ of the $S_i$'s contain at least $\frac{\delta_d c_d}{2} R^2$ vertices from $A$. 
Indeed, if such a choice did not exist, then we would have 
\[
\delta_d c_d R^d \leq \delta_d\cdot |A|\leq |A\cap\cup_i S_i| < R^2\cdot \frac{\delta_d c_d}{2} R^{d-2} + \frac{\delta_d c_d}{2}R^2\cdot \left(k - \frac{\delta_d c_d}{2}R^{d-2}\right)
\leq \delta_d c_d R^d.
\]
\end{corollary}

\begin{proof}[Proof of Lemma~\ref{l:selection}]
The proof is by induction on $d$. For $d=2$ the statement is obvious. We assume that $d\geq 3$.  

Consider all two dimensional slices of the form $[0,R_1)\times [0,R_2)\times x$, $x\in [0,R_3)\times\dots\times[0,R_d)$. 
If among them there exist slices $S_1,\dots,S_k$ such that $|A\cap\cup_iS_i|\geq \delta_d\cdot |A|$ and for all $i$, $1\leq |A\cap S_i|\leq C_2\cdot R_1R_2$, 
then we are done. 

Thus, assume the contrary. 
Let $\mathcal S_1$ be the subset of those slices that contain $> C_2\cdot R_1R_2$ vertices from $A$, and 
$\mathcal S_2$ the rest. By definition, $|\mathcal S_1| \leq \frac{|A|}{C_2\cdot R_1R_2}$, and by assumption, 
$|A\cap\cup_{S\in\mathcal S_2}S| < \delta_d \cdot |A|$.

\begin{figure}[!tp]
\centering
\resizebox{7cm}{!}{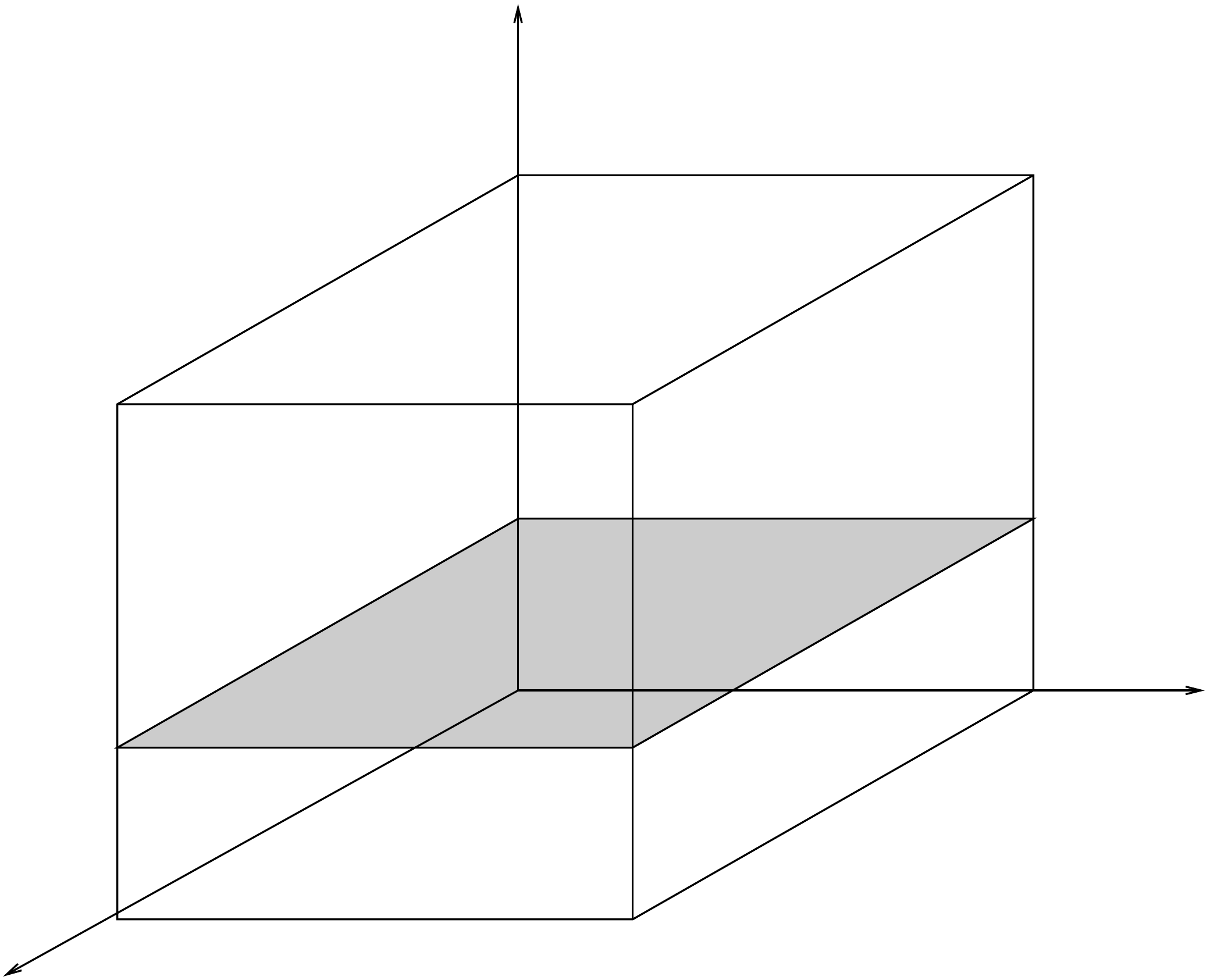}~\resizebox{7cm}{!}{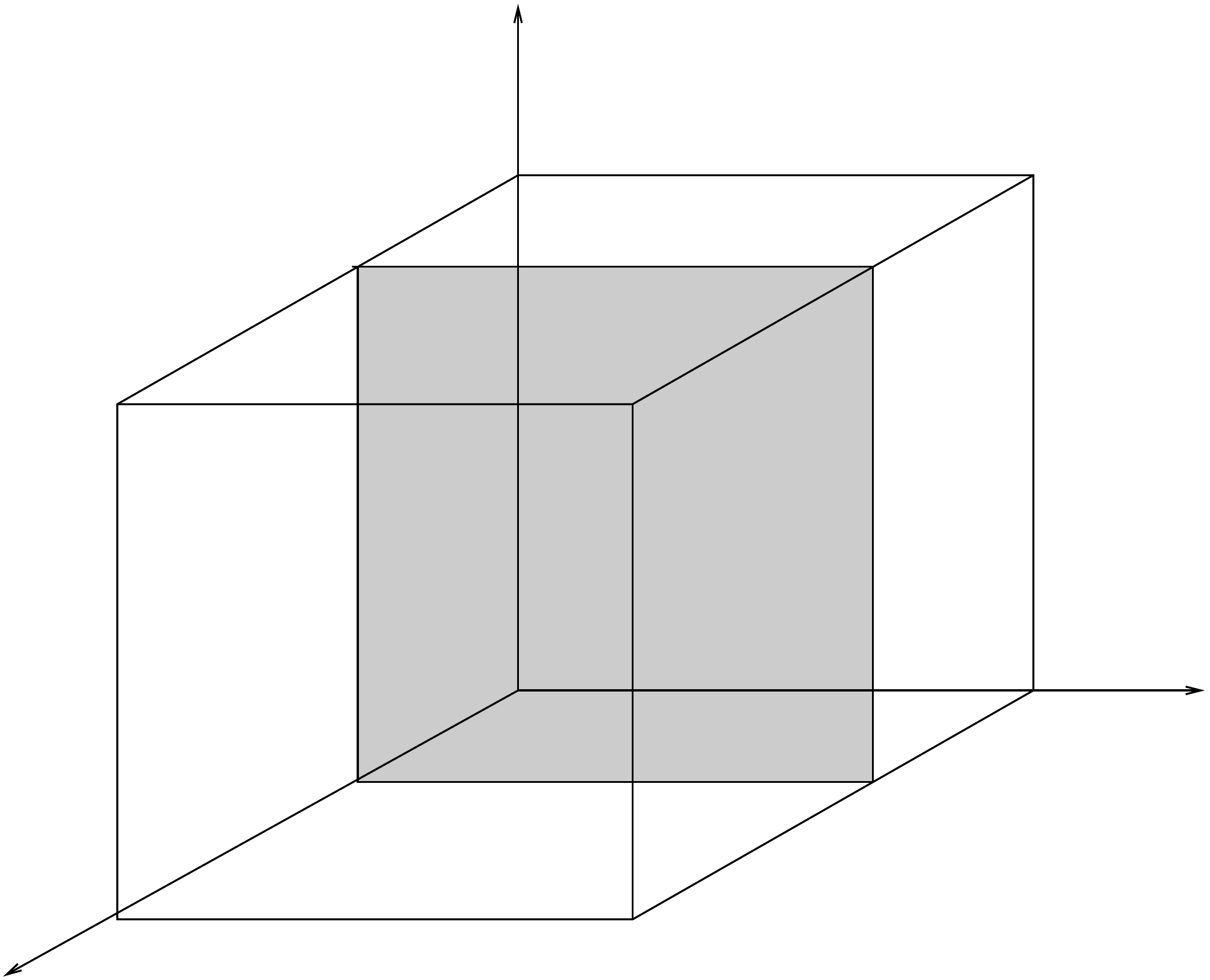}
\caption{An illustration of a slice $[0,R_1)\times [0,R_2)\times z$, $z\in [0,R_3)$ (left), and 
a rectangle $x\times[0,R_2)\times[0,R_3)$, $x\in[0,R_1)$ from $\mathcal M$ (right) in $3$ dimensions.}
\label{fig:select}
\end{figure}

\medskip

Consider $(d-1)$ dimensional rectangles 
\[
\mathcal M = \left\{x\times[0,R_2)\times\dots\times[0,R_d),\quad x\in[0,R_1)\right\},
\]
and consider separately their intersections with $\mathcal S_1$ and $\mathcal S_2$. 

First, consider intersections with $\mathcal S_1$. 
Each of the rectangles from $\mathcal M$ intersects $\cup_{S\in\mathcal S_1}S$ in at most $R_2\cdot \frac{|A|}{C_2\cdot R_1R_2} = \frac{|A|}{C_2\cdot R_1}$ vertices. 
Since $|A\cap\cup_{S\in\mathcal S_1}S| \geq (1 - \delta_d)\cdot |A|$, 
the number of rectangles $M\in\mathcal M$ with $|M\cap A|\geq \frac{|A|}{3\cdot R_1}$ is at least $\frac 23  R_1$.
Indeed, if not, then at least $\frac 13 R_1$ of rectangles from $\mathcal M$ contain $< \frac{|A|}{3\cdot R_1}$ vertices from $A$, and 
\[
|A\cap\cup_{S\in\mathcal S_1}S| < \frac 13 R_1\cdot \frac{|A|}{3\cdot R_1} + \frac 23 R_1\cdot \frac{|A|}{C_2\cdot R_1}
= \left(\frac 19 + \frac{2}{3\cdot C_2}\right)\cdot |A|
\leq \frac 89\cdot |A| \leq (1 - \delta_d)\cdot |A|,
\]
which is a contradiction. 

Next, consider intersections with $\mathcal S_2$. 
Since $|A\cap\cup_{S\in\mathcal S_2}S| \leq \delta_d\cdot |A|$, 
the number of rectangles $M\in\mathcal M$ with $|A\cap M\cap\cup_{S\in\mathcal S_2}S|\leq 3\delta_d\cdot \frac{|A|}{R_1}$ is 
at least $\frac 23 R_1$.
Indeed, if not, then for at least $\frac 13 R_1$ of them,  $|A\cap M\cap\cup_{S\in\mathcal S_2}S|>3\delta_d\cdot \frac{|A|}{R_1}$, and 
\[
|A\cap\cup_{S\in\mathcal S_2}S|> \frac 13 R_1 \cdot 3\delta_d\cdot \frac{|A|}{R_1} = \delta_d \cdot |A|,
\]
which is a contradiction. 

Therefore, we can choose $M_1,\dots,M_{\frac 13R_1}\in\mathcal M$ 
such that for each $1\leq i\leq \frac 13R_1$, 
\[
|A\cap M_i|\geq \frac{|A|}{3\cdot R_1}, \quad 
|A\cap M_i\cap\cup_{S\in\mathcal S_1}S| \leq \frac{|A|}{C_2\cdot R_1}, 
\quad
|A\cap M_i\cap\cup_{S\in\mathcal S_2}S| \leq 3\delta_d\cdot \frac{|A|}{R_1}.
\]
In particular, for each $1\leq i\leq \frac 13R_1$, 
\begin{eqnarray*}
|A\cap M_i|
&= 
&|A\cap M_i\cap\cup_{S\in\mathcal S_1}S| + |A\cap M_i\cap\cup_{S\in\mathcal S_2}S|
\leq 
\frac{|A|}{C_2\cdot R_1} + 3\delta_d\cdot \frac{|A|}{R_1}\\
&\leq &\frac{C_d}{C_2}\cdot \prod_{j=2}^dR_j\cdot \left(1 + \frac{3}{9^{d-2}}\right)
= C_{d-1}\cdot\prod_{j=2}^dR_j
\end{eqnarray*}
and 
\[
|A\cap\cup_i M_i| = \sum_i|A\cap M_i| \geq \frac 13 R_1 \cdot \frac{|A|}{3\cdot R_1} = \frac{|A|}{9}.
\]
If $d=3$, then $M_i$ are disjoint two dimensional rectangles satisfying all the requirements of the lemma. 
If $d>3$, consider the sets $A_i = A\cap M_i$, $1\leq i\leq \frac 13R_1$.  
They satisfy assumption of the lemma with $d$ replaced by $d-1$.
Therefore, there exist disjoint two dimensional rectangles $(S_{ij})_{1\leq j\leq k_j}$ in $M_i$ such that 
for all $1\leq j\leq k_i$, 
\[
|A_i\cap S_{ij}|\leq C_2\cdot |S_{ij}|,
\]
and
\[
|A_i\cap\cup_j S_{ij}| \geq \delta_{d-1}\cdot |A_i|.
\]
It is now easy to conclude that the two dimensional rectangles $(S_{ij})_{1\leq j\leq k_i, 1\leq i\leq \frac 13R_1}$ 
satisfy all the requirements of the lemma. 
Indeed, they are disjoint, 
\[
|A\cap\cup_{ij} S_{ij}| = \sum_i |A_i\cup_j S_{ij}|\geq \frac 13R_1\cdot \delta_{d-1}\cdot |A_i|
\geq \frac 13R_1\cdot \delta_{d-1}\cdot\frac{|A|}{3\cdot R_1} = \delta_d\cdot |A|,
\]
and for each $i$ and $j$, 
\[
|A\cap S_{ij}|= |A_i\cap S_{ij}| \leq C_2 \cdot |S_{ij}|.
\]
The proof is complete. 
\end{proof}

\subsection{Isoperimetric inequality in two dimensions}

The main goal of this section is to prove the following lemma. 
It immediately implies Theorem~\ref{thm:isop:pl} in the case $d=2$, but actually gives 
an isoperimetric inequality which holds for {\it all} $\mathcal A\in \pl_{K,s,0}(x_s)$ with $1\leq |\mathcal A|\leq \frac 12\cdot |Q_{K,s}\cap\GG_0|$.
\begin{lemma}\label{l:isopineqG:2d}
Let $d=2$. Let $l_n$ and $r_n$, $n\geq 0$, be integer sequences such that for all $n$, $l_n>8r_n$, $l_n$ is divisible by $r_n$, and 
\begin{equation}\label{eq:isopineqG:2d:ratio}
\prod_{j=0}^\infty \left(1 - \left(\frac{4r_j}{l_j}\right)^2\right) \geq \frac{15}{16}\qquad\text{and}\qquad
3456\cdot \sum_{j=0}^\infty\frac{r_j}{l_j}\leq \frac{1}{10^6}. 
\end{equation}
Then for any integers $s\geq 0$, $L_0\geq 1$, and $K\geq 1$, $x_s\in\GG_s$, 
and two families of events $\seedde$ and $\seedin$,  
if all the vertices in $\GG_s\cap Q_{K,s}(x_s)$ are $s$-good, then
for any $\mathcal A\subseteq \pl_{K,s,0}(x_s)$ such that 
$1 \leq |\mathcal A|\leq \frac 12\cdot |Q_{K,s}(x_s)\cap\GG_0|$, 
\[
|\partial_{\pl_{K,s,0}(x_s)} \mathcal A|\geq \frac{1}{10^6}\cdot |\mathcal A|^{\frac 12}.
\]
\end{lemma}
\begin{remark}
\begin{itemize}\itemsep0pt
\item[(1)]
Assumptions \eqref{eq:isopineqG:2d:ratio} and the constant $\frac{1}{10^6}$ in the result of Lemma~\ref{l:isopineqG:2d} are not optimal for our proof, 
but rather chosen to simplify calculations. 
\item[(2)]
We believe that an analogue of Lemma~\ref{l:isopineqG:2d} holds for all $d\geq 2$, but cannot prove it. 
There is only one place in the proof where the assumption $d=2$ is used, see Remark~\ref{rem:extensiond>2}.
\end{itemize}
\end{remark}
\begin{proof}
Fix $s\geq 0$ and $K\geq 1$ integers, $x_s\in\GG_s$. 
Recall the definition of $\pl_{K,s,i}(x_s)$ from \eqref{def:pl}, and write $\pl_i$ for $\pl_{K,s,i}(x_s)$ throughout the proof. 
Note that $\pl_s = Q_{K,s}(x_s)\cap\GG_0$ and for all $i$, $\pl_{i-1}\subseteq\pl_i$. 

\medskip

Let $\mathcal A$ be a subset of $\pl_0$ such that $1\leq |\mathcal A|\leq \frac 12\cdot |\pl_s|$. 
We need to prove that $|\partial_{\pl_0} \mathcal A|\geq \frac{1}{10^6}\cdot |\mathcal A|^{\frac 12}$.
First of all, without loss of generality we can assume that both $\mathcal A$ and $\pl_0\setminus \mathcal A$ are connected in $\GG_0$. 
(For the proof of this claim, see page 112 in \cite[Section~3.1]{MathieuRemy}.)

\medskip

\begin{figure}[!tp]
\centering
\resizebox{7cm}{!}{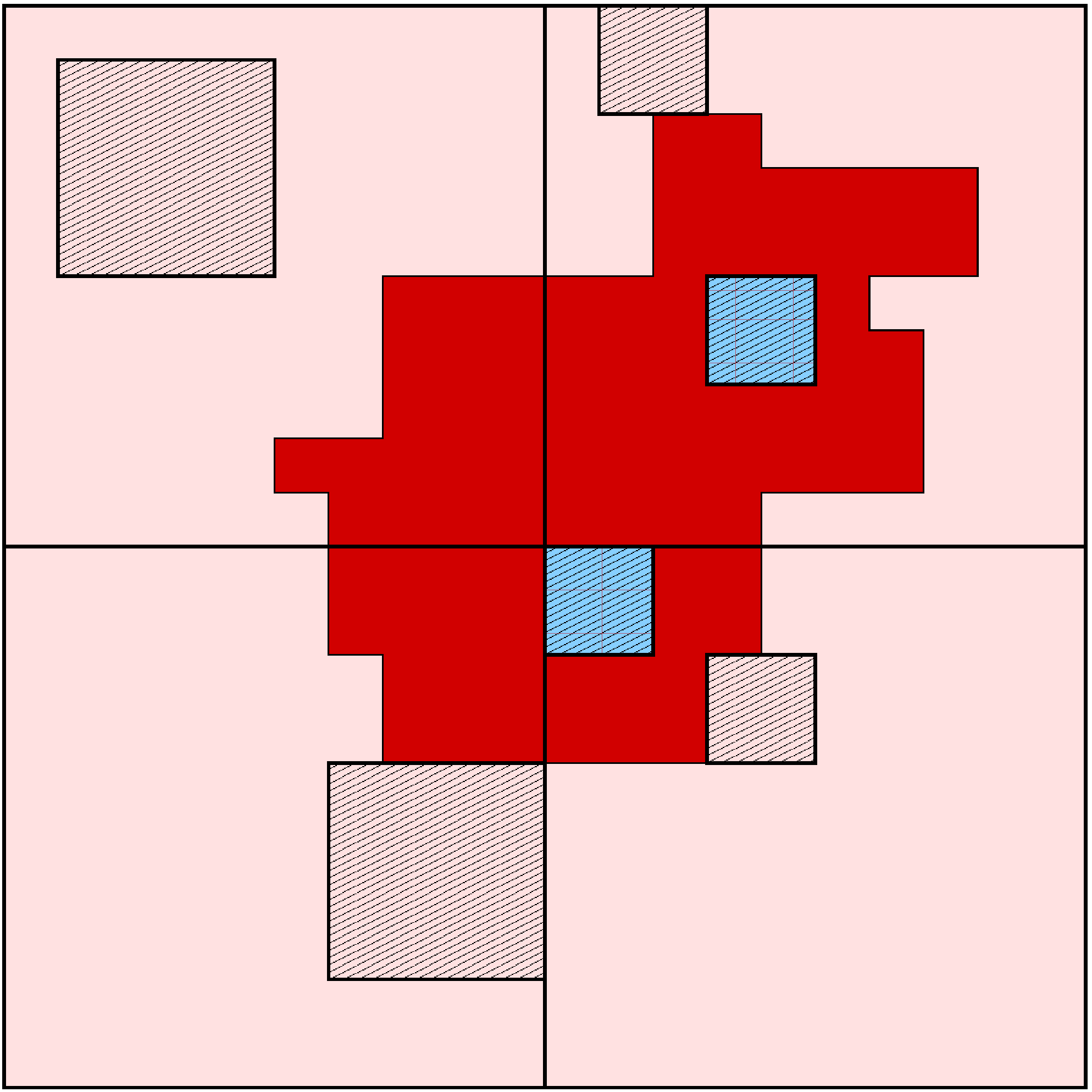}~\resizebox{7cm}{!}{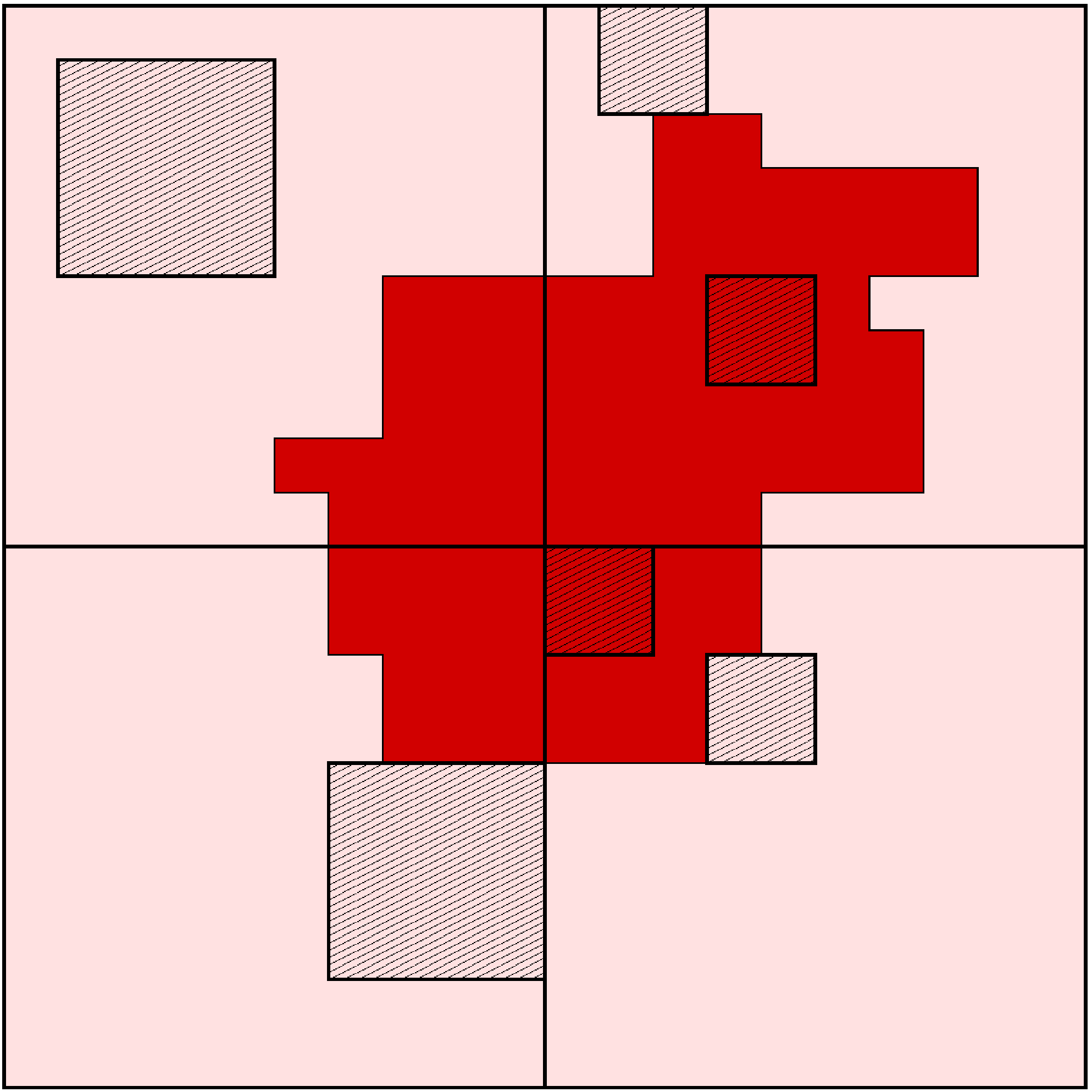}~
\caption{The set $\mathcal A'$ is obtained from $\mathcal A$ by adding to it all the holes $B_i$ completely surrounded by $\mathcal A$. 
This operation does not change the boundary of $\mathcal A$ in $\pl_0$.}
\label{fig:mathcalA}
\end{figure}

Let $B,B_1,\dots, B_m$ be all the connected components (in $\GG_0$) of $\pl_s\setminus\mathcal A$, of which 
$B$ is the unique component intersecting $\pl_0$, and $B_i$'s are the ``holes'' in $\pl_s$ completely surrounded by $\mathcal A$. 
(See Figure~\ref{fig:mathcalA}.) 
The boundary of $\mathcal A$ in $\pl_0$ does not contain any edges adjacent to $B_i$'s. 
It is convenient to absorb all the holes $B_i$'s into $\mathcal A$ to get the set $\mathcal A'$ 
with the same boundary in $\pl_0$, but with an important feature that its exterior vertex boundary in $\pl_s$ is $*$-connected.
More precisely, let 
\[
\mathcal A' = \mathcal A\cup\bigcup_{i=1}^m B_i \quad \text{and} \quad \pl_i' = \pl_i\cup \bigcup_{i=1}^m B_i.
\] 
Then, (a) $\partial_{\pl_0'}\mathcal A' = \partial_{\pl_0}\mathcal A$, (b) $|\mathcal A'| \geq |\mathcal A|$, 
(c) $\mathcal A'$ is connected in $\GG_0$, (d) $\pl_0'\setminus \mathcal A' = \pl_0\setminus\mathcal A$ (in particular, connected in $\GG_0$), 
and (e) for any $x,x'\in \mathcal E = \{y\in\pl_s~:~\{x,y\}\in\partial_{\pl_s}\mathcal A'\text{ for some }x\in\mathcal A'\}$ 
(the exterior vertex boundary of $\mathcal A'$ in $\pl_s$) there exist 
$z_0 = x,z_1,\dots,z_m = x'\in\mathcal E$ such that $|z_k - z_{k+1}|_\infty = L_0$ for all $k$ 
(i.e., $\mathcal E$ is $*$-connected).
Properties (a-d) are immediate from the definition of $\mathcal A'$, and property (e) follows from \cite[Lemma~2.1(ii)]{DeuschelPisztora} 
and the facts that $\mathcal A'$ and $\pl_s\setminus \mathcal A'$ are connected in $\GG_0$.
 
\medskip

By properties (a-b) of $\mathcal A'$, it suffices to prove that
\[
|\partial_{\pl_0'} \mathcal A'|\geq \frac{1}{10^6}\cdot |\mathcal A'|^{\frac 12}.
\]
By Lemma~\ref{l:rect} and the first part of \eqref{eq:isopineqG:2d:ratio},
$|\mathcal A'| \leq |\mathcal A| + |\pl_s\setminus\pl_0| \leq \frac{9}{16}\cdot |\pl_s|$.
Thus, by Lemma~\ref{l:boundaries}, 
\begin{equation}\label{eq:isopineqG:2d:boundaryGs}
|\partial_{\pl_s} \mathcal A'|\geq \frac 14 \cdot |\mathcal A'|^{\frac 12}.
\end{equation}
Therefore, it suffices to prove that 
\begin{equation}\label{eq:mathcalA':boundaryQsQ0}
|\partial_{\pl_0'} \mathcal A'|\geq \frac{2}{5\cdot 10^5}\cdot |\partial_{\pl_s} \mathcal A'|. 
\end{equation}
The proof of \eqref{eq:mathcalA':boundaryQsQ0} is done by partitioning $\partial_{\pl_s} \mathcal A'\setminus \partial_{\pl_0'} \mathcal A'$ 
into the sets $\delta_i$ of edges with one end vertex in $\mathcal A'$ and the other in $\pl_i \setminus\pl_{i-1}$ 
and comparing the cardinality of $\delta_i$'s with that of $\partial_{\pl_s} \mathcal A'$. 
If $\partial_{\pl_s} \mathcal A'$ is very large (macroscopic), then all $\delta_i$ are negligibly small in comparison to $\partial_{\pl_s} \mathcal A'$. 
It is more delicate to estimate the size of $\delta_i$'s if $\partial_{\pl_s} \mathcal A'$ is small, as the contribution of some $\delta_i$'s 
to the boundary $\partial_{\pl_s} \mathcal A'$ may be quite significant. In this case, we will introduce a suitable scale on which $\partial_{\pl_s} \mathcal A'$ 
is large, and view $\mathcal A'$ as a disjoint union of subsets of boxes on the new scale. Let
\[
\delta_i = \partial_{\pl_i'}\mathcal A'\setminus \partial_{\pl_{i-1}'}\mathcal A'.
\]
Then, $\delta_j$'s are disjoint and for any $1\leq i\leq s$, 
\begin{equation}\label{eq:boundaries:i}
|\partial_{\pl_0'}\mathcal A'| = |\partial_{\pl_i'}\mathcal A'| - \sum_{j=1}^i|\delta_j|.
\end{equation}
Let
\[
t = \max\left\{0\leq i\leq s~:~|\partial_{\pl_s}\mathcal A'|\geq \frac{1}{12}\cdot\frac{L_i}{L_0}\right\}.
\]
The scale $L_t$ is the correct scale to study $\partial_{\pl_s}\mathcal A'$. As we will see below in \eqref{eq:relationdelta}, 
the intersection of $\partial_{\pl_s}\mathcal A'$ with $\pl_i\setminus\pl_{i-1}$, $i\leq t$ (holes of size significantly smaller than $L_t$),
is negligible in comparison to  $\partial_{\pl_s}\mathcal A'$. 
In particular, it will be enough to conclude \eqref{eq:mathcalA':boundaryQsQ0} in the case $t=s$, see \eqref{eq:isopineqG:2d:boundaryG0:1}. 
If $t<s$, then $\partial_{\pl_s}\mathcal A'$ is small and may have a significant intersection with $\pl_{t+1}\setminus\pl_t$. 
An additional argument will be used to deal with this case, see below \eqref{eq:caset<s}. 

\medskip

We begin with an estimation of the part of $\partial_{\pl_s}\mathcal A'$ adjacent to ``small holes''. 
\begin{claim}\label{cl:relationdelta}
For all $1\leq i\leq t$,
\begin{equation}\label{eq:relationdelta}
|\delta_i|\leq 3456\cdot \frac{r_{i-1}}{l_{i-1}}\cdot |\partial_{\pl_s}\mathcal A'|.
\end{equation}
\end{claim}
\begin{proof}[Proof of Claim~\ref{cl:relationdelta}]
By the definition of $\pl_i$'s, 
the set $\pl_i \setminus\pl_{i-1}$ can be expressed as the disjoint union of boxes $S_j = \GG_0\cap(y_j + [0,2r_{i-1}L_{i-1})^2)$, 
for some $y_1,\dots,y_k\in(r_{i-1}L_{i-1})\cdot \Z^2$, such that 
every box $S_j$ is within $\ell^\infty$ distance $L_i$ from at most $36$ $S_{j'}$'s. (By Remark~\ref{rem:propRzi}, each $L_i$-box contains at most $4$ $S_j$'s, 
and it is adjacent to at most $8$ other $L_i$-boxes, hence $4\cdot 9 = 36$.)

\medskip

To estimate the size of $\delta_i$, we consider two cases: (a) $\partial_{\pl_s}\mathcal A'$ is adjacent to few $S_j$'s, in 
which case $\delta_i$ is very small, (b) $\partial_{\pl_s}\mathcal A'$ is adjacent to many $S_j$'s, in which case many of 
the $S_j$'s will be well-separated and $\mathcal A'$ will be spread out. To handle this case we will use the fact that the exterior vertex boundary of 
$\mathcal A'$ is $*$-connected, thus the majority of edges in $\partial_{\pl_s}\mathcal A'$ will be ``in between'' $S_j$'s. (See Figure~\ref{fig:deltai}.)

\begin{figure}[!tp]
\centering
\resizebox{8cm}{!}{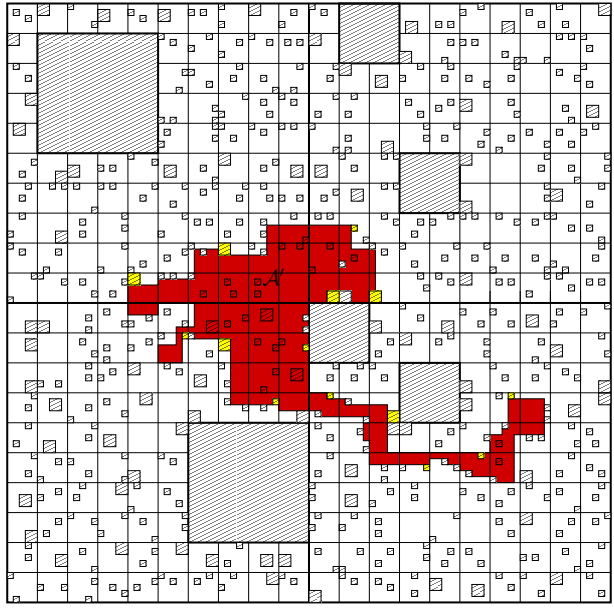}
\caption{Since every $S_j$ is within $L_i$ distance from at most $35$ other $S_j$'s, 
if the set $\mathcal A'$ is adjacent to many $S_j$'s then it must be adjacent to 
some sufficiently separated $S_j$ (drawn in yellow), and its boundary is thus stretched between these $S_j$'s. 
In two (and only two) dimensions, this is sufficient to conclude that the boundary of $\mathcal A'$ 
is much larger than its part adjacent to all the $S_j$'s, which we call $\delta_i$.}
\label{fig:deltai}
\end{figure}

Let $N_i$ be the total number of those $S_j$'s which are adjacent (in $\GG_0$) to $\mathcal A'$.
Since for each $j$, $|\partial_{\pl_s} S_j|\leq 8\frac{r_{i-1}L_{i-1}}{L_0}$, 
it follows that $|\delta_i| \leq N_i\cdot 8 \frac{r_{i-1}L_{i-1}}{L_0}$. 
We consider separately the cases $N_i\leq 36$ and $N_i> 36$. 

If $N_i\leq 36$, then 
\begin{equation}\label{eq:deltai:Ni<100}
|\delta_i| \leq N_i\cdot 8 \frac{r_{i-1}L_{i-1}}{L_0}
\leq 36\cdot 8\cdot \frac{r_{i-1}}{l_{i-1}}\cdot \frac{L_i}{L_0}
\leq 36\cdot 8\cdot 12\cdot\frac{r_{i-1}}{l_{i-1}}\cdot |\partial_{\pl_s}\mathcal A'|,
\end{equation}
where the last inequality follows from the definition of $t$ and the fact that $i\leq t$. 

If $N_i>36$, then $\mathcal A'$ is adjacent to at least $\lceil \frac{N_i}{36}\rceil (\geq 2)$ of $S_j$'s which are 
pairwise at $\ell^\infty$ distance at least $L_i$ from each other. 
Recall from property (e) of $\mathcal A'$ that $\mathcal E$ is the exterior vertex boundary of $\mathcal A'$, which is $*$-connected. 
Since $\mathcal E$ intersects each of the $\lceil \frac{N_i}{36}\rceil$ well separated $S_j$'s, 
the intersections of $\mathcal E$ with $\frac13 L_i$-neighborhoods of the $S_j$'s are disjoint sets of vertices of cardinality $\geq \frac 13\frac{L_i}{L_0}$ each.
Therefore, $|\mathcal E|\geq \frac 13\frac{L_i}{L_0}\cdot \frac{N_i}{36}$, and we obtain that  
\begin{equation}\label{eq:deltai:Ni>100}
|\delta_i| \leq N_i\cdot 8 \frac{r_{i-1}L_{i-1}}{L_0}
\leq 36\cdot 3\cdot 8\cdot \frac{r_{i-1}}{l_{i-1}}\cdot |\mathcal E|
\leq 36\cdot 3\cdot 8\cdot 4\cdot \frac{r_{i-1}}{l_{i-1}}\cdot |\partial_{\pl_s}\mathcal A'|, 
\end{equation}
where the last inequality follows from the fact that each vertex of $\mathcal E$ is adjacent to at most $4$ edges from $\partial_{\pl_s}\mathcal A'$.

Combining \eqref{eq:deltai:Ni<100} and \eqref{eq:deltai:Ni>100} we get \eqref{eq:relationdelta}.
\end{proof}

\medskip

If the boundary $\partial_{\pl_s}\mathcal A'$ is macroscopic, namely, if $t=s$, then 
the intersection of $\partial_{\pl_s}\mathcal A'$ with any hole is negligible, and Claim~\ref{cl:relationdelta} immediately implies \eqref{eq:mathcalA':boundaryQsQ0}. 
Indeed, by \eqref{eq:boundaries:i} and \eqref{eq:relationdelta},
\begin{equation}\label{eq:isopineqG:2d:boundaryG0:1}
|\partial_{\pl_0'}\mathcal A'| = |\partial_{\pl_t'}\mathcal A'| - \sum_{j=1}^t|\delta_j|
\geq 
\left(1 - 3456\cdot \sum_{j=0}^\infty\frac{r_j}{l_j}\right)\cdot|\partial_{\pl_s}\mathcal A'|, 
\end{equation}
and \eqref{eq:mathcalA':boundaryQsQ0} follows from \eqref{eq:isopineqG:2d:boundaryG0:1} and the second part of \eqref{eq:isopineqG:2d:ratio}.

\medskip

In the rest of the proof we consider the case of small $\partial_{\pl_s}\mathcal A'$, namely $t<s$. 
In this case,  
\begin{equation}\label{eq:caset<s}
\frac{1}{12}\cdot\frac{L_t}{L_0}\leq |\partial_{\pl_s}\mathcal A'|< \frac{1}{12}\cdot\frac{L_{t+1}}{L_0}\leq \frac{1}{12}\cdot\frac{L_s}{L_0}.
\end{equation}
As already mentioned, this case is more delicate, since $\partial_{\pl_s}\mathcal A'$ may have large intersection with big holes, 
for instance, $\delta_{t+1}$ is generally not negligible in comparison to $\partial_{\pl_s}\mathcal A'$. 

\medskip

We first consider the case when $\partial_{\pl_s}\mathcal A'$ is still relatively big in comparison to the boundary of holes in $\pl_{t+1}\setminus \pl_t$. 
Assume that $|\partial_{\pl_s}\mathcal A'|> 14\cdot 36\cdot 8\cdot \frac{r_tL_t}{L_0}$. 
In this case we will show that 
\begin{equation}\label{eq:deltat+1plt+1}
|\delta_{t+1}| \leq \frac{1}{14}\cdot |\partial_{\pl_s}\mathcal A'|\quad \text{and}\quad
|\partial_{\pl_{t+1}'}\mathcal A'|\geq \frac{1}{7}\cdot |\partial_{\pl_s}\mathcal A'|.
\end{equation}
Together with Claim~\ref{cl:relationdelta}, 
\eqref{eq:deltat+1plt+1} is sufficient for \eqref{eq:mathcalA':boundaryQsQ0}. Indeed, by \eqref{eq:boundaries:i}, 
\begin{equation}\label{eq:isopineqG:2d:boundaryG0:2}
|\partial_{\pl_0'}\mathcal A'| 
= |\partial_{\pl_{t+1}'}\mathcal A'| - \sum_{j=1}^{t+1}|\delta_j|
\geq \left(\frac{1}{14} - 3456\cdot \sum_{j=0}^\infty\frac{r_j}{l_j}\right)\cdot |\partial_{\pl_s}\mathcal A'|, 
\end{equation}
and \eqref{eq:mathcalA':boundaryQsQ0} follows from \eqref{eq:isopineqG:2d:boundaryG0:2} and the second part of \eqref{eq:isopineqG:2d:ratio}.

\begin{proof}[Proof of \eqref{eq:deltat+1plt+1}]
To estimate the size of $\delta_{t+1}$, we proceed as in the proof of \eqref{eq:relationdelta}. 
The set $\pl_{t+1}\setminus\pl_t$ can be 
expressed as a disjoint union of boxes $S_j = \GG_0\cap(y_j + [0,2r_tL_t)^2)$, 
for some $y_1,\dots,y_k\in(r_tL_t)\cdot \Z^2$, such that 
every box is within $\ell^\infty$ distance $L_{t+1}$ from at most $36$ of the boxes. 
Since $|\partial_{\pl_s}\mathcal A'|< \frac{1}{12}\cdot\frac{L_{t+1}}{L_0}$ and the exterior vertex boundary of $\mathcal A'$ is $*$-connected, 
the set $\mathcal A'$ can be adjacent (in $\GG_0$) to at most $36$ such boxes 
(in fact, to at most $4\cdot 4 = 16$), 
which implies that 
\begin{equation}\label{eq:Gsdeltat+1}
|\delta_{t+1}| \leq 36\cdot8\frac{r_tL_t}{L_0} \leq \frac{1}{14}\cdot |\partial_{\pl_s}\mathcal A'|,
\end{equation}
where the last inequality follows from the assumption on $|\partial_{\pl_s}\mathcal A'|$. 

To estimate $|\partial_{\pl_{t+1}'}\mathcal A'|$ from below, 
we view $\mathcal A'$ as a disjoint union of subsets $\mathcal A_j'$ of $L_{t+1}$-boxes, and 
estimate from below the relative boundary of each $\mathcal A_j'$ in the corresponding box. 
By definition, $\pl_{t+1}$ is the disjoint union of boxes $\GG_0\cap(z_j + [0,L_{t+1})^2)$, $z_j\in\mathcal G_{K,s,t+1}(x_s)$. 
Let $\mathcal A_j'$ be the restriction of $\mathcal A'$ to the box $(z_j + [0,L_{t+1})^2)$.
By \eqref{eq:isopineqG:2d:boundaryGs} and \eqref{eq:caset<s}, for every $j$,
\[
|\mathcal A_j'|\leq |\mathcal A'| \leq 16\cdot |\partial_{\pl_s}\mathcal A'|^2
\leq \frac{1}{9} \cdot |\GG_0\cap[0,L_{t+1})^2|.
\]
By applying Lemma~\ref{l:boundaries} in each of $\GG_0\cap(z_j + [0,L_{t+1})^2)$, 
\begin{equation}\label{eq:boundaryt+1}
|\partial_{\pl_{t+1}'}\mathcal A'|\geq \sum_j|\partial_{\GG_0\cap(z_j+[0,L_{t+1})^d)}\mathcal A_j'|
\geq \frac{1}{7}\cdot \sum_j|\partial_{\GG_0}\mathcal A_j'|
\geq \frac{1}{7}\cdot |\partial_{\pl_s}\mathcal A'|.
\end{equation}
The combination of \eqref{eq:Gsdeltat+1} and \eqref{eq:boundaryt+1} gives \eqref{eq:deltat+1plt+1}.
\end{proof}

\bigskip

It remains to consider the case $|\partial_{\pl_s}\mathcal A'|\leq 14\cdot 36\cdot 8\cdot \frac{r_tL_t}{L_0}$. 
In this case $\partial_{\pl_s}\mathcal A'$ is comparable to the boundary of holes in $\pl_{t+1}\setminus \pl_t$.
We will show that 
\begin{equation}\label{eq:plt}
|\partial_{\pl_t'}\mathcal A'|\geq \frac{1}{2\cdot 10^5}\cdot |\partial_{\pl_s}\mathcal A'|.
\end{equation}
Together with Claim~\ref{cl:relationdelta}, 
\eqref{eq:plt} is sufficient for \eqref{eq:mathcalA':boundaryQsQ0}. Indeed, by \eqref{eq:boundaries:i}, 
\begin{equation}\label{eq:isopineqG:2d:boundaryG0:3}
|\partial_{\pl_0'}\mathcal A'| 
= |\partial_{\pl_t'}\mathcal A'| - \sum_{j=1}^t|\delta_j|
\geq \left(\frac{1}{2\cdot 10^5} - 3456\cdot \sum_{j=0}^\infty\frac{r_j}{l_j}\right)\cdot |\partial_{\pl_s}\mathcal A'|,
\end{equation}
and \eqref{eq:mathcalA':boundaryQsQ0} follows from \eqref{eq:isopineqG:2d:boundaryG0:3} and the second part of \eqref{eq:isopineqG:2d:ratio}.
\begin{proof}[Proof of \eqref{eq:plt}]
Since $\partial_{\pl_s}\mathcal A'$ is comparable to the boundary of holes in $\pl_{t+1}\setminus \pl_t$, 
this time we will look at $\mathcal A'$ on the scale $r_tL_t$. 
By Lemma~\ref{l:rect} and the assumption that $l_t$ is divisible by $r_t$, 
$\pl_t$ can be expressed as a disjoint union of boxes $(z_j + [0,r_tL_t)^2)$, $z_j\in(r_tL_t)\cdot\Z^2$. 
Let $\mathcal A_j'$ be the restriction of $\mathcal A'$ to the box $(z_j + [0,r_tL_t)^2)$. 
We will compare the boundary $\partial_{\pl_s}\mathcal A'$ to the relative boundary of $\mathcal A_j'$'s in the respective boxes. 

\medskip

If for all $j$, $|\mathcal A_j'|\leq \frac{1}{4}\cdot |\GG_0\cap [0,r_tL_t)^2|$, 
then by Lemma~\ref{l:boundaries} applied in each of $\GG_0\cap(z_j+[0,r_tL_t)^2)$,
\[
|\partial_{\GG_0\cap(z_j+[0,r_tL_t)^2)}\mathcal A_j'|
\geq \frac{1}{9}\cdot |\partial_{\GG_0}\mathcal A_j'|.
\]
Since the sets $\partial_{\GG_0\cap(z_j+[0,r_tL_t)^2)}\mathcal A_j'$ are disjoint subsets of $\partial_{\pl_t'}\mathcal A'$,  
\[
|\partial_{\pl_t'}\mathcal A'| \geq \sum_j|\partial_{\GG_0\cap(z_j+[0,r_tL_t)^2)}\mathcal A_j'|
\geq \frac{1}{9}\cdot \sum_j|\partial_{\GG_0}\mathcal A_j'|
\geq \frac{1}{9}\cdot |\partial_{\pl_s}\mathcal A'|,
\]
which implies \eqref{eq:plt}.

\begin{figure}[!tp]
\centering
\resizebox{7cm}{!}{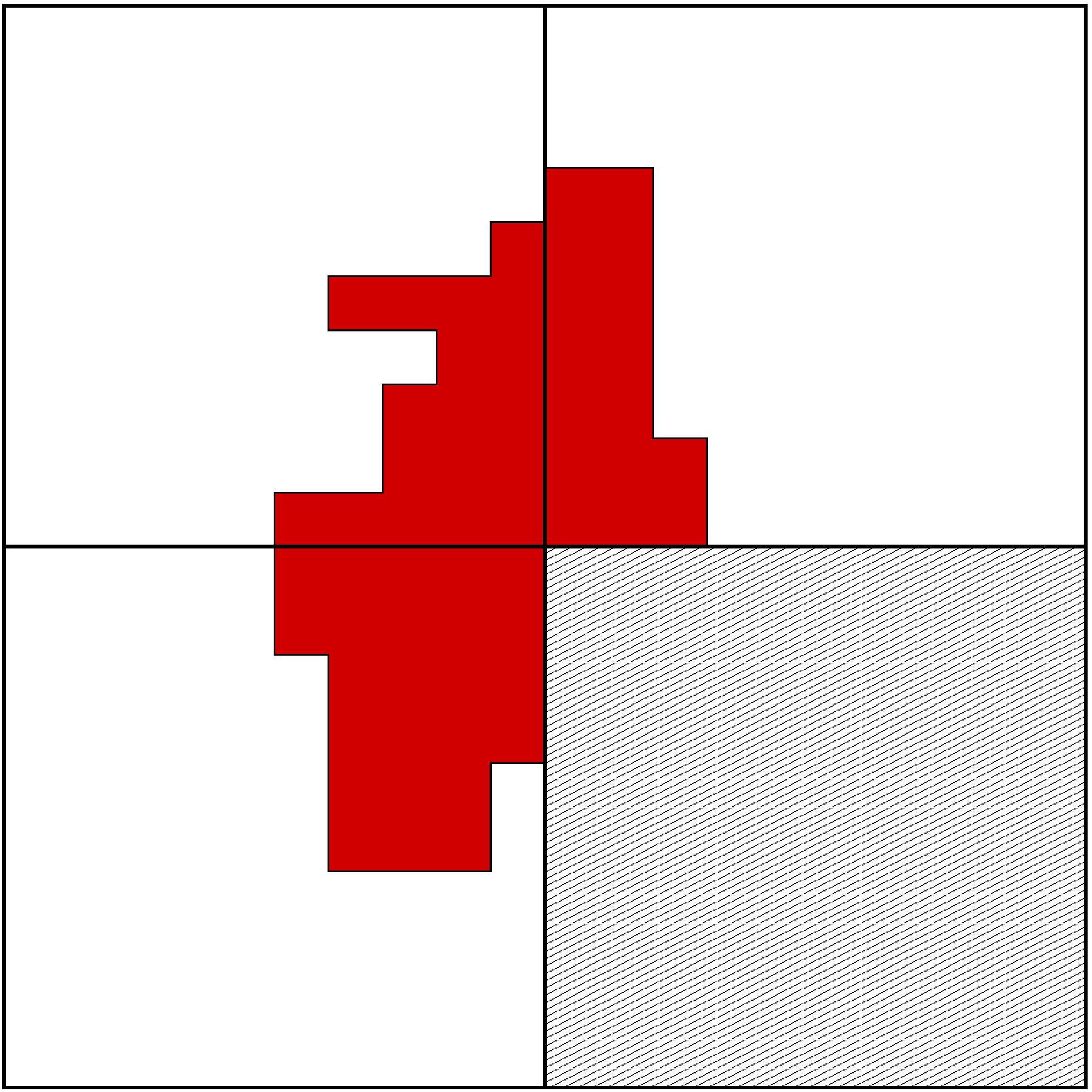}~\resizebox{7cm}{!}{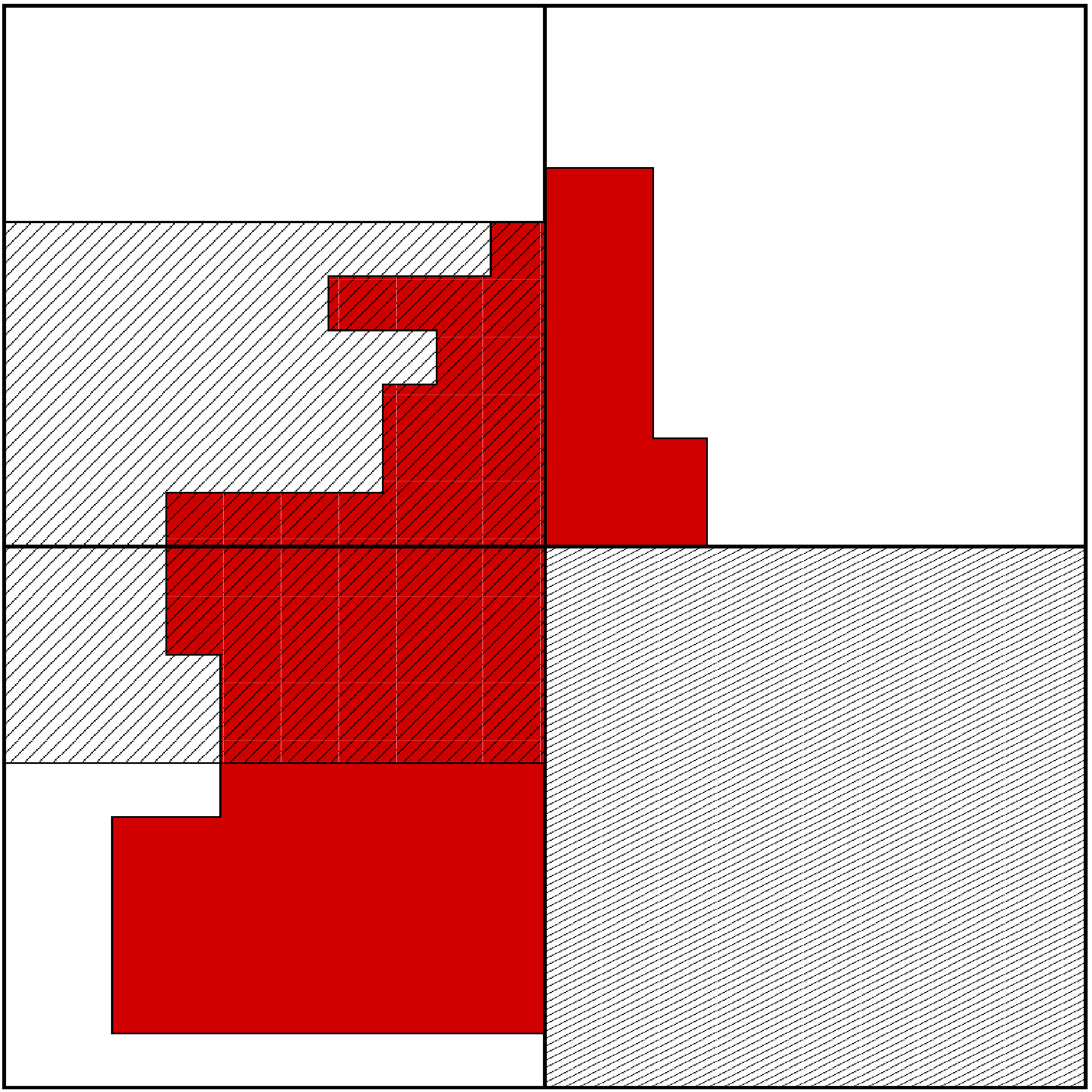}
\caption{The case when the boundary of $\mathcal A'$ is comparable to the boundary of holes on the scale of $\mathcal A'$. 
Two subcases: $\mathcal A'$ has small intersection with every $r_tL_t$-box (left) or 
large intersection with some $r_tL_t$-box (right). In the second subcase we can identify a box $(\widetilde z + [0,r_tL_t)^d)$ 
in which $\mathcal A'$ has non-trivial density.}
\label{fig:smallA}
\end{figure}

\medskip

On the other hand, if $|\mathcal A_j'|> \frac{1}{4}\cdot |\GG_0\cap [0,r_tL_t)^2|$ for at least one $j$, then 
there exists $\widetilde z\in\GG_t$ such that 
\begin{itemize}\itemsep0pt
\item 
$\GG_0\cap(\widetilde z + [0,r_tL_t)^2)\subset \pl_t$ and
\item
$\frac 14\cdot |\GG_0\cap [0,r_tL_t)^2|\leq |\mathcal A'\cap(\widetilde z + [0,r_tL_t)^2)|\leq \frac 34\cdot |\GG_0\cap [0,r_tL_t)^2|$.
\end{itemize}
Indeed, if none of $z_j$'s satisfies the two requirements, then there exist
$j_1$ and $j_2$ such that $|z_{j_1} - z_{j_2}|_\infty = r_tL_t$, 
$|\mathcal A_{j_1}'|> \frac{3}{4}\cdot |\GG_0\cap [0,r_tL_t)^2|$ and 
$|\mathcal A_{j_2}'|\leq \frac{1}{4}\cdot |\GG_0\cap [0,r_tL_t)^2|$. 
Then, $\widetilde z = \lambda\cdot z_{j_1} + (1-\lambda)\cdot z_{j_2}$ satisfies the two requirements for some $\lambda\in(0,1)$. 
(If $r_t$ is divisible by $2$, then one can take $\lambda = \frac 12$.)

By applying Lemma~\ref{l:boundaries} to $\GG_0\cap(\widetilde z + [0,r_tL_t)^2)$, 
\begin{multline*}
|\partial_{\pl_t'}\mathcal A'| \geq |\partial_{\GG_0\cap(\widetilde z + [0,r_tL_t)^2)}(\mathcal A'\cap(\widetilde z + [0,r_tL_t)^2)|
\geq \left(1 - \frac{\sqrt{3}}{2}\right)\cdot |\mathcal A'\cap(\widetilde z + [0,r_tL_t)^2)|^{\frac 12} \\
\geq \left(1 - \frac{\sqrt{3}}{2}\right)\cdot \frac{1}{2}\cdot \frac{r_tL_t}{L_0}
\geq\frac{1}{16}\cdot \frac{r_tL_t}{L_0}
\geq \frac{1}{16\cdot 14\cdot 36\cdot 8}\cdot |\partial_{\pl_s}\mathcal A'|,
\end{multline*}
where the last inequality follows from the assumption on $|\partial_{\pl_s}\mathcal A'|$. 
This inequality completes the proof of \eqref{eq:plt}. 
\end{proof}

\medskip

To summarize, the desired relation \eqref{eq:mathcalA':boundaryQsQ0} between $\partial_{\pl_0'}\mathcal A'$ and $\partial_{\pl_s}\mathcal A'$ follows from 
the three inequalities \eqref{eq:isopineqG:2d:boundaryG0:1} (the boundary $\partial_{\pl_s}\mathcal A'$ is macroscopic), \eqref{eq:isopineqG:2d:boundaryG0:2} 
(the boundary $\partial_{\pl_s}\mathcal A'$ is small, but much bigger than the boundaries of holes on the given scale), 
and \eqref{eq:isopineqG:2d:boundaryG0:3} (the boundary $\partial_{\pl_s}\mathcal A'$ is small and comparable to the boundaries of holes on the given scale).
The proof of Lemma~\ref{l:isopineqG:2d} is complete. 
\end{proof}

\medskip

\begin{remark}\label{rem:extensiond>2}
The only step in the proof of Lemma~\ref{l:isopineqG:2d} that uses (crucially!) the assumption $d=2$ is 
the derivation of \eqref{eq:deltai:Ni>100}. More precisely, the 
fact that the boundary of a set is well approximated by simple paths. 
In higher dimensions this is clearly not the case (the dimension of the boundary is generally bigger than 
the dimension of a simple path), and the above argument breaks down. See Figure~\ref{fig:deltai}.
\end{remark}

\subsection{Isoperimetric inequality in any dimension for large enough subsets}

In this section we prove the following theorem, which includes Theorem~\ref{thm:isop:pl} as a special case. 
\begin{theorem}\label{thm:isopframe}
Let $d\geq 2$, $c>0$. Let $l_n$ and $r_n$, $n\geq 0$, be integer sequences satisfying assumptions of Lemma~\ref{l:isopineqG:2d} and 
such that 
\begin{equation}\label{eq:isopframe:ratio}
\prod_{i=0}^\infty\left(1 - \left(\frac{4r_i}{l_i}\right)^2\right) \geq e^{-\frac{1}{16(d-1)}}\quad\text{and}\quad
\prod_{i=0}^\infty\left(1 - \left(\frac{4r_i}{l_i}\right)^d\right) \geq \frac{1 - \frac{1}{2^{d+2}}}{1 - \frac{1}{2^{d+3}}}.
\end{equation}
Then for any integers $s\geq 0$, $L_0\geq 1$, and $K\geq 1$, $x_s\in\GG_s$, 
and two families of events $\seedde$ and $\seedin$,  
if all the vertices in $\GG_s\cap Q_{K,s}(x_s)$ are $s$-good, then
any $\mathcal A\subseteq \pl_{K,s,0}(x_s)$ with
\[
\min\left\{ c\cdot |Q_{K,s}\cap\GG_0|,~\left(\frac{L_s}{L_0}\right)^{d^2}\right\} \leq |\mathcal A|\leq \frac 12\cdot |Q_{K,s}\cap \GG_0|
\]
satisfies
\[
|\partial_{\pl_{K,s,0}(x_s)} \mathcal A|\geq 
\frac{c^2}{2d\cdot 32^d\cdot 27^d\cdot 10^6}\cdot \left(1 - \left(\frac 23\right)^{\frac 1d}\right)\cdot \left(1 - e^{-\frac{1}{16(d-1)}}\right)\cdot |\mathcal A|^{\frac{d-1}{d}}.
\]
\end{theorem}
\begin{proof}[Proof of Theorem~\ref{thm:isopframe}]
Fix $s\geq 0$ and $K\geq 1$ integers, $x_s\in\GG_s$, 
and assume that all the vertices in $\GG_s\cap Q_{K,s}(x_s)$ are $s$-good. 
Take $\mathcal A \subseteq \pl_{K,s,0}(x_s)$ such that $|\mathcal A|\leq \frac 12\cdot |Q_{K,s}\cap\GG_0|$. 

\medskip

We consider separately the cases $|\mathcal A|\geq c\cdot |Q_{K,s}\cap\GG_0|$ and $|\mathcal A|\geq \left(\frac{L_s}{L_0}\right)^{d^2}$. 
In fact, we will use the result for the first case to prove the result for the second. 

In the first case, we use Corollaries~\ref{cor:selection1} and \ref{cor:selection2} to the selection lemma from Section~\ref{sec:selection} 
to identify a large number of disjoint two dimensional slices in the ambient box $Q_{K,s}\cap\GG_0$ which  
on the one hand have a small non-empty intersection with $\mathcal A$, and on the other, all together contain a positive fraction of the volume of $\mathcal A$. 
We estimate the boundary of $\mathcal A$ in each of the slices using the two dimensional isoperimetric inequality of Lemma~\ref{l:isopineqG:2d}. 
Since the slices are pairwise disjoint, we can estimate the boundary of $\mathcal A$ by the sum of the boundaries of $\mathcal A$ in each of the slices. 

In the second case, we consider a coarse graining of $\mathcal A$ by densely occupied $L_s$-boxes. 
If the number of densely occupied $L_s$-boxes is small, then $\mathcal A$ is scattered in $Q_{K,s}\cap\GG_0$ and 
has big boundary. If, on the other hand, the number of densely occupied $L_s$-boxes is big, then 
the set of such boxes has large boundary (the poorly occupied boxes adjacent to some densely occupied ones). 
Each pair of adjacent densely and poorly occupied $L_s$-boxes are contained in a $2L_s$-box. 
Vertices from $\mathcal A$ occupy a non-trivial fraction of vertices in this $2L_s$-box. 
Thus, we can estimate the boundary of $\mathcal A$ restricted to this box using the first part of the theorem. 
By summing over all pairs of adjacent densely and poorly occupied $L_s$-boxes we obtain a desired lower bound on the size of the boundary of $\mathcal A$. 

\bigskip

\begin{figure}[!tp]
\centering
\resizebox{15cm}{!}{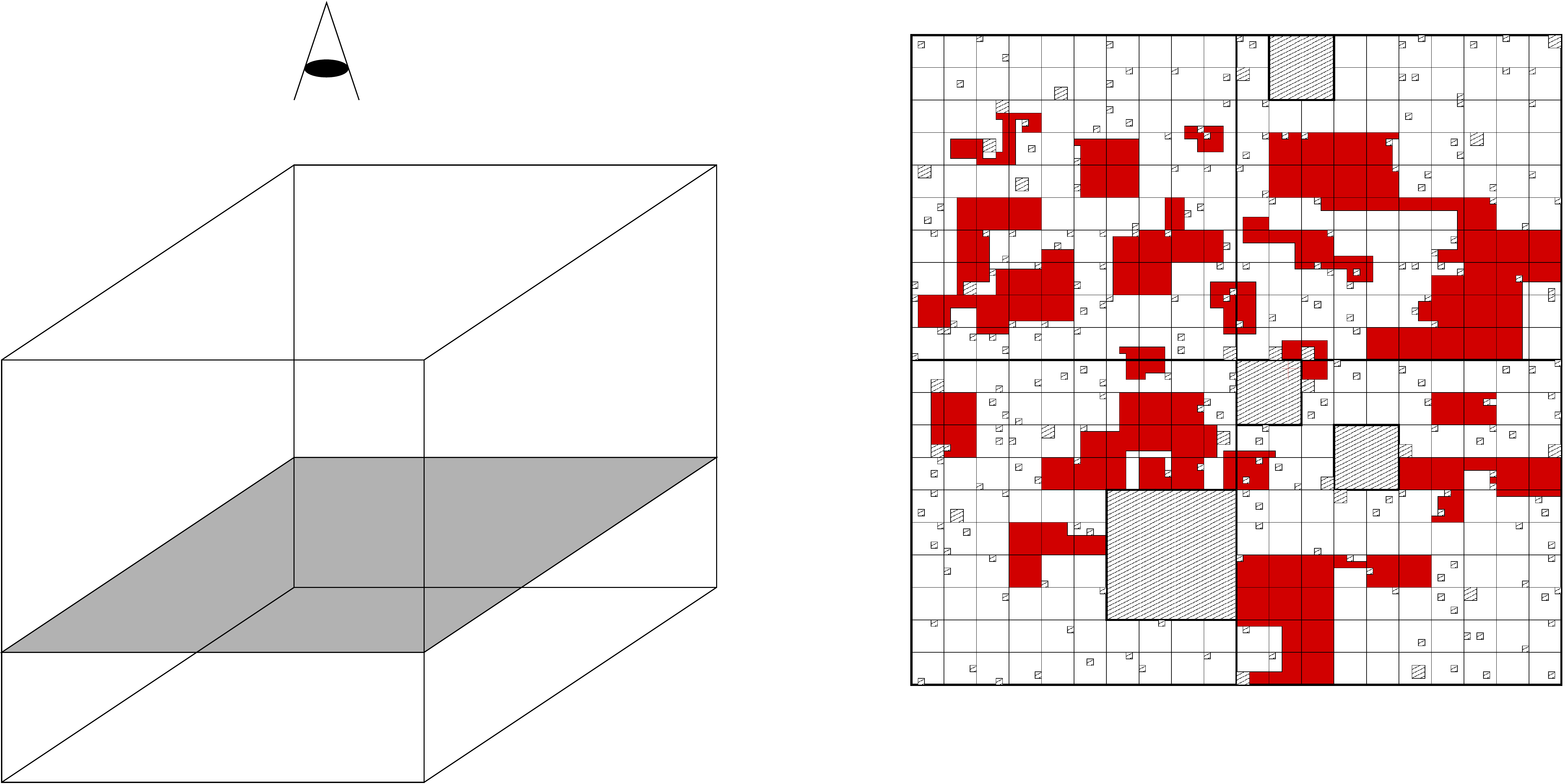}
\caption{Left: a two dimensional slice $S_i$. Right: perforation $\pl_{K,s,0}(x_s)\cap S_i$ of $S_i$ and the intersection of $\mathcal A$ with $S_i$.}
\label{fig:slice}
\end{figure}

We first consider the case $|\mathcal A|\geq c\cdot |Q_{K,s}\cap\GG_0|$. 
By Corollaries~\ref{cor:selection1} and \ref{cor:selection2}, 
there exist 
\[
\geq \frac{c}{2\cdot 9^{d-2}}\cdot \left(\frac{KL_s}{L_0}\right)^{d-2}
\] 
two dimensional subrectangles $S_i$ in $Q_{K,s}\cap\GG_0$ (see Figure~\ref{fig:slice}) such that for all $i$, 
\[
|\mathcal A\cap S_i|\geq \frac{c}{2\cdot 9^{d-2}}\cdot \left(\frac{KL_s}{L_0}\right)^2\quad\text{and}\quad
|\mathcal A\cap S_i|\leq e^{-\frac{1}{8(d-1)}}\cdot \left(\frac{KL_s}{L_0}\right)^2.
\]
By Lemma~\ref{l:rect} (applied to the perforation $\pl_{K,s,0}(x_s)\cap S_i$ of $S_i$) and the first part of \eqref{eq:isopframe:ratio}, 
$|\pl_{K,s,0}(x_s)\cap S_i|\geq e^{-\frac{1}{16(d-1)}}\cdot \left(\frac{KL_s}{L_0}\right)^2$, which implies that for all $i$, 
\[
|\mathcal A\cap S_i|\leq e^{-\frac{1}{16(d-1)}}\cdot|\pl_{K,s,0}(x_s)\cap S_i|.
\]
We apply the two dimensional isoperimetric inequality of Lemma~\ref{l:isopineqG:2d} and Remark~\ref{rem:isop} 
to each of the sets $\mathcal A\cap S_i$ in $\pl_{K,s,0}(x_s)\cap S_i$, and obtain that for all $i$, 
\begin{eqnarray*}
|\partial_{\pl_{K,s,0}(x_s)\cap S_i}(\mathcal A\cap S_i)| &\geq &\frac{1}{10^6}\cdot \left(1 - e^{-\frac{1}{16(d-1)}}\right)\cdot |\mathcal A\cap S_i|^{\frac 12}\\
&\geq &\frac{1}{10^6}\cdot \left(1 - e^{-\frac{1}{16(d-1)}}\right)\cdot\frac{c}{2\cdot 3^{d-2}}\cdot \frac{KL_s}{L_0}.
\end{eqnarray*}
Since all $\partial_{\pl_{K,s,0}(x_s)\cap S_i}(\mathcal A\cap S_i)$ are disjoint subsets of $\partial_{\pl_{K,s,0}(x_s)} \mathcal A$, 
\begin{eqnarray}
|\partial_{\pl_{K,s,0}(x_s)} \mathcal A| 
&\geq &\sum_i |\partial_{\pl_{K,s,0}(x_s)\cap S_i}(\mathcal A\cap S_i)|\nonumber\\
&\geq
&\frac{c}{2\cdot 9^{d-2}}\cdot \left(\frac{KL_s}{L_0}\right)^{d-2}\cdot 
\frac{1}{10^6}\cdot \left(1 - e^{-\frac{1}{16(d-1)}}\right)\cdot\frac{c}{2\cdot 3^{d-2}}\cdot \frac{KL_s}{L_0}\nonumber\\ 
&\geq
&\frac{c^2}{4\cdot 27^{d-2}\cdot 10^6}\cdot \left(1 - e^{-\frac{1}{16(d-1)}}\right)\cdot|\mathcal A|^{\frac{d-1}{d}}.\label{eq:boundaryA1}
\end{eqnarray}
This completes the proof of Theorem~\ref{thm:isopframe} for sets with $|\mathcal A|\geq c\cdot |Q_{K,s}\cap\GG_0|$. 

\bigskip

Next, we consider the case $|\mathcal A|\geq \left(\frac{L_s}{L_0}\right)^{d^2}$. 
Let 
\[
\mathbb A_s = \left\{x\in \GG_s ~:~\mathcal A\cap (x + [0,L_s)^d)\neq\emptyset\right\}
\]
be the set of bottom-left corners of $L_s$-boxes which contain a vertex from $\mathcal A$. 
Note that $|\mathbb A_s| \geq |\mathcal A|\cdot \left(\frac{L_0}{L_s}\right)^d$. 
We also define the subset $\widetilde{\mathbb A}_s$ of $\mathbb A_s$ corresponding to the densely occupied boxes, 
\[
\widetilde {\mathbb A}_s = \left\{x\in\GG_s~:~|\mathcal A\cap (x + [0,L_s)^d)|\geq \frac 34\cdot \left(\frac{L_s}{L_0}\right)^d\right\}.
\]

We consider separately the cases when $|\widetilde {\mathbb A}_s|\geq \frac 12\cdot |\mathbb A_s|$ and $|\widetilde {\mathbb A}_s|\leq \frac 12\cdot |\mathbb A_s|$. 

\medskip

We first consider the case $|\widetilde {\mathbb A}_s|\geq \frac 12\cdot |\mathbb A_s|$, i.e., the number of densely occupied boxes is large. 

Since
$\frac 34\cdot \left(\frac{L_s}{L_0}\right)^d\cdot |\widetilde {\mathbb A}_s| \leq |\mathcal A|\leq \frac 12\cdot |Q_{K,s}\cap\GG_0|$, 
\[
|\widetilde {\mathbb A}_s| \leq \frac 23\cdot |Q_{K,s}\cap\GG_0|\cdot \left(\frac{L_0}{L_s}\right)^d = \frac 23\cdot |Q_{K,s}(x_s)\cap\GG_s|.
\]
By applying Lemma~\ref{l:boundaries} to $\widetilde {\mathbb A}_s\subset Q_{K,s}(x_s)\cap\GG_s$, we get
\begin{equation}\label{eq:isop:widetildeAs}
|\partial_{Q_{K,s}(x_s)\cap\GG_s}\widetilde {\mathbb A}_s| \geq \left(1 - \left(\frac 23\right)^{\frac 1d}\right)\cdot |\widetilde {\mathbb A}_s|^{\frac{d-1}{d}}. 
\end{equation}
Next, we zoom in onto the boundary $\partial_{Q_{K,s}(x_s)\cap\GG_s}\widetilde {\mathbb A}_s$. 
Take any pair $x\in\widetilde {\mathbb A}_s$ and $y\in (Q_{K,s}(x_s)\cap\GG_s)\setminus\widetilde{\mathbb A}_s$ from $\partial_{Q_{K,s}(x_s)\cap\GG_s}\widetilde {\mathbb A}_s$. 
Note that 
\[
|\mathcal A\cap (x + [0,L_s)^d)| \geq \frac 34\cdot \left(\frac{L_s}{L_0}\right)^d 
\quad\text{and}\quad
|\mathcal A\cap (y + [0,L_s)^d)| < \frac 34\cdot \left(\frac{L_s}{L_0}\right)^d.
\]
Take a box $(z + [0,2L_s)^d)$ in $Q_{K,s}(x_s)\cap\GG_0$ containing both $(x+[0,L_s)^d)$ and $(y + [0,L_s)^d)$, 
where $z\in(Q_{K,s}(x_s)\cap\GG_s)$.
Note that $\mathcal A$ occupies a non-trivial fraction of vertices in $(z + [0,2L_s)^d)$. More precisely, 
\[
\frac{3}{2^{d+2}}\cdot |(z+[0,2L_s)^d)\cap\GG_0|\leq  |\mathcal A\cap (z+[0,2L_s)^d)|\leq \left(1 - \frac{1}{2^{d+2}}\right)\cdot |(z+[0,2L_s)^d)\cap\GG_0|.
\]
Moreover, all the vertices in $(z+[0,2L_s)^d\cap\GG_s$ are $s$-good. 
We are in a position to apply the first part of the theorem to $\mathcal A\cap (z+[0,2L_s)^d)$ in $(z+[0,2L_s)^d)$. 
Combining the upper bound on $|\mathcal A\cap (z+[0,2L_s)^d)|$ with 
the lower bound on the volume of the perforation $\pl_{2,s,0}(z) = \pl_{K,s,0}(x_s)\cap (z+[0,2L_s)^d)$ given by Lemma~\ref{l:rect} 
and the second part of the assumption \eqref{eq:isopframe:ratio}, 
we obtain that 
\[
|\mathcal A\cap (z+[0,2L_s)^d)|\leq \left(1 - \frac{1}{2^{d+3}}\right)\cdot |\pl_{K,s,0}(x_s)\cap (z+[0,2L_s)^d)|.
\] 
Therefore, by the first part of the theorem (with $c = \frac{3}{2^{d+2}}$) 
applied to the subset $\mathcal A\cap (z+[0,2L_s)^d)$ of $\pl_{2,s,0}(z)$ 
and Remark~\ref{rem:isop},
\begin{multline*}
|\partial_{\pl_{K,s,0}(x_s)\cap (z+[0,2L_s)^d)}(\mathcal A\cap (z+[0,2L_s)^d)|\\ 
\geq 
\frac{1}{2^{d+3}}\cdot
\frac{9}{4\cdot 4^{d+2}\cdot 27^{d-2}\cdot 10^6}\cdot \left(1 - e^{-\frac{1}{16(d-1)}}\right)\cdot|\mathcal A\cap (z+[0,2L_s)^d)|^{\frac{d-1}{d}}\\
\geq \frac{3}{4}\cdot \frac{9}{8^{d+3}\cdot 27^{d-2}\cdot 10^6}\cdot \left(1 - e^{-\frac{1}{16(d-1)}}\right)\cdot \left(\frac{L_s}{L_0}\right)^{d-1}.
\end{multline*}
This inequality gives us an estimate on the part of the boundary $\partial_{\pl_{K,s,0}(x_s)} \mathcal A$ contained in $(z + [0,2L_s)^d)$ 
for each $z\in (Q_{K,s}(x_s)\cap\GG_s)$ such that the cube $(z + [0,2L_s)^d)$ contains an overcrowded and undercrowded adjacent $L_s$-boxes 
$(x + [0,L_s)^d)$ and $(y+[0,L_s)^d)$ with $x\in\widetilde {\mathbb A}_s$ and $y\in (Q_{K,s}(x_s)\cap\GG_s)\setminus\widetilde{\mathbb A}_s$. 
By \eqref{eq:isop:widetildeAs}, the total number of such $z$'s is 
\[
\geq \frac{1}{d2^{d-1}}\cdot |\partial_{Q_{K,s}(x_s)\cap\GG_s}\widetilde {\mathbb A}_s|
\geq \frac{1}{d2^{d-1}}\cdot  \left(1 - \left(\frac 23\right)^{\frac 1d}\right)\cdot |\widetilde {\mathbb A}_s|^{\frac{d-1}{d}}, 
\]
where the factor $\frac{1}{d2^{d-1}}$ counts for possible overcounting, since 
every cube $(z+[0,2L_s)^d)$, $z\in\GG_s$, contains at most $d2^{d-1}$ pairs $x,y$ with $\{x,y\}\in\partial_{Q_{K,s}(x_s)\cap\GG_s}\widetilde {\mathbb A}_s$.

Moreover, every edge from $\partial_{\pl_{K,s,0}(x_s)} \mathcal A$ belongs to at most $2^d$ cubes $(z+[0,2L_s)^d)$, $z\in\GG_s$. 
Thus, 
\[
|\partial_{\pl_{K,s,0}(x_s)} \mathcal A|\geq \frac{1}{2^d}\cdot\sum_{z\in\GG_s}|\partial_{\pl_{K,s,0}(x_s)\cap (z+[0,2L_s)^d)}(\mathcal A\cap (z+[0,2L_s)^d)|.
\]
By putting all the estimates together, we obtain that 
\begin{multline}\label{eq:isopframe:gamma}
|\partial_{\pl_{K,s,0}(x_s)} \mathcal A|\geq
\frac{1}{2^d}\cdot\sum_{z\in\GG_s}|\partial_{\pl_{K,s,0}(x_s)\cap (z+[0,2L_s)^d)}(\mathcal A\cap (z+[0,2L_s)^d)|\\
\geq \frac{1}{2^d}\cdot 
\frac{1}{d2^{d-1}}\cdot  \left(1 - \left(\frac 23\right)^{\frac 1d}\right)\cdot |\widetilde {\mathbb A}_s|^{\frac{d-1}{d}}\cdot 
\frac{3}{4}\cdot \frac{9}{8^{d+3}\cdot 27^{d-2}\cdot 10^6}\cdot \left(1 - e^{-\frac{1}{16(d-1)}}\right)\cdot \left(\frac{L_s}{L_0}\right)^{d-1}\\
\geq \frac{1}{2d\cdot 32^d\cdot 27^d\cdot 10^6}\cdot \left(1 - \left(\frac 23\right)^{\frac 1d}\right)\cdot \left(1 - e^{-\frac{1}{16(d-1)}}\right)\cdot
|\mathcal A|^{\frac{d-1}{d}},
\end{multline}
where the last inequality follows from the case assumption $|\widetilde {\mathbb A}_s|\geq \frac 12\cdot |\mathbb A_s|\geq\frac12\cdot|\mathcal A|\cdot \left(\frac{L_0}{L_s}\right)^d$. 

\bigskip

It remains to consider the case $|\widetilde {\mathbb A}_s|\leq \frac 12\cdot |\mathbb A_s|$. 
In this case, $\mathcal A$ is scattered in $Q_{K,s}(x_s)\cap\GG_0$, and should have big boundary. 
Indeed, for each $x\in \mathbb A_s\setminus \widetilde{\mathbb A}_s$, 
\[
1\leq |\mathcal A\cap(x+[0,L_s)^d|< \frac 34\cdot \left(\frac{L_s}{L_0}\right)^d.
\] 
By the lower bound on the volume of the perforation $\pl_{1,s,0}(x) = \pl_{K,s,0}(x_s)\cap (x+[0,L_s)^d)$ given in Lemma~\ref{l:rect} and the second part of \eqref{eq:isopframe:ratio}, 
\[
|\pl_{K,s,0}(x_s)\cap (x+[0,L_s)^d|
\geq \frac{1 - \frac{1}{2^{d+2}}}{1 - \frac{1}{2^{d+3}}}\cdot \left(\frac{L_s}{L_0}\right)^d
\geq\frac 34\cdot \left(\frac{L_s}{L_0}\right)^d.
\] 
Thus, $(x+[0,L_s)^d)$ contains vertices from both $\mathcal A$ and $\pl_{K,s,0}(x_s)\setminus\mathcal A$. 
By Lemma~\ref{l:rect}, $\pl_{1,s,0}(x) = \pl_{K,s,0}(x_s)\cap (x+[0,L_s)^d)$ is connected in $\GG_0$, thus it contains an edge from $\partial_{\pl_{K,s,0}(x_s)}\mathcal A$. 
Since all $(x+[0,L_s)^d)$, $x\in \mathbb A_s\setminus \widetilde{\mathbb A}_s$ are disjoint, we conclude that 
\begin{equation}\label{eq:isopframe:gamma2}
|\partial_{\pl_{K,s,0}(x_s)}\mathcal A|\geq |\mathbb A_s\setminus \widetilde{\mathbb A}_s|\geq \frac 12\cdot |\mathbb A_s|
\geq \frac 12\cdot |\mathcal A|\cdot \left(\frac{L_0}{L_s}\right)^d
\geq \frac 12\cdot |\mathcal A|^{\frac{d-1}{d}},
\end{equation}
where the last inequality follows from the case assumption. 

The proof of Theorem~\ref{thm:isopframe} in the case $|\mathcal A|\geq \left(\frac{L_s}{L_0}\right)^{d^2}$ is complete by \eqref{eq:isopframe:gamma} and \eqref{eq:isopframe:gamma2}. 
\end{proof}

\begin{remark}\label{rem:isopframe:conditions}
We believe that Theorem~\ref{thm:isopframe} holds for all $\mathcal A$ with $|\mathcal A|\leq \frac12\cdot |Q_{K,s}\cap\GG_0|$. 
With a more involved proof, we can relax the assumption $|\mathcal A|\geq \left(\frac{L_s}{L_0}\right)^{d^2}$ of Theorem~\ref{thm:isopframe} 
to $|\mathcal A|\geq \left(\frac{L_s}{L_0}\right)^{2d}$. 
Since this does not give us the result for all $\mathcal A$, and the current statement of Theorem~\ref{thm:isopframe} suffices for the applications in this paper, 
we do not include this proof here. 
\end{remark}

\bigskip

\appendix

\section{Proofs of Theorems~\ref{thm:hk:grad}--\ref{thm:gf:asymp}}\label{sec:proofs}

In this section we give proof sketches of Theorems~\ref{thm:hk:grad}, \ref{thm:gf:bounds}, \ref{thm:hf:d+1}, \ref{thm:localclt}, and \ref{thm:gf:asymp}.
Their proofs are straightforward adaptations of main results in \cite{BH09,BDCKY14} from Bernoulli percolation to our setup. 

\begin{proof}[Proof of Theorem~\ref{thm:hk:grad}]
The proof is essentially the same as that of \cite[Theorem~6]{BDCKY14}. 
The only minor care that is required comes from the fact that the bound \eqref{eq:hk:T0} is not stretched exponential. 
Since this fact is used several times, we provide a general outline of the proof. 
As in the proof of \cite[Theorem~6]{BDCKY14}, by stationarity \p{} and the ergodicity of $\set_\infty$ with respect to the shift by $X_1$ (see, e.g., \cite[Theorem~3.1]{BergerBiskup}),
it suffices to prove that 
\[
\mathbb E^u\left[\left(p_{2n}(0,x) - p_{2n-1}(X_1,x)\right)^2\cdot \mathds{1}_{x\in\set_\infty}\right]\leq \frac{C}{n^{d+1}}\cdot e^{-c\frac{\dist_{\Z^d}(0,x)^2}{n}},
\]
where $C$ and $c$ only depend on $d$ and $u$. 
If $\dist_{\Z^d}(0,x)\geq n^{\frac12}(\log n)^{\frac{1 + \constS}{2}}$, where $\constS$ is defined in \eqref{eq:funcS}, 
then by the general upper bound on the heat kernel (see, e.g., \cite[(1.5)]{Barlow}), 
\[
\mathbb E^u\left[\left(p_{2n}(0,x) - p_{2n-1}(X_1,x)\right)^2\cdot \mathds{1}_{x\in\set_\infty}\right]\leq
C\cdot e^{-c\frac{\dist_{\Z^d}(0,x)^2}{n}}
\leq \frac{C'}{n^{d+1}}\cdot e^{-c'\frac{\dist_{\Z^d}(0,x)^2}{n}}.
\]
Thus, we can assume that $\dist_{\Z^d}(0,x)\leq n^{\frac12}(\log n)^{\frac{1 + \constS}{2}}$.

Let $N = N(\omega) = \max\left\{T_{\scriptscriptstyle {\mathrm{hk}}}(y):y\in\ballZ_{\Z^d}(0,n)\right\}$.
By \eqref{eq:hk:T0},  
\begin{multline*}
\mathbb E^u\left[\left(p_{2n}(0,x) - p_{2n-1}(X_1,x)\right)^2\cdot \mathds{1}_{x\in\set_\infty}\cdot \mathds{1}_{N(\omega)\geq n}\right]\leq
\mathbb P^u\left[N(\omega)\geq n\right]\\
\leq Cn^d\cdot e^{-c \cdot (\log n)^{1+\constS}}
\leq \frac{C'}{n^{d+1}}\cdot e^{-c' \cdot (\log n)^{1+\constS}}
\leq \frac{C'}{n^{d+1}}\cdot e^{-c'\frac{\dist_{\Z^d}(0,x)^2}{n}}.
\end{multline*}
It remains to bound $\mathbb E^u\left[\left(p_{2n}(0,x) - p_{2n-1}(X_1,x)\right)^2\cdot \mathds{1}_{x\in\set_\infty}\cdot \mathds{1}_{N(\omega)\leq n}\right]$.
As in \cite[Section~2]{BDCKY14}, define the quenched entropy of the simple random walk on $\set_\infty$ by 
$\mathbf H_n = \sum_x\phi(p_{\set_\infty,n}(0,x))$, where $\phi(0) = 0$ and $\phi(t) =-t\log t$ for $t>0$, 
and the mean entropy by $H_n = \mathbb E^u[\mathbf H_n]$.
By a general argument in the proof of \cite[Theorem~6]{BDCKY14}, the heat kernel upper bound \eqref{eq:hk:ub} implies that
\[
\mathbb E^u\left[\left(p_{2n}(0,x) - p_{2n-1}(X_1,x)\right)^2\cdot \mathds{1}_{x\in\set_\infty}\cdot \mathds{1}_{N(\omega)\leq n}\right]
\leq \left(H_n-H_{n-1}\right)\cdot \frac{C}{n^d}\cdot e^{-c\frac{\dist_{\Z^d}(0,x)^2}{n}}.
\]
The proof of \cite[Theorem~6]{BDCKY14} is completed by showing in \cite[Lemma~20]{BDCKY14} that $H_n -H_{n-1}\leq \frac{C}{n}$.  
Thus, in order to finish the proof of Theorem~\ref{thm:hk:grad}, it suffices to prove that $H_n -H_{n-1}\leq \frac{C}{n}$ in our setting too. 
This is a simple consequence of Theorem~\ref{thm:hk:dZd}. 
Indeed, by writing $\mathbf H_n$ as the sums over $x$ with $\dist_{\Z^d}(0,x)^{\frac 32}\leq n$ and $\dist_{\Z^d}(0,x)^{\frac 32}\geq n$, 
applying \eqref{eq:hk:ub} and \eqref{eq:hk:lb} to the summands in the first sum, and showing smallness of the second sum by using, for instance, the general upper bound on the heat kernel 
(see, e.g., \cite[(1.5)]{Barlow}), we prove that for all $n\geq T_{\scriptscriptstyle {\mathrm{hk}}}(0)$, $\mathbf H_n = \frac d2\log n + O(1)$. 
For $n\leq T_{\scriptscriptstyle {\mathrm{hk}}}(0)$, we use the crude bound $\mathbf H_n \leq d\log (2n)$ (see the proof below \cite[(25)]{BDCKY14}). 
By integrating $\mathbf H_n$ and using \eqref{eq:hk:T0}, we get that $H_n = \frac d2\log n+ O(1)$, 
which implies that $H_n - H_{\lfloor n/2\rfloor}\leq C$ for some $C$. 
Since $H_n-H_{n-1}$ is decreasing by \cite[Corollary~10]{BDCKY14}, we conclude that $H_n-H_{n-1}\leq \frac{2C}{n}$, 
finishing the proof of Theorem~\ref{thm:hk:grad}.
\end{proof}

\medskip

\begin{proof}[Proof of Theorem~\ref{thm:gf:bounds}]
The proof of Theorem~\ref{thm:gf:bounds} is literally the same as the proof of \cite[Theorem~1.2(a)]{BH09}. 
For the upper bound, one splits the Green function into the integrals over $[0,\min\{T_{\scriptscriptstyle {\mathrm{hk}}}(x),T_{\scriptscriptstyle {\mathrm{hk}}}(y)\}]$ 
and $[\min\{T_{\scriptscriptstyle {\mathrm{hk}}}(x),T_{\scriptscriptstyle {\mathrm{hk}}}(y)\},\infty)$. 
Using general bounds on the heat kernel (see \cite[(6.4) and (6.5)]{BH09}), one shows that the first integral is $o(\dist_{\Z^d}(x,y)^{2-d})$, 
and by \eqref{eq:hk:ub}, the second integral is bounded by $C\dist_{\Z^d}(x,y)^{2-d}$. 
For the lower bound, one estimates the Green function from below by the integral of heat kernel over $[\dist_{\Z^d}(x,y)^2,\infty)$, 
applies \eqref{eq:hk:lb}, and arrives at the desired bound. 
\end{proof}

\medskip

\begin{proof}[Proof of Theorem~\ref{thm:hf:d+1}]
The proof of Theorem~\ref{thm:hf:d+1} is identical to the one of \cite[Theorem~5]{BDCKY14}. 
The constant functions and the projections of $x+\chi(x)$ (see Theorem~\ref{thm:qip}(a)) on coordinates of $\Z^d$ 
are independent harmonic functions with at most linear growth. Thus, the dimension of such functions is at least $(d+1)$. 
It remains to show that the above functions form a basis.
Let $h$ be a harmonic function $h$ on $\set_\infty$ with at most linear growth and $h(0) = 0$, and 
assume that it is extended on $\R^d$ (see above \cite[Proposition~19]{BDCKY14}). 
By Theorem~\ref{thm:vgb:main} and the upper bound on the heat kernel \eqref{eq:hk:ub}, the proof of \cite[Proposition~19]{BDCKY14} goes through without any changes in our setting, implying that 
the sequence $h_n(\cdot) = \frac 1n h(n\cdot)$ is uniformly bounded and equicontinuous on compacts. 
Thus, there exists a sequence $n_k$ such that $h_{n_k}$ converges uniformly on compact sets to a continuous function $\widetilde h$. 
By using the quenched invariance principle of Theorem~\ref{thm:qip}, one obtains by repeating the proof of \cite[Theorem~5]{BDCKY14} that 
$\widetilde h$ is harmonic in $\R^d$. Since $\widetilde h$ has at most linear growth and $\widetilde h(0) = 0$, it is linear. 
Therefore, the function $f(x) = h(x) - \widetilde h(x+\chi(x))$ is harmonic on $\set_\infty$ and for every $\varepsilon>0$ and all large enough $k$, 
$|f(x)|\leq \varepsilon n_k$ for all $x\in\ballZ_{\set_\infty}(0,n_k/\varepsilon)$. 
By \eqref{eq:hk:ub}, $\mathrm{E}_{\set_\infty,0}\left[f(X_{n_k^2})^2\right]\leq \varepsilon n_k^2$ for all large $k$. 
The proof of \cite[Theorem~5]{BDCKY14} is finished by applying \cite[Corollary~21]{BDCKY14} which states that $f$ must be constant. 
The proof of \cite[Corollary~21]{BDCKY14} is rather general and only uses the fact that the mean entropy $H_n$ (see the proof of Theorem~\ref{thm:hk:grad}) 
satisfies $H_n-H_{n-1}\leq \frac{C}{n}$. 
We already proved this bound in the proof of Theorem~\ref{thm:hk:grad}. 
Thus, \cite[Corollary~21]{BDCKY14} holds in our setting, and we conclude that $f$ must be constant. 
The proof is complete. 
\end{proof}

\medskip

\begin{proof}[Proof of Theorem~\ref{thm:localclt}]
Theorem~\ref{thm:localclt} was proved in the case of supercritical Bernoulli percolation in \cite[Theorem~1.1]{BH09} 
by first providing general assumptions \cite[Assumption~4.4]{BH09} for the local limit theorem on infinite subgraphs of $\Z^d$ 
(see \cite[Theorems~4.5 and 4.6]{BH09}), and then verifying these assumptions for the infinite cluster of Bernoulli percolation. 
\cite[Assumption~4.4]{BH09} is tailored for random subgraphs of $\Z^d$ with laws invariant under reflections with respect to coordinate axes and 
rotations by $\frac \pi2$. These assumptions only simplify the expression for the heat kernel of the limiting Brownian motion, 
and can be naturally extended to the case without such symmetries. 

\medskip

We only consider the case of discrete time random walk (the continuous time case is the same). 
As in \cite[Theorem~4.5]{BH09}, to prove Theorem~\ref{thm:localclt} it suffices to show that 
there exist an event $\Omega'\in\mathcal F$ with $\mathbb P^u[\Omega']=1$, positive constants $\delta$, $C_i$, and $C_H$, and a covariance matrix $\Sigma$, 
such that for all $\omega\in\Omega'\cap\{0\in\set_\infty\}$, 
\begin{itemize}\itemsep0pt
\item[(a)]
for any $y\in\R^d$ and $r>0$, as $n\to\infty$, $\mathrm{P}_{\set_\infty,0}\left[\widetilde B_n(t)\in (y+[-r,r]^d)\right]$ converges to 
$\int_{y+[-r,r]^d}k_{\Sigma,t}(y') dy'$ uniformly over compact subsets of $(0,\infty)$ 
($\widetilde B_n(t)$ is as in \eqref{def:widetildeBn}),
\item[(b)]
there exists $T_1 = T_1(\omega)<\infty$ such that for all $n\geq T_1$ and $x\in\set_\infty$, 
$p_n(0,x)\leq C_1\cdot n^{-\frac d2}\cdot e^{-C_2\cdot \frac{\dist_{\set_\infty}(0,x)^2}{n}}$,
\item[(c)]
for each $y\in\set_\infty$, there exists $R_H(y) = R_H(y,\omega)<\infty$ such that the parabolic Harnack inequality holds with constant $C_H$ 
in $(0,R^2]\times\ballZ_{\set_\infty}(y,R)$ for all $R\geq R_H(y)$,
\item[(d)]
for $h(r) = \max\{r':\exists y\in[-r,r]^d\text{ such that $\set_\infty\cap(y+[-r',r']^d)=\emptyset$}\}$, the ratio $\frac{h(r)}{r}$ 
tends to $0$ as $r\to\infty$, 
\item[(e)]
for any $x\in\Z^d$ and $r>0$, $\lim_{n\to\infty}\frac{\mu(\set_\infty\cap(\sqrt{n}x + [-\sqrt{n}r,\sqrt{n}r]^d))}{(2\sqrt{n}r)^d} = \mathbb E^u[\mu_0\cdot\mathds{1}_{0\in\set_\infty}]$,
\item[(f)]
for each $x\in\Z^d$ and $r>0$, there exists $T_2(x) = T_2(x,\omega)<\infty$ such that for all $n\geq T_2$, and $x',y'\in\set_\infty\cap(\sqrt{n}x + [-\sqrt{n}r,\sqrt{n}r]^d)$, 
$\dist_{\set_\infty}(x',y')\leq C_3\cdot\max\{\dist_{\Z^d}(x',y'),n^{\frac 12 - \delta}\}$, 
\item[(g)]
for $x\in\Z^d$ and $R_H$ as in (c), $\lim_{n\to\infty}n^{-\frac 12}R_H(g_n(x)) = 0$. 
\end{itemize}
It is easy to see that the above assumptions are satisfied in our setting: 
\begin{itemize}\itemsep0pt
\item[(a)]
follows from Theorem~\ref{thm:qip},
\item[(b)]
follows from \eqref{eq:hk:ub},
\item[(c)]
follows from Theorems~\ref{thm:phi:vgb} and \ref{thm:vgb:main},
\item[(d)]
follows from stationarity, \eqref{eq:C1:infty}, and the Borel-Cantelli lemma,
\item[(e)]
follows from a spatial ergodic theorem \cite[Theorem~2.8 in Chapter~6]{Krengel}, 
since the sequence of boxes $(\sqrt{n}x + [-\sqrt{n}r,\sqrt{n}r]^d)_{n\geq 1}$ is 
regular in the sense of \cite[Definition~2.4 in Chapter~6]{Krengel} (see \cite[Lemma~5.1]{ADS14}), 
\item[(f)]
follows from Theorem~\ref{thm:chd},
\item[(g)]
follows from \eqref{eq:vgb:main:R}, Theorem~\ref{thm:phi:vgb}, and the Borel-Cantelli lemma.
\end{itemize}
The proof of Theorem~\ref{thm:localclt} is complete.
\end{proof}

\medskip

\begin{proof}[Proof of Theorem~\ref{thm:gf:asymp}]
Statement (a) follows from Theorem~\ref{thm:localclt} and \eqref{eq:hk:ub} by repeating the proof of \cite[Theorem~1.2(b)]{BH09} 
without any changes. 
For the statement (c) we use bounds \cite[(6.30) and (6.31)]{BH09} and \eqref{eq:hk:T0}, to get 
\begin{multline*}
\frac{(1-\varepsilon)\mathrm{G}_\Sigma(x)}{m}\mathbb P^u\left[M\leq |x|~\Big|~0\in\set_\infty\right]
\leq \mathbb E^u\left[g_{\set_\infty}(0,x)~\Big|~0\in\set_\infty\right]\\
\leq \frac{(1+\varepsilon)\mathrm{G}_\Sigma(x)}{m}+\frac{C'\mathbb P^u\left[M> |x|~\Big|~0\in\set_\infty\right]}{|x|^{d-2}} \\
+ C'\left(\mathbb E^u\left[g_{\set_\infty}(0,x)^2~\Big|~0\in\set_\infty\right]\right)^{\frac 12}\cdot e^{-c'(\log |x|)^{1+\constS}},
\end{multline*}
where $M$ is defined in the statement of Theorem~\ref{thm:gf:asymp}.
As in \cite[(6.17)]{BH09}, by \eqref{eq:hk:ub}, 
\[
g_{\set_\infty}(0,x)\leq g_{\set_\infty}(0,0)\leq T_0(0) + \int_{T_{0}(0)}^\infty C't^{-\frac d2}dt \leq (1+2C')T_0(0).
\]
Combining this bound with \eqref{eq:hk:T0}, we obtain that $\mathbb E^u\left[g_{\set_\infty}(0,x)^2~\Big|~0\in\set_\infty\right]<C''$. 
Let $x=ky$. Since $\mathrm{G}_\Sigma(ky) = k^{2-d}\mathrm{G}_\Sigma(y)$, by taking limits $k\to\infty$ and then $\varepsilon\to 0$, we compete the proof of statement (c).
\end{proof}

\paragraph{Acknowledgements.}
The author thanks Takashi Kumagai and Alain-Sol Sznitman for encouragements to write this paper and for valuable comments, 
Jean-Dominique Deuschel for sending him the master thesis of Tuan Anh Nguyen \cite{Nguyen} and for a discussion about 
isoperimetric and Sobolev inequalities, and Ji\v r\'\i \, \v Cern\'y for a suggestion to use general densities $\den_1$ and $\den_2$ in Section~\ref{sec:propertiesofclusters}. 
Special thanks go to the anonymous referee for helpful suggestions and comments.

\end{document}

%% file: Rzi.pdf_t
\begin{picture}(0,0)%
\includegraphics{Rzi.pdf}%
\end{picture}%
\setlength{\unitlength}{4144sp}%
\begingroup\makeatletter\ifx\SetFigFont\undefined%
\gdef\SetFigFont#1#2#3#4#5{%
  \reset@font\fontsize{#1}{#2pt}%
  \fontfamily{#3}\fontseries{#4}\fontshape{#5}%
  \selectfont}%
\fi\endgroup%
\begin{picture}(10634,4959)(661,-4976)
\put(6841,-1546){\makebox(0,0)[lb]{\smash{{\SetFigFont{20}{24.0}{\rmdefault}{\mddefault}{\updefault}{\color[rgb]{0,0,0}$a_{z_i}$}%
}}}}
\put(2836,-2986){\makebox(0,0)[lb]{\smash{{\SetFigFont{20}{24.0}{\rmdefault}{\mddefault}{\updefault}{\color[rgb]{0,0,0}$b_{z_i}$}%
}}}}
\put(1081,-1591){\makebox(0,0)[lb]{\smash{{\SetFigFont{20}{24.0}{\rmdefault}{\mddefault}{\updefault}{\color[rgb]{0,0,0}$a_{z_i}$}%
}}}}
\put(676,-4876){\makebox(0,0)[lb]{\smash{{\SetFigFont{20}{24.0}{\rmdefault}{\mddefault}{\updefault}{\color[rgb]{0,0,0}$z_i$}%
}}}}
\put(6481,-4876){\makebox(0,0)[lb]{\smash{{\SetFigFont{20}{24.0}{\rmdefault}{\mddefault}{\updefault}{\color[rgb]{0,0,0}$z_i$}%
}}}}
\put(6931,-2536){\makebox(0,0)[lb]{\smash{{\SetFigFont{20}{24.0}{\rmdefault}{\mddefault}{\updefault}{\color[rgb]{0,0,0}$c_{z_i}$}%
}}}}
\put(7786,-1996){\makebox(0,0)[lb]{\smash{{\SetFigFont{20}{24.0}{\rmdefault}{\mddefault}{\updefault}{\color[rgb]{0,0,0}$b_{z_i}$}%
}}}}
\end{picture}%

%% file: 0good.pdf_t
\begin{picture}(0,0)%
\includegraphics{0good.pdf}%
\end{picture}%
\setlength{\unitlength}{4144sp}%
\begingroup\makeatletter\ifx\SetFigFont\undefined%
\gdef\SetFigFont#1#2#3#4#5{%
  \reset@font\fontsize{#1}{#2pt}%
  \fontfamily{#3}\fontseries{#4}\fontshape{#5}%
  \selectfont}%
\fi\endgroup%
\begin{picture}(10866,10866)(-32,-9994)
\put(3316,-6661){\makebox(0,0)[lb]{\smash{{\SetFigFont{70}{84.0}{\rmdefault}{\mddefault}{\updefault}{\color[rgb]{0,0,0}$x$}%
}}}}
\put(4396,-4456){\makebox(0,0)[lb]{\smash{{\SetFigFont{70}{84.0}{\rmdefault}{\mddefault}{\updefault}{\color[rgb]{0,0,0}$\mathcal C_x$}%
}}}}
\end{picture}%

%% file: select_1.pdf_t
\begin{picture}(0,0)%
\includegraphics{select_1.pdf}%
\end{picture}%
\setlength{\unitlength}{4144sp}%
\begingroup\makeatletter\ifx\SetFigFont\undefined%
\gdef\SetFigFont#1#2#3#4#5{%
  \reset@font\fontsize{#1}{#2pt}%
  \fontfamily{#3}\fontseries{#4}\fontshape{#5}%
  \selectfont}%
\fi\endgroup%
\begin{picture}(9494,7694)(2679,-7283)
\put(3556,-7171){\makebox(0,0)[lb]{\smash{{\SetFigFont{20}{24.0}{\rmdefault}{\mddefault}{\updefault}{\color[rgb]{0,0,0}$R_1$}%
}}}}
\put(10801,-5191){\makebox(0,0)[lb]{\smash{{\SetFigFont{20}{24.0}{\rmdefault}{\mddefault}{\updefault}{\color[rgb]{0,0,0}$R_2$}%
}}}}
\put(6841,-826){\makebox(0,0)[lb]{\smash{{\SetFigFont{20}{24.0}{\rmdefault}{\mddefault}{\updefault}{\color[rgb]{0,0,0}$R_3$}%
}}}}
\end{picture}%

%% file: select_2.pdf_t
\begin{picture}(0,0)%
\includegraphics{select_2.pdf}%
\end{picture}%
\setlength{\unitlength}{4144sp}%
\begingroup\makeatletter\ifx\SetFigFont\undefined%
\gdef\SetFigFont#1#2#3#4#5{%
  \reset@font\fontsize{#1}{#2pt}%
  \fontfamily{#3}\fontseries{#4}\fontshape{#5}%
  \selectfont}%
\fi\endgroup%
\begin{picture}(9494,7694)(2679,-7283)
\put(3556,-7171){\makebox(0,0)[lb]{\smash{{\SetFigFont{20}{24.0}{\rmdefault}{\mddefault}{\updefault}{\color[rgb]{0,0,0}$R_1$}%
}}}}
\put(10801,-5191){\makebox(0,0)[lb]{\smash{{\SetFigFont{20}{24.0}{\rmdefault}{\mddefault}{\updefault}{\color[rgb]{0,0,0}$R_2$}%
}}}}
\put(6841,-826){\makebox(0,0)[lb]{\smash{{\SetFigFont{20}{24.0}{\rmdefault}{\mddefault}{\updefault}{\color[rgb]{0,0,0}$R_3$}%
}}}}
\end{picture}%

%% file: mathcalA_1.pdf_t
\begin{picture}(0,0)%
\includegraphics{mathcalA_1.pdf}%
\end{picture}%
\setlength{\unitlength}{4144sp}%
\begingroup\makeatletter\ifx\SetFigFont\undefined%
\gdef\SetFigFont#1#2#3#4#5{%
  \reset@font\fontsize{#1}{#2pt}%
  \fontfamily{#3}\fontseries{#4}\fontshape{#5}%
  \selectfont}%
\fi\endgroup%
\begin{picture}(9066,9066)(-32,-8194)
\put(901,-2761){\makebox(0,0)[lb]{\smash{{\SetFigFont{20}{24.0}{\rmdefault}{\mddefault}{\updefault}{\color[rgb]{0,0,0}$B$}%
}}}}
\put(3601,-3211){\makebox(0,0)[lb]{\smash{{\SetFigFont{20}{24.0}{\rmdefault}{\mddefault}{\updefault}{\color[rgb]{0,0,0}$\mathcal A$}%
}}}}
\put(4816,-4291){\makebox(0,0)[lb]{\smash{{\SetFigFont{20}{24.0}{\rmdefault}{\mddefault}{\updefault}{\color[rgb]{0,0,0}$B_1$}%
}}}}
\put(6076,-1996){\makebox(0,0)[lb]{\smash{{\SetFigFont{20}{24.0}{\rmdefault}{\mddefault}{\updefault}{\color[rgb]{0,0,0}$B_2$}%
}}}}
\end{picture}%

%% file: mathcalA_2.pdf_t
\begin{picture}(0,0)%
\includegraphics{mathcalA_2.pdf}%
\end{picture}%
\setlength{\unitlength}{4144sp}%
\begingroup\makeatletter\ifx\SetFigFont\undefined%
\gdef\SetFigFont#1#2#3#4#5{%
  \reset@font\fontsize{#1}{#2pt}%
  \fontfamily{#3}\fontseries{#4}\fontshape{#5}%
  \selectfont}%
\fi\endgroup%
\begin{picture}(9066,9066)(-32,-8194)
\put(901,-2761){\makebox(0,0)[lb]{\smash{{\SetFigFont{20}{24.0}{\rmdefault}{\mddefault}{\updefault}{\color[rgb]{0,0,0}$B$}%
}}}}
\put(3601,-3211){\makebox(0,0)[lb]{\smash{{\SetFigFont{20}{24.0}{\rmdefault}{\mddefault}{\updefault}{\color[rgb]{0,0,0}$\mathcal A'$}%
}}}}
\end{picture}%

%% file: deltai_1.pdf_t
\begin{picture}(0,0)%
\includegraphics{deltai_1.pdf}%
\end{picture}%
\setlength{\unitlength}{4144sp}%
\begingroup\makeatletter\ifx\SetFigFont\undefined%
\gdef\SetFigFont#1#2#3#4#5{%
  \reset@font\fontsize{#1}{#2pt}%
  \fontfamily{#3}\fontseries{#4}\fontshape{#5}%
  \selectfont}%
\fi\endgroup%
\begin{picture}(9066,9066)(-32,-8194)
\put(3806,-3394){\makebox(0,0)[lb]{\smash{{\SetFigFont{20}{24.0}{\rmdefault}{\mddefault}{\updefault}{\color[rgb]{0,0,0}$\mathcal A'$}%
}}}}
\end{picture}%

%% file: smallA_1.pdf_t
\begin{picture}(0,0)%
\includegraphics{smallA_1.pdf}%
\end{picture}%
\setlength{\unitlength}{4144sp}%
\begingroup\makeatletter\ifx\SetFigFont\undefined%
\gdef\SetFigFont#1#2#3#4#5{%
  \reset@font\fontsize{#1}{#2pt}%
  \fontfamily{#3}\fontseries{#4}\fontshape{#5}%
  \selectfont}%
\fi\endgroup%
\begin{picture}(9066,9066)(-32,-8194)
\put(3601,-3211){\makebox(0,0)[lb]{\smash{{\SetFigFont{20}{24.0}{\rmdefault}{\mddefault}{\updefault}{\color[rgb]{0,0,0}$\mathcal A'$}%
}}}}
\end{picture}%

%% file: smallA_2.pdf_t
\begin{picture}(0,0)%
\includegraphics{smallA_2.pdf}%
\end{picture}%
\setlength{\unitlength}{4144sp}%
\begingroup\makeatletter\ifx\SetFigFont\undefined%
\gdef\SetFigFont#1#2#3#4#5{%
  \reset@font\fontsize{#1}{#2pt}%
  \fontfamily{#3}\fontseries{#4}\fontshape{#5}%
  \selectfont}%
\fi\endgroup%
\begin{picture}(9066,9066)(-32,-8194)
\put(3601,-3211){\makebox(0,0)[lb]{\smash{{\SetFigFont{20}{24.0}{\rmdefault}{\mddefault}{\updefault}{\color[rgb]{0,0,0}$\mathcal A'$}%
}}}}
\put( 71,-5779){\makebox(0,0)[lb]{\smash{{\SetFigFont{20}{24.0}{\rmdefault}{\mddefault}{\updefault}{\color[rgb]{0,0,0}$\widetilde z$}%
}}}}
\end{picture}%

%% file: slice.pdf_t
\begin{picture}(0,0)%
\includegraphics{slice.pdf}%
\end{picture}%
\setlength{\unitlength}{4144sp}%
\begingroup\makeatletter\ifx\SetFigFont\undefined%
\gdef\SetFigFont#1#2#3#4#5{%
  \reset@font\fontsize{#1}{#2pt}%
  \fontfamily{#3}\fontseries{#4}\fontshape{#5}%
  \selectfont}%
\fi\endgroup%
\begin{picture}(21655,10844)(-12621,-9533)
\end{picture}%

%% file: HK_20151123.bbl
\begin{thebibliography}{99}

\bibitem{ABDH13}
S. Andres, M. T. Barlow, J.-D. Deuschel, and B. M. Hambly (2013) Invariance
principle for the random conductance model. {\it Pr. Theory Rel. Fields}
{\bf 156(3-4)}, 535--580.

\bibitem{ADS13}
S. Andres, J.-D. Deuschel, and M. Slowik (2013) 
Invariance principle for the random conductance model in a degenerate ergodic environment. 
{\it To appear in Ann. Probab.} arXiv:1306.2521.

\bibitem{ADS14}
S. Andres, J.-D. Deuschel, and M. Slowik (2014) 
Harnack inequalities on weighted graphs and some applications to the random conductance model. 
arXiv:1312.5473v2.

\bibitem{Barlow}
M. T. Barlow (2004) Random walks on supercritical percolation clusters. 
{\it Ann. Probab.} {\bf 32}, 3024--3084.

\bibitem{BD10}
M. T. Barlow and J.-D. Deuschel (2010) Invariance principle for the random conductance
model with unbounded conductances. {\it Ann. Probab.} {\bf 38(1)}, 234--276.

\bibitem{BH09}
M. T. Barlow and B. M. Hambly (2009) Parabolic Harnack inequality and local limit theorem for percolation clusters. 
{\it Electron. J. Probab.}, {\bf 14}, paper 1.

\bibitem{BarlowChen14}
M. T. Barlow and X. Chen (2014) Gaussian bounds and parabolic Harnack inequality on locally irregular graphs. 
Preprint.

\bibitem{BDCKY14}
I. Benjamini, H. Duminil-Copin, G. Kozma, and A. Yadin (2014) 
Disorder, entropy and harmonic functions. 
{\it To appear in Ann. Probab.} arXiv:1111.4853.

\bibitem{BergerBiskup}
N. Berger and M. Biskup (2007) 
Quenched invariance principle for simple random walk on percolation cluster. 
{\it Probab. Theory Rel. Fields} {\bf 137}, 83--120.

\bibitem{BBHK}
N. Berger, M. Biskup, C. Hoffman and G. Kozma (2008)
Anomalous heat-kernel decay for random walk on among bounded random conductances. 
{\it Ann. Inst. H. Poincar\'e Probab. Statist.} {\bf 44(2)}, 374--392.

\bibitem{BiskupPrescott}
M. Biskup and T. Prescott (2007)
Functional CLT for random walk among bounded random conductances. 
{\it Electron. J. Probab.}, {\bf 12}, paper 49, 1323--1348.

\bibitem{BLM}
J. Bricmont, J. L. Lebowitz and C. Maes (1987) 
Percolation in strongly correlated systems: the massless Gaussian field. 
{\it J. Stat. Phys.} {\bf 48} (5/6), 1249--1268.

\bibitem{CH08}
D. A. Croydon and B. M. Hambly (2008) Local limit theorems for sequences of simple random walks on graphs. 
{\it Potential Anal.} {\bf 29} (4), 351--389.

\bibitem{DeGiorgi}
E. De Giorgi (1957) Sulla differenziabilit\`a e l'analiticit\`a del le estremali degli integrali multipli regolari.
{\it Mem. Accad. Sci. Torino. Cl. Sci. Fis. Mat. Nat.} {\bf 3}, 25--43.

\bibitem{Delmotte98}
T. Delmotte (1998) Harnack inequalities on graphs. {\it S\'eminaire de th\'eorie spectrale et g\'eom\'etrie. Grenoble.}
{\bf 16}, 217--228.

\bibitem{Delmotte99}
T. Delmotte (1999) Parabolic Harnack inequality and estimates of Markov chains on graphs. 
{\it Rev. Mat. Iberoamericana} {\bf 15} (1), 181--232.

\bibitem{DeuschelPisztora}
J. D. Deuschel and A. Pisztora (1996)
Surface order large deviations for high-density percolation.
{\it Probab. Theory Related Fields} {\bf 104}, 467--482.

\bibitem{DRS}
A. Drewitz, B. R\'ath and A. Sapozhnikov (2012) 
Local percolative properties of the vacant set of random interlacements with small intensity. 
{\it To appear in the Ann. Inst. Henri Poincar\'e}. arXiv:1206.6635.

\bibitem{DRS12}
A. Drewitz, B. R\'ath and A. Sapozhnikov (2012) 
On chemical distances and shape theorems in percolation models with long-range correlations. 
{\it  J. Math. Phys.} {\bf 55}, 083307.

\bibitem{Grigoryan}
A. A. Grigoryan (1992) Heat equation on a noncompact Riemannian manifold.
{\it Math. USSR-Sb.} {\bf 72}, 47--77.

\bibitem{Grimmett}
G. Grimmett (1999) {\it Percolation.} Second edition. Springer-Verlag, Berlin. {\bf 321}.

\bibitem{Krengel}
U. Krengel (1985) {\it Ergodic theorems.} Walter de Gruyter and Co., Berlin.

\bibitem{Kumagai}
T. Kumagai (2014) {\it Random walks on disordered media and their scaling limits.}
Ecole d'\'Et\'e de Probabilit\'es de Saint-Flour XL-2010. Lecture notes in mathematics, {\bf 2101}.

\bibitem{MathieuRemy}
P. Mathieu and E. Remy (2004) 
Isoperimetry and heat kernel decay on percolations clusters. 
{\it Ann. Probab.} {\bf 32}, 100--128.
 
\bibitem{MathieuPiatnitski}
P. Mathieu and A. L. Piatnitski (2007) 
Quenched invariance principles for random walks on percolation clusters. 
{\it Proceedings of the Royal Society A} {\bf 463}, 2287--2307. 

\bibitem{MorrisPeres}
B. Morris and Y. Peres (2005) Evolving sets, mixing and heat kernel bounds. 
{\it Probab. Theory Rel. Fields} {\bf 133} (2), 245--266.

\bibitem{Moser61}
J. Moser (1961) On Harnack's theorem for elliptic differential equations. 
{\it Commun. Pure Appl. Math.} {\bf 14} (3), 577--591.

\bibitem{Moser64}
J. Moser (1964) A Harnack inequality for parabolic differential equations. 
{\it Commun. Pure Appl. Math.} {\bf 17} (1), 101--134.

\bibitem{Nash}
J. Nash (1958) Continuity of solutions of parabolic and elliptic equations. 
{\it Amer. J. Math.} {\bf 80}, 931--954.

\bibitem{Nguyen}
T. A. Nguyen (2014) Quenched invariance principle for random conductance model. Master thesis, Berlin. 

\bibitem{PRS}
E. Procaccia, R. Rosenthal, and A. Sapozhnikov (2013) 
Quenched invariance principle for simple random walk on clusters in correlated percolation models. 
arXiv:1310.4764v2.

\bibitem{RS:Transience}
B. R\'ath and A. Sapozhnikov (2011) On the transience of random interlacements. 
{\it Electron. Commun. Probab.} {\bf 16}, 379--391.

\bibitem{RS:Disordered} 
B. R\'ath and A. Sapozhnikov (2011) The effect of small quenched noise on connectivity properties of random interlacements. 
{\it Electron. J. of Prob.} {\bf 18} (4), 1--20.

\bibitem{RodSz}
P. F. Rodriguez and A.-S. Sznitman (2012)
Phase transition and level-set percolation for the Gaussian free field. 
{\it Comm. Math. Phys.} {\bf 320(2)}, 571--601.
 
\bibitem{SC92}
L. Saloff-Coste (1992) A note on Poincar\'e, Sobolev, and Harnack inequalities. 
{\it Internat. Math. Res. Notices} {\bf 2}, 27--38.

\bibitem{SC97}
L. Saloff-Coste (1997) {\it Lectures on finite Markov chains.} 
Ecole d'\'Et\'e de Probabilit\'es de Saint-Flour XXVI-1996. Lecture notes in mathematics, vol. 1665, Springer, Berlin, 
301--413.

\bibitem{SS04}
V. Sidoravicius and A.-S. Sznitman (2004) 
Quenched invariance principles for walks on clusters of percolation or among random conductances. 
{\it Prob. Theory Rel. Fields} {\bf 129}, 219-244.
 
\bibitem{SidoraviciusSznitman_RI}
V. Sidoravicius and A.-S. Sznitman (2009) Percolation for the Vacant Set of Random Interlacements.
{\it Comm. Pure Appl. Math.} {\bf 62} (6), 831--858.
 
\bibitem{SznitmanAM}
A.-S. Sznitman (2010) Vacant set of random interlacements and percolation. 
{\it Ann. Math.} {\bf 171} (2), 2039--2087.
 
\bibitem{Sznitman:Decoupling}
A.-S. Sznitman (2012) 
Decoupling inequalities and interlacement percolation on $G\times \Z$. 
{\it Invent. Math.} {\bf 187} (3), 645--706.

\bibitem{Teixeira09}
A. Teixeira (2009) On the uniqueness of the infinite cluster of the vacant set of random interlacements.
{\it Ann. Appl. Probab.} {\bf 19} (1), 454--466.

\end{thebibliography}
